\documentclass[12pt,reqno]{amsart}
\UseRawInputEncoding
\usepackage[active]{srcltx}
\newcommand{\h}{\hat}

\usepackage[cp1252]{inputenc}
\usepackage{amsthm}
\usepackage{amsmath}
\usepackage{amsfonts}
\usepackage{amssymb}
\usepackage{graphicx}

\newcommand{\bex}{\begin{example}}
\newcommand{\eex}{\end{example}}
\newcommand{\be}{\begin{equation} }
\newcommand{\ee}{\end{equation} }
\newcommand{\bcs}{\begin{cases}}
\newcommand{\ecs}{\end{cases}}

\newcommand{\ip}[2]{\langle #1, #2 \rangle}

\newcommand{\gtc}[1]{g^{#1}_{\C}}

\usepackage{amsfonts}
\usepackage{amsthm}
\usepackage{amsmath}
\usepackage{amscd}
\usepackage{calc}

\newcommand{\brem}{\begin{rem}}
\newcommand{\erem}{\end{rem}}

\usepackage{mathtools}
\mathtoolsset{showonlyrefs}

\newcommand{\pa}{\partial}

\newcommand{\baa}{\begin{align*}}
\newcommand{\eaa}{\end{align*}}
\newcommand{\bea}{\begin{eqnarray*} }
\newcommand{\eea}{\end{eqnarray*} }
\newcommand{\bee}{\begin{eqnarray} }
	\newcommand{\eee}{\end{eqnarray} }
\newcommand{\beq}{\begin{equation} }
\newcommand{\eeq}{\end{equation} }
\newcommand{\bpp}{\begin{prop}}
\newcommand{\epp}{\end{prop}}
\newcommand{\bt}{\begin{theorem}}
\newcommand{\et}{\end{theorem}}
\newcommand{\bpf}{\begin{proof}}
\newcommand{\epf}{\end{proof}}
\newcommand{\bl}{\begin{lem}}
\newcommand{\el}{\end{lem}}
\newcommand{\bc}{\begin{cor}}
\newcommand{\ec}{\end{cor}}
\newcommand{\bd}{\begin{defin}}
\newcommand{\ed}{\end{defin}}

\newcommand{\edit}[1]{{\color{red}{$\clubsuit$#1$\clubsuit$}}}

\newcommand{\bma}{\begin{bmatrix}}
\newcommand{\ema}{\end{bmatrix}}

\usepackage{color}

\usepackage{hyperref}
\def\name{Z-Z}
\hypersetup{
  colorlinks = true,
  urlcolor = black,
  pdfauthor = {\name},
  pdfkeywords = {mathematics},
  pdftitle = {\name: },
  pdfsubject = {Mathematics},
  pdfpagemode = UseNone
}

\usepackage{wrapfig}
\usepackage{pstricks}

\usepackage[arrow, matrix, curve]{xy}

\usepackage[active]{srcltx}

\usepackage{latexsym}
\usepackage{graphicx}

\renewcommand{\Re}{{\operatorname{Re}\,}}
\renewcommand{\Im}{{\operatorname{Im}\,}}

\newcommand{\szego}{Szeg\"o}
\usepackage{latexsym}
\usepackage{graphicx}

\renewcommand{\Re}{{\operatorname{Re}\,}}
\renewcommand{\Im}{{\operatorname{Im}\,}}

\renewcommand{\H}{{\mathbf H}}

\renewcommand{\epsilon}{\varepsilon}

\newcommand{\kahler}{K\"ahler }

\newcommand{\wt}{\widetilde}

\newcommand{\N}{{\mathbb N}}
\newcommand{\R}{{\mathbb R}}
\newcommand{\C}{{\mathbb C}}

\newcommand{\Z}{{\mathbb Z}}

\newcommand{\Ss}{{\mathbb S}}

\newcommand{\dbar}{\bar\partial}
\newcommand{\ddbar}{\partial\dbar}

\newcommand{\half}{{\textstyle \frac 12}}

\renewcommand{\phi}{\varphi}

\newcommand{\acal}{\mathcal{A}}

\newcommand{\dcal}{\mathcal{D}}
\newcommand{\ecal}{\mathcal{E}}
\newcommand{\fcal}{\mathcal{F}}
\newcommand{\gcal}{\mathcal{G}}
\newcommand{\hcal}{\mathcal{H}}

\newcommand{\jcal}{\mathcal{J}}
\newcommand{\kcal}{\mathcal{K}}
\newcommand{\lcal}{\mathcal{L}}
\newcommand{\mcal}{\mathcal{M}}

\newcommand{\ocal}{\mathcal{O}}
\newcommand{\pcal}{\mathcal{P}}
\newcommand{\qcal}{\mathcal{Q}}

\newcommand{\scal}{\mathcal{S}}
\newcommand{\tcal}{\mathcal{T}}

\newcommand{\wcal}{\mathcal{W}}

\newtheorem{theo}{{\sc Theorem}}[section]

\newtheorem{defin}{{\sc Definition}}

\newtheorem{cor}[theo]{{\sc Corollary}}

\newtheorem{lem}[theo]{{\sc Lemma}}
\newtheorem{rem}[theo]{{\sc Remark}}
\newtheorem{prop}[theo]{{\sc Proposition}}

\newenvironment{example}{\medskip\noindent{\it Example:\/} }{\medskip}

\newtheorem{defn}[theo]{{\sc Definition}}

\title[$L^{\infty} $ norms of Husimi distributions of eigenfunctions ]
{$L^{\infty}$ norms of Husimi distributions of eigenfunctions}

\author{Steve Zelditch}
\address{Department of Mathematics, Northwestern  University, Evanston, IL 60208, USA}

\email{zelditch@math.northwestern.edu}

\thanks{Research partially supported by NSF grant  DMS-1810747}

\date{\today}

\begin{document}

\begin{abstract} We give  two term   pointwise Weyl laws for analytic continuations
of  eigenfunctions $ \phi_j^{\C}(\zeta)$ 
 of a real analytic Riemannian manifold $(M, g)$ without boundary to a Grauert tube $M_{\tau}$.
The Weyl laws are asymptotic formulae for Weyl sums $\sum_{j: \lambda_j \leq \lambda} e^{- 2 \tau \lambda_j}
|\phi_j^{\C}(\zeta)|^2$ with $\zeta \in \partial M_{\tau}$. The summands, when $L^2$ normalized,
are  special types of `microlocal lifts' or Husimi distributions $ \frac{|\phi_{\lambda}^{\C}(\zeta)|^2}{|| \phi_{\lambda}||^2_{L^2(\partial M_{\tau})}}$, whose weak limits are the microlocal defect measures studied in quantum chaos.  Rather
than weak* limits, we study the asymptotics of their sup norms. 
The asymptotics depend on whether or not  $\zeta$ is a periodic point  of the geodesic flow, and on whether
the periodic orbit is of  elliptic  type or not. The two-term  Weyl law is analogous to the two-term Weyl asymptotics of Y. Safarov in 
the real domain. 
The remainder estimate gives universal  growth bounds on $|\phi_j^{\C}(\zeta)|^2$ for $\zeta \in \partial M_{\tau}$, which are shown to be
sharp (they are
attained by analytic continuations of Gaussian beams). 

\end{abstract}

\maketitle

\tableofcontents

This article is concerned with analytic continuations $ \phi_j^{\C}(\zeta)$   of  eigenfunctions of the Laplacian $\Delta$
 of a real analytic Riemannian manifold $(M, g)$ of dimension $m$ without boundary to a Grauert tube $M_{\tau} \subset
 M_{\C}$ in the complexification of $M$; see \S \ref{GRB}
for definitions and background.  In the `real domain'   $M$, we denote
by  \begin{equation}\label{EP}  \Delta_g\; \phi_{j} =  \lambda_j^2\; \phi_{j}, \;\;\; \langle \phi_j, \phi_k \rangle = \delta_{j k} \end{equation}
an orthonormal basis of eigenfunctions with $\lambda_0 = 0 \leq \lambda_1 \leq \lambda_2 \leq \cdots. $ \footnote{We use the convention where  Laplacian $\Delta$ is
minus the usual sum of squares, hence is a positive operator,  because we will take many square roots}The
classical  pointwise Weyl law of Avakumovic, Levitan and H\"ormander  in the real domain asserts that,
\begin{equation} \label{LWL} N(\lambda, x) := \sum_{j: \lambda_j \leq \lambda} |\phi_{j}(x)|^2 = C_m \lambda^m + R(\lambda, x).  \end{equation}    where $C_m$ is a dimensional constant and where the remainder satisfies, 
\begin{equation} \label{R} R(\lambda, x) = O(\lambda^{m-1}), \;\; \mbox{uniformly in x}. \end{equation}  An important application of the pointwise
Weyl law is to bound  sup-norms  of eigenfunctions, since  $\phi^2_{\lambda_j}(x)$ is bounded above by   the  jump in the remainder at an eigenvalue, i.e.
there exists a constant $C_g > 0$ depending only on the metric $g$ so that, \begin{equation}\label{REALSUP} ||\phi_{j}||_{L^{\infty}} \leq \sup_x |R(\lambda_j + 0 , x) - R(\lambda_j -  0, x)| \leq C_g \lambda_j^{\frac{m-1}{2}}. 
\end{equation}

%\begin{rem} 
%The sup norm bounds hold for any $\Delta$-eigenfunction 
%satisfying
%\begin{equation} \label{DELTA} \Delta \phi_{\lambda} = \lambda^2 \phi_{\lambda}, \;\;\;\; ||\phi_{\lambda}||_{L^2(M)} = 1, 
%\end{equation}  and so we often state results in terms of a general eigenfunction.\end{rem}

In  \cite{Saf, SV, SoZ,SoZ16}  it is shown that the size of the remainder and the
sup norm bounds depend on the structure of the set $\lcal_x$ of geodesic loops based at $x$. For a real analytic surface, it is shown in \cite{SoZ16} that the 
sup norm bound \eqref{REALSUP} is only achieved if there exists a point $p \in M$ so that all geodesics through p are closed.  The author  has conjectured
that, in all dimensions, the sup norm bound is achieved only when there exists a point $p \in M$ so that all geodesics through p are closed. 

    The purpose of this article is to formulate and prove a phase version of \eqref{LWL} and \eqref{REALSUP} in terms of analytic continuations $\phi_j^{\C}$  of the
    eigenfunctions $\phi_j$ to Grauert tubes $M_{\tau}$  in the complexifcation $M_{\C}$ of $M$.
As reviewed in  \S \ref{GRB}, the metric $g$ induces a Grauert tube radius function $\sqrt{\rho}$, essentially (half) the distance between
$\zeta$ and $\bar{\zeta}$ in a natural metric on $M_{\tau}$. 
For $\tau$ sufficiently small, $\partial M_{\tau} = \{\zeta \in M_{\C}: \sqrt{\rho}(\zeta) = \tau\}$ is 
equivalent (under the complexified exponential map \eqref{EXP}) to the cosphere bundle $S^*_{\tau} M$ of radius $\tau$.
The  purpose of this article is to prove  a  pointwise `phase-space' Weyl law \eqref{LWL}
and  sharp universal  sup norm bounds  \eqref{REALSUP}  for
 the Husimi distributions,  \begin{equation} \label{HUSIMI}  U_j^{\tau}(\zeta): = 
\frac{|\phi_{j}^{\C}(\zeta)|^2}{|| \phi_j||^2_{L^2(\partial M_{\tau})}} \end{equation}  
Husimi distributions are  special constructions 
of `Wigner distributions' or `microlocal lifts' of eigenfunctions 
(see \S \ref{RELATED} for  other microlocal lifts); their weak* limits are the well-known microlocal defect measures or quantum limits studied
in quantum chaos.  They are probability distributions on $\partial M_{\tau} \simeq S^*_{\tau} M$, and   are viewed as giving
the probability density of finding a quantum particle at the phase space point $\zeta
\in \partial M_{\tau}$.  What makes the construction in terms of analytic continuations of eigenfunctions attractive is that these microlocal lifts
are postive, relatively concrete and can be studied using complex analytic methods (see \cite{ZJDG} and its references  for applications to nodal sets).

The motivation to study sup-norms of Husimi functions \eqref{HUSIMI} is similar to that
in the physical space $M$: namely, to determine  the maximal degree of concentration at a point  in phase space of  the  probability density of the quantum particle. 
The main results of this article  give  universal upper bounds on the sup-norms of \eqref{HUSIMI} and show that, when a sequence of  Husimi distribution attains the maximal sup norm bounds, there must exist an elliptic closed geodesic along which it attains the bounds see Theorem
\ref{ATTAINED}). This is a novel
kind of `scarring'.      A natural question under investigation  is  the relation of this kind of scarring  to that of weak* limits of the Husimi distributions.

The first result gives the universal upper bound (the lower bound is not important for this article).

\begin{theo} \label{PWintro} Suppose  $(M, g)$ is real analytic, with $\dim M = m$,  and 
%and that $\phi_{\lambda}$
%satisfies \eqref{DELTA}. 
let $\zeta \in \partial M_{\tau}$.  Then,
for any $C > 0$, 
there exists $\mu, A, a > 0$ (independent of $\lambda$)  so that, for $\frac{C}{\lambda} \leq \sqrt{\rho}(\zeta) \leq \tau, $

\begin{equation} \label{GOODSUP}  a
\lambda_j^{-\mu} e^{ \tau \lambda_j} \leq \sup_{\zeta \in M_{\tau}}
|\phi^{\C}_{j}(\zeta)| \leq A
  \lambda_j^{\frac{m-1}{4}} e^{\tau \lambda_j}.
\end{equation}\bigskip

Moreover, the square roots of the Husimi distributions \eqref{HUSIMI} satisfy the bounds,  \begin{equation} \label{BADSUP}  a
\lambda_j^{-\mu} \lambda_j^{\frac{m-1}{4}} \leq \sup_{\zeta \in M_{\tau}}
\frac{|\phi_{j}^{\C}(\zeta)|}{|| \phi_{j}||_{L^2(\partial M_{\tau})}}  \leq A
  \lambda_j^{\frac{m-1}{2}}
\end{equation}\bigskip

\end{theo}
The upper bound of \eqref{GOODSUP}  is sharp and is attained by complexified  highest weight spherical harmonics on $\Ss^m$ and by general Gaussian beams (Section \ref{GBSECT}).  
The bounds of Theorem \ref{PWintro}  substantially improve the estimates 
\begin{equation} 
%\label{SUPESTEIG} 
\sup_{\zeta \in M_{\tau}} |\phi^{\C}_{j}(\zeta)|
\leq C_{\tau} \lambda_j^{m + 1} e^{\tau  \lambda_j}.
\end{equation}
in \cite{Bou,GLS} and their improvements in   
 \cite[Corollary 3]{ZPSH1} (see also \cite{L18}).

The upper  bounds on the Husimi distributions follow from the bounds on $|\phi^{\C}_{j}(\zeta)| $ and from the following $L^2$ norm asymptotics: 
\begin{lem} \label{L2LEMintro} Under the asumptions of Theorem \ref{PWintro}, then there exists a universal postive constant $C(m, \tau)> 0$ so that,    $$||\phi_j^{\C}||^2_{L^2(\partial M_{\tau})}  =
C(m, \tau)  e^{2 \tau \lambda_j}  \lambda_j^{-\frac{m-1}{2}} (1 + O(\lambda_j^{-1})). $$
 \end{lem}
 \noindent See \eqref{L2TORUS} for the explicit calculation for plane waves $e_k$ on  a flat torus. 
 %In that case, the sup norm of the $L^2$ normalized
 %analytic continuation, i.e. the square root of \eqref{HUSIMI} is asymptotic to $ |\lambda|^{\frac{m-1}{4}} $.
%In \cite{ZWCP} it is shown that $Q_{\zeta}(\lambda)$ is uniformly continuous as long 
%as $(M, g)$ is a real analytic manifold without conjugate points. 

The upper  bound \eqref{BADSUP} on the Husimi distribution is also  sharp and is  also obtained by Gaussian beams. In the case of the standard basis $Y_{\ell}^m$ on  standard spheres, it is straightforward to relate the sup norms of  $e^{-2 \tau \lambda_j}
|\phi_{j}^{\C}(\zeta)|^2$ and of the Husimi distributions of $Y_{\ell}^m$, 
and to show that  Husimi distributions of highest weight spherical harmonics of $\Ss^m$  attain the upper bound \eqref{BADSUP}
 (see Section \ref{SPHERESECT}).
It is also attained by complex coherent states (see Definition \ref{CSTDEF} for their
definition), but they are not quite of the form \eqref{HUSIMI}.

As in the real domain, a motivating problem  is  to characterize the Riemannian metrics $g$ possessing sequences
of eigenfunctions whose Husimi distributions attain the maximal sup-norm bounds, then to characterize the points $\zeta$ at which the sup norm
is attained, and  to characterize the associated sequence  $\{\phi_{j_k}\} $ of eigenfunctions.  One may conjecture that the sup norm bound is only
attained by complexified Gaussian beams. Theorem \ref{ATTAINED} at least shows that the elliptic geodesic must exist on such manifolds. 
 
  Of course, the $L^{\infty}$ norms of the Husimi distributions \eqref{HUSIMI} are just one type of norm to study. The most significant norms
  to study in  phase space $M_{\tau}$ (i.e. $B^*_{\tau} M$) are not necessarily the same ones as in configuration space $M$.  Most studies
  of Husimi distributions concern the weak * limits of the sequence \eqref{HUSIMI}, which are invariant probability measures under the geodesic
  flow. The  relation between $L^{\infty}$ norms along subsequences and its weak * limits is currently under exploration. Another interesting
  norm is the microlocal Kakeya-Nikodym norm, 
  $$||\phi_j||^2_{MKN}: = \sup_{\gamma \in \Pi} \int_{\tcal_{\lambda_j^{-\half}}(\gamma)} U_j^{\tau} dV_{\tau}, $$
  where $\tcal_{\lambda_j^{-\half}}(\gamma)$ is the tube in  $\partial M_{\tau}$ around the phase space geodesic arc $\gamma$, i.e.
  the orbit of a point $\zeta \in \partial M_{\tau}$ under the geodesic flow (transported to $\partial M_{\tau}$).
  As Theorem \ref{ATTAINED} shows, maximal
 sup norm growth only occurs at points along elliptic closed geodesics, and such Husimi distributions also seem to  saturate the microlocal Kakeya-Nikodym
 norm.   A related  microlocal Kakeya-Nikodym norm was also defined and studied in \cite[page 515]{BlS17}.  The definitions are apriori quite different;
the  exact relations between the two microlocal Kakeya-Nikodym norms   is also  under investigation. 
%It is proved in
% \cite[Lemma 3.2]{BlS17} that $$f_h ||_{MKN} \leq ||\psi_h||_{KN} $ where $\psi_h$ is a quasi-mode (cf. \cite[(1.5)]{BlS17}),  where 
% $ ||\psi_h||_{KN} $ is its Kakeya-Nikodym norm in configuration space $M$ and where $f_h$ is a second quasimode related to $\psi_h$ (cf. \cite[(2.5)]%%{BlS17}/
  
\subsection{\label{STATEMENT} Background and precise statement of results}

To state the next results,   we need to introduce some further notation and background regarding Grauert tubes and
 their relation to $T^*M$. More systematic expositions of the relevant background can be found in \cite{GS1, LS1, GLS, ZPSH1, L18}. 
 % (Theorem \ref{SHORTINTSa}). 

 A  real analytic Riemannian manifold $M$ always possesses a complexification $M_{\C}$ into which it embeds 
  as a totally real submanifold. 
 A
real analytic metric $g$ induces a unique plurisubharmonic
exhaustion function $\sqrt{\rho}$ known as the Grauert tube
function. It is related to the square $r^2(x, y)$ of the
Riemannian distance function on $M \times M$  by
\begin{equation} \label{RHOFORM} \sqrt{\rho}(\zeta) = \frac{1}{2 i} \sqrt{r^2_{\C}(\zeta,
\bar{\zeta})} \end{equation} where $r^2_{\C}$ is the holomorphic
extension of $r^2(x, y)$  to a small neighborhood of the
anti-diagonal $(\zeta, \bar{\zeta})$ in $M_{\C} \times M_{\C}$.
The open Grauert tube of radius $\tau$ is defined by  \begin{equation} \label{MTAU} M_{\tau} =
\{\zeta \in M_{\C}, \sqrt{\rho}(\zeta) < \tau\}.  \end{equation}
There is a maximal radius $\tau_{\max} \in (0, \infty]$ such that $M_{\tau}$ is an embedded tube for $\tau < \tau_{\max}$, and all 
eigenfunctions $\phi_{\lambda}$ admit analytic continuations to $M_{\tau}$ and are smooth up to the boundary. 
For further background on Grauert tubes, we refer to Section \ref{GRAUERTSECT}.

   As in the real domain (see \eqref{REALSUP}),   sup norms of normalized complexified
eigenfunctions are  bounded by   the associated  jump of the remainder term
in the pointwise Weyl law at the eigenvalue. In the complex domain, there are several choices of the relevant Weyl law. One way to study
the average growth of modulus squares of analytic continuations of eigenfunctions is to is complexify the  spectral function, 
\begin{equation}\label{CXSPMa}   \Pi_{I_{\lambda}}^{\C}(\zeta, \bar{\zeta}):=
\sum_{j: \lambda_j \in  I_{\lambda}}
|\phi_{j}^{\C}(\zeta)|^2,
\end{equation}
of   $\Delta$ for an interval $I_{\lambda} = [a(\lambda), b(\lambda)]$, and restrict it to  the totally real diagonal 
of $M_{\C} \times \overline{M_{\C}}$. 
In   \cite{ZPSH1} it is proved  that the kernels \eqref{CXSPMa}
grow exponentially at the rate $e^{2 \lambda_j\sqrt{\rho}(\zeta)}$.
% where $\sqrt{\rho}$ is the Grauert tube function of
%$g$.
 To obtain polynomial growth, we introduce the 
`tempered'  spectral
projections
\begin{equation}\label{TCXSPM}   P_{[0, \lambda]}^{\tau}(\zeta, \bar{\zeta}) =
\sum_{j: \lambda_j \leq \lambda} e^{-2 \tau \lambda_j}
|\phi_{j}^{\C}(\zeta)|^2, \;\; (\sqrt{\rho}(\zeta) \leq
\tau).
\end{equation}
More generally, as in \eqref{CXSPMa}, we could study $P_{ I_{\lambda}}^{\tau}(\zeta, \bar{\zeta}) $, 
where $I_{\lambda} $ could be a short interval  $[\lambda, \lambda
+ 1]$ of frequencies or a long window $[0, \lambda]$. We let $I_{\lambda} = [\lambda -1, \lambda +1]$ when $(M,g)$ is a Zoll
manifold, with the intervals centered so that $I_{\lambda}$ contains exactly one full cluster. But for the the main results, we only consider $I_{\lambda} = [0, \lambda]$ and focus
on the special case \eqref{TCXSPM}. 
%The kernels (\ref{CXSPMa}) and
%(\ref{TCXSPM}) differ by the metric factor $e^{2 \lambda_j
%\sqrt{\rho}}$.
  The tempered
kernels $P_{ I_{\lambda}}^{\tau}(\zeta, \bar{\zeta})$ are in some ways
analogous to  the semi-classical  Szeg\"o kernels  $\Pi_{h^k}(x, y)$ of positive line bundles over
\kahler manifolds (see Section \ref{COMPARISON}).  Henceforth, we generally assume that $\sqrt{\rho}(\zeta) = \tau$ when
we study \eqref{TCXSPM}, since the sums \eqref{TCXSPM}
are exponentially decaying if $\sqrt{\rho}(\zeta) < \tau$.

Our pointwise Weyl law  is a two-term asymptotic expansion for \eqref{TCXSPM} in terms of a  certain function $Q_{\zeta}(\lambda)$
(Theorem \ref{SHORTINTSa}). The $Q_{\zeta}$ function is the phase space analogue of a function introduced by Yu. Safarov
\cite{Saf,SV} in the real domain. Before stating the result, we digress to define $Q_{\zeta}$.

\begin{rem} \label{RENORMREM} 
Rather than use \eqref{TCXSPM},  one may wish  to work with sums of Husimi distributions,
$$
%\begin{equation} \label{RENORMWEYL}
\wt{P} _{[0, \lambda]}^{\tau}(\zeta, \bar{\zeta}) = 
\sum_{j: \lambda_j \leq \lambda }
\frac{|\phi_{j}^{\C}(\zeta)|^2}{ ||\phi_j^{\C}||^2_{L^2(\partial M_{\tau})}} , \;\; (\sqrt{\rho}(\zeta) =
\tau), $$
%\end{equation}
%adapted to the  Husimi distributions \eqref{HUSIMI}.  
%The reason to use \eqref{TCXSPM} instead  is that
But  analytic continuations
 $\phi_j^{\C}$ of an orthonormal basis of eigenfunctions are rarely
orthogonal on $\partial M_{\tau}$, unless $M$ has a large symmetry forcing the orthogonality. 
%In Section \ref{KtauSECT}, we introduce a
%kernel $K_{\tau}(\zeta, \overline{\zeta}')$  whose diagonal equals \eqref{RENORMWEYL}.
In general, we are unable to analyze sums of Husimi distributions directly; 
instead, we first use  \eqref{TCXSPM} and then derive results for Husimi distributions by using
Lemma \ref{L2LEMintro}. 
\end{rem}

\subsection{The $Q_{\zeta}(\lambda)$ function} To define $Q_{\zeta}$ we first need to introduce the {\it osculating Bargmann-Fock space} $\hcal^2_{\zeta}$ at $\zeta \in \partial M_{\tau}$ 
(Definition \ref{OSCBFDEF}).   This 
is the Bargmann-Fock space constructed from the  complexification $H_{\zeta}^{1,0} \oplus H_{\zeta}^{0,1}$  of the CR (Cauchy-Riemann) subspace $H_{\zeta}(\partial M_{\tau})$  in the complexified
tangent space $T_{\zeta} \partial M_{\tau} \otimes \C$ at  $\zeta$ (see  \eqref{CR} and Section \ref{CRSECT} for background
on CR structures and on Bargmann-Fock spaces and Section \ref{BFHSECT} for more details).  Thus,  $\hcal^2_{\zeta}$ 
 is the space  of entire holomorphic functions on 
$H_{\zeta}^{1,0}$ (with respect to the complex structure $J_{\zeta} $ of $M_{\tau}$ at $\zeta$) which are square integrable with respect to the {\it ground state} $\Omega_{J_{\zeta}}$ (defined in  \eqref{GSJ}). 

The  (real)  CR subspace $H_{\zeta} \subset T_{\zeta} (\partial M_{\tau})$ is tangent to a symplectic transversal to the geodesic flow. In this article, we use
a distinguished symplectic transversal that we term a 
`Phong-Stein leaf' (Section \ref{PSSECT}). 
Given  a periodic point $\zeta$ of $g_{\tau}^t$ of period $n$, we obtain a complexified linear symplectic Poincar\'e map \eqref{Dgtdef},
\begin{equation} \label{DGn} D_{\zeta} g^{n T(\zeta)} _{\tau}:  H_{\zeta}(\partial M_{\tau}) \to H_{\zeta}(\partial M_{\tau}),
%H^{1,0}_{\zeta} \oplus H^{0,1}_{\zeta} \to H^{1,0}_{\zeta} \oplus H^{0,1}_{\zeta}, 
 \end{equation}
on the
CR subspace. In a symplectic basis of  $H_{\zeta}(\partial M_{\tau})$,
% $H^{1,0}_{\zeta} \oplus H^{0,1}_{\zeta} $,
% so that $J = \begin{pmatrix} 0 & I
%\\&\\ -I & 0 \end{pmatrix}$ and relative to this basis we write
\begin{equation} \label{Dgn} D_{\zeta} g^{n T(\zeta)}= \begin{pmatrix} A_n(\zeta) & B_n(\zeta) \\ & \\ C_n(\zeta) & D_n(\zeta) \end{pmatrix} \in Sp(n, \R), \end{equation} 
where as usual $Sp(n, \R)$ denotes the symplectic group. 
For simplicity of notation, we often write   $$ S^n_{\zeta} := D_{\zeta} g^{n T(x)}_{\tau}.$$
The   symplectic  matrix  \eqref{Dgn} will arise often
and is discussed in Section \ref{METASECT} (see  \eqref{PDEF}).

 Since $(M, g)$ is real analytic, its exponential map $\exp_x t \xi$ admits an analytic
continuation in $t$ and the imaginary time exponential map
\begin{equation} \label{EXP} E: B_{\epsilon}^* M \to M_{\C}, \;\;\; E (x, \xi) = \exp_x i \xi \end{equation} is, for small enough $\epsilon$, a
diffeomorphism from the ball bundle $B^*_{\epsilon} M$ of radius
$\epsilon $ in $T^*M$ to the Grauert tube $M_{\epsilon}$ in
$M_{\C}$.  As reviewed in \S \ref{GRB}  (see \cite{GS1,LS1,ZPSH1,ZJDG}  for more details), $E$ conjugates the homogeneous
geodesic flow $G^t$ on $B^*_{\epsilon} M$ to the Hamiltonian flow of the Grauert tube function $\sqrt{\rho}$
with respect to the K\"ahler form $\omega_{\rho} = i \ddbar \rho$. We denote by
\begin{equation} \label{gttau} g^t_{\tau} = E  \circ  G^t \circ E^{-1} | _{\partial M_{\tau}} \end{equation}
the transfer  of the geodesic flow of $S^*_{\tau} M$ to $\partial M_{\tau}$.   We say that  $\zeta$  a periodic point if it is a periodic point of
$g_{\tau}^t$ and denote the set of periodic points by,  \begin{equation} \label{PCALDEF} \zeta \in \pcal \iff \zeta\; \rm{is \;a periodic \;point \;for\;}  g^t_{\tau}. \end{equation} We denote its primitive period by $T(\zeta)$,
Thus,
\begin{equation} \label{Tzeta} T(\zeta) = \inf\{ t > 0: g^t_{\tau}(\zeta) = \zeta\}. \end{equation}
The set of periodic points of period $T$ is the set of fixed points of $g^T_{\tau}$.  As usual, we say that the fixed point set $F$ of a map $T$ is clean if $F$ is a manifold and $T_x F = \rm{Fix} (D_x T)$. 

%Let $\pcal$ denote the set of periodic points of $g^t_{\tau}$. If $\zeta$
%does not lie on a closed geodesic, then we write $\zeta \notin \pcal$. If $\zeta$ does lie on a closed geodesic then we denote its primitive period
%by $T(\zeta)$.

Next, we define the metaplectic representation $W_{J_{\zeta}}$ of the derivative  $D g^t_{\zeta}$ on the osculating Bargmann-Fock space.   In the model
case of  $\R^{2m}$ with complex structure $J$, a  symplectic  linear map $S \in \rm{Sp}(m, \R)$ can be quantized, $S \to W_J(S)$ by 
the  metaplectic representation as a unitary operator on Bargmann-Fock space  (reviewed in \S \ref{METASECT}; see
 also Sections \ref{BFHSECT} and \ref{LINSECT}). Identifying \eqref{DGn}  with a symplectic map \eqref{Dgn}  on the model space, \eqref{DGn} may be quantized  as a unitary operator   on the  osculating Bargmann-Fock space
 at $\zeta$, 
\begin{equation} \label{WJ}  W_{J_{\zeta}} \;(D g^{n T(\zeta)}_{\zeta}): \hcal^2_{\zeta} \to \hcal^2_{\zeta}. \end{equation}
% We choose

The space $\hcal_{\zeta}^2$ has a distinguished  ground state $\Omega_{J_{\zeta}}$, a Gaussian associated to the complex structure $J_{\zeta}$  (see Section \ref{HEISMETSECT} and \eqref{GSJ} for
background).
We  denote by
\begin{equation} \label{ABCDintro} \gcal_n(\zeta): =  \langle W_{J_{\zeta}} \;(D g^{n T(\zeta)}_{\zeta}) \;\Omega_{J_{\zeta}}, \Omega_{J_{\zeta}} \rangle \end{equation}
the matrix element 
of  \eqref{WJ}
relative to the ground state $\Omega_{J_{\zeta}}$. As reviewed in Sections  \ref{METASECT} and  \ref{BFHSECT}  (see in particular, Lemma \ref{DAULEM} ),
\begin{equation} \label{ABCD} \begin{array}{lll} 
\gcal_n(\zeta) & = & 
2^{n/2} (\det \left(A_n(\zeta) + D_n(\zeta) + i (B_n(\zeta)  - C_n(\zeta)) \right)^{-\half} \\ &&\\ & = &
\det P_{J_{\zeta}} S^n_{\zeta} P_{J_{\zeta}}, \end{array} \end{equation}
where $ P_{J_{\zeta}} S^n_{\zeta} P_{J_{\zeta}}$ is the  holomorphic block  of a unitary conjugate of \eqref{Dgn}.
%in \eqref{gcalndef}.

Thus, the `quantum invariant' \eqref{ABCDintro} equals the `classical invariant' \eqref{ABCD}.
We then define the function  $Q_{\zeta}(\lambda)$ by:

\begin{defin} \label{QDEF} Recalling the set \eqref{PCALDEF},   \begin{equation} \label{Q} 
 Q_{\zeta}(\lambda) = \left\{\begin{array}{ll}
 =
  0, & 
  \zeta \; \notin \pcal \\ & \\ \sum_{n = 1} ^{\infty}  \frac{\sin  \lambda n T(\zeta)} {  n T(\zeta)}
\;\gcal_n(\zeta),&

% \langle W_{J_z} \;(D g^{n T}_{\zeta}) \;\Omega_{J_{\zeta}}, \Omega_{J_{\zeta}} \rangle, & 
 \; \zeta \in \pcal.\end{array}
\right. \end{equation}

% \edit{Why right definition? Is $\gcal_n = \gcal_{-n}$ by time-reversal invariance? is $g^T$ conjugate to its inverse, i.e. reversible? Why not 
% $$ \sum_{n \not=0} ^{\infty}  \frac{e^{i\lambda n T(\zeta)}} {  n T(\zeta)}
%\;\gcal_n(\zeta) $$}
\end{defin} 
Note that since \eqref{ABCD} is purely classical, \eqref{Q} gives a formula for $Q_{\zeta}(\lambda)$ defined
purely in terms of classical quantities. On the other hand, \eqref{ABCDintro} gives a `quantum formula'. 
By \eqref{ABCDintro},   
\begin{equation} \label{INFSERIES} \qcal_{\zeta}(\lambda) = \frac{1}{2i} \left(  \sum_{n = 1}^{\infty}  \frac{e^{i   \lambda n T(\zeta)}} {  n  T(\zeta)} \langle W_J(S_{\zeta})^n  \Omega_{J_{\zeta}}, \Omega_{J_{\zeta}}\rangle
-   \sum_{n = 1}^{\infty}  \frac{e^{- i   \lambda n T(\zeta)}} {  n T(\zeta)} \langle W_J(S_{\zeta})^{-n} \Omega_{J_{\zeta}}, \Omega_{J_{\zeta}}\rangle \right).
\;\end{equation}

The  classical formula \eqref{ABCD} seems simpler than the quantum formula, since it
comes down to the diagonalization of $S_{\zeta} \in \rm{Sp}(n, \R)$. This classical formula does not seem
to have an analogue in the real domain, hence does not have an analogue in \cite{Saf,SV} (it does, of course,
have an analogue in \cite{ZZ18}).

\subsection{Statement of the two term pointwise Weyl asymptotics}

The main result on pointwise Weyl asymptotics encompasses  three scenarios: (i) where $\zeta$ is not a periodic point; (ii) where $\zeta$
is a periodic point, and where $Q_{\zeta}(\lambda)$ is uniformly
continuous, and $P^{\tau}_{[0, \lambda] }(\zeta, \bar{\zeta}) $ admits asymptotics with a  well-defined `middle term'; (iii) where $\zeta$
is a periodic point, and $Q_{\zeta}$ has jumps.

\begin{defin} \label{JUMPSET} Let $\zeta \in \pcal$. 
 We define the jump-set of $Q_{\zeta}$ by,
$$\jcal(\zeta): = \{\lambda \in \R_+: [Q_{\zeta}(\lambda)] : = Q_{\zeta}(\lambda + 0) - Q_{\zeta}(\lambda - 0) > 0 \}.$$
%Given $\epsilon > 0$, we define the $\epsilon$-jump-set by,
%$$\jcal_{\epsilon} (\zeta): = \{\lambda \in \R_+: [Q_{\zeta}(\lambda)] : = Q_{\zeta}(\lambda + 0) - Q_{\zeta}(\lambda - 0) \geq  \epsilon \}.$$
%Further define,
%$$\jcal_{\epsilon}(\zeta) \cap \rm{Sp} (\sqrt\Delta). $$
\end{defin}

 In Section \ref{QzetaSECT}, we study the possible jumps of $Q_{\zeta}(\lambda)$. As mentioned above, the two equations
 \eqref{ABCDintro}, resp.  \eqref{ABCD} indicate that there is a  `quantum dynamical '  definition of $\jcal(\zeta)$, resp.
a `classical  mechanical' definition. Of course, they must agree.

%Before stating the result, we explain
%the last  notion more precisely. 

In the following, we enumerate  the jump  points
% \eqref{LAMBDAjell}
 as the  sequence $\{\nu_k\}_{k=1}^{\infty}$.

\begin{theo} \label{SHORTINTSa} Suppose  $(M, g)$ is real analytic, and $\zeta \in M_{\tau}$. Then, for fixed $\tau > 0$,

\begin{enumerate}

\item When $Q_{\zeta}(\lambda)$ is uniformly continuous in $\lambda$, then  for $\sqrt{\rho}(\zeta) \geq \frac{C}{\lambda}, $
$$ P^{\tau}_{[0, \lambda] }(\zeta, \bar{\zeta}) = \lambda \left(\frac{\lambda}{\sqrt{\rho}} \right)^{\frac{m - 1}{2}}
  \left( 1 +  Q_{\zeta}(\lambda) \lambda^{-1} +  o(\lambda^{-1})  \right);
$$

%\item For $\sqrt{\rho}(\zeta) \leq \frac{C}{\lambda}, $ $$
%P^{\tau}_{[0, \lambda]}(\zeta, \bar{\zeta}) =  (2\pi)^{-n} \;
%\lambda^{m} +   Q_{\zeta}(\lambda) \lambda^{m-1} + o(\lambda^{m-1}
%).
%$$

where
 the remainders are uniform in $\zeta$. Moreover, $Q_{\zeta}(\lambda)=0 $ if $\zeta \notin \pcal$. \bigskip
 
 \item  If  $\sqrt{\rho}(\zeta) = \tau$, if $\zeta \in \pcal$ and the fixed point set of $g_{\tau}^{T(\zeta)}$  is clean, 
  and if $Q_{\zeta}(\lambda) $ has jumps at the points $\jcal(\zeta) = \{\nu_k\}_{k=1}^{\infty}$, then there exists a sequence
   $\epsilon_k = O(\nu_k^{-\half})$ and eigenvalues $\{\lambda_{j_k}\}_{k=1}^{\infty} $ such that $\lambda_{j_k} = \nu_k +O(\epsilon_k)$ and a
   positive constant $C > 0$,
 such that 
 \begin{equation} \label{LBTHEO} P^{\tau}_{[\nu_k - \epsilon_k,  \nu_k + \epsilon_k]}(\zeta, \bar{\zeta}) 
 % P^{\tau}_{[0, \nu_k + \epsilon_k]}(\zeta, \bar{\zeta}) - P^{\tau}_{[0, \nu_k - \epsilon_k]]}(\zeta, \bar{\zeta})
  \geq C\; \nu_k^{\frac{m-1}{2}}. 
\end{equation}

\item   If  $\sqrt{\rho}(\zeta) = \tau$, if $\zeta \in \pcal$, then for all $\lambda \in \R_+$, and functions $o(\lambda)$ tending monotonically to $0$ as
$\lambda \to \infty$,
\begin{equation} \label{INEQ}\begin{array}{l}  \lambda \left(\frac{\lambda}{\tau} \right)^{\frac{m - 1}{2}}
  \left( 1 +  Q_{\zeta}(\lambda - o(\lambda)) \lambda^{-1} +  o(\lambda^{-1})  \right) \\ \\ \leq   P^{\tau}_{[0, \lambda] }(\zeta, \bar{\zeta}) 
  \\ \\ \leq   \lambda \left(\frac{\lambda}{\tau} \right)^{\frac{m - 1}{2}}
  \left( 1 +  Q_{\zeta}(\lambda + o(\lambda)) \lambda^{-1} +  o(\lambda^{-1})  \right).
\end{array} \end{equation}

\end{enumerate}

\end{theo}

In (3) we do not have asymptotics in the middle term when $Q_{\zeta}(\lambda)$ has a non-empty jumpset (Definition \ref{JUMPSET}). The necessity
of the inequalities \eqref{INEQ}  is due to the fact that the jump-set of $Q_{\zeta}(\lambda)$ is not necessarily  contained in  the jumpset of $P^{\tau}_{[0, \lambda]}$, namely the set $\{\lambda_j\}$.   The exact relation  is explained below Corollary \ref{SUPNORMCOR}.  Examples illustrating the scenarios are given in Section \ref{EXAMPLESECT}. 
The Zoll case illustrates the need for the  somewhat imprecise inequalities in (3).  As discussed in Section \ref{EXAMPLESECT} and in Section  \ref{ZOLL} in the Zoll case, there will exist a cluster
 of eigenvalues in  $\lambda_j \in  [\nu_k - \epsilon_k, \nu_k + \epsilon_k]$ of cardinality comparable to $\lambda^{m-1}$, for which each Husimi distribution
 has non-extremal sup norm but for which the sum has the lower bound (b).  
 However, in the Zoll case cases there is a more precise
result than Theorem  \ref{SHORTINTSa} (see Theorem \ref{Zoll}) and the 
unwieldy inequalities in (3) are only necessary  if one sums over intervals $[0, \lambda]$ which contain  incomplete
portions of the eigenvalue clusters reviewed in Section \ref{ZOLL}.  One obtains much better asymptotics
for the Weyl sums $P^{\tau}_{I_{\lambda}}$ over intervals
$I_{\lambda}$ containing exactly one cluster.

The next result (Theorem \ref{ATTAINED}) relates the continuity properties of $Q_{\zeta}$ to the  dynamical properties of the geodesic through $\zeta$.
To prepare for it we study we now give further information on $Q_{\zeta}$.

\subsection{\label{Qzeta} Further properties of $Q_{\zeta}$}

We denote by $\{x \}_{2 \pi} = x + 2 \pi \Z \in [0, 2 \pi)$  the residue of $x$ modulo $2 \pi$, which we identify with the  function $x$ on $[0, 2 \pi]$,  extended periodically of period $2 \pi$ to $\R$. Equivalently, $\{x \}_{2 \pi} = 2 \pi \{\frac{x}{2 \pi} + \half\}
-\pi$ where $\{x\}$ is the fractional part. 
Its Fourier series is given by
$ \{x- \pi \}_{2 \pi}  =  \sum_{n \not= 0} \frac{e^{in x}}{in}=  2 \sum_{n =1} ^{\infty} \frac{\sin n x}{n} $. \footnote{In \cite{SV}, $\{x \}_{2 \pi} = x + 2 \pi \Z \in [-\pi,  \pi)$. }

 In the following Proposition,   we recall that eigenvalues of a symplectic matrix occur in quadruples $\lambda, \bar{\lambda}, \lambda^{-1},
\bar{\lambda}^{-1}$; if $\lambda \in \R$ or $|\lambda | =1$, the eigenvalues come just in pairs.  We say that $S$ is semi-simple if it is diagonalizable over $\C$,  that $S$ is elliptic if its eigenvalues all have
modulus $1$, i.e. $S \in U(m)$, and that it is hyperbolic if it is positive symmetric symplectic and none of its eigenvalues
have modulus $1$.

\begin{prop} \label{CLJUMPPROP}  Assume that $\zeta \in \pcal$ and that  $S_{\zeta}$ is semi-simple. Then,
\begin{enumerate}
\item $Q_{\zeta}(\lambda)$ is uniformly continuous if and only if $S_{\zeta}$ is not elliptic. Hence,   $\jcal(\zeta) = \emptyset$ unless     $S_{\zeta}$
is elliptic. \bigskip

\item If $S_{\zeta}$ is non-degenerate elliptic, and if $ \det P_J S_{\zeta} P_J
= e^{is_0}$,  then $Q_{\zeta}(\lambda) = \{s_0 + \lambda T(\zeta) - \pi\}_{2 \pi}$, and
$$\jcal(\zeta) =  \{\lambda: s_0 + \lambda T(\zeta) =  \pi + 2 \pi \Z\}.$$
%\edit{Which $\epsilon$?}
\end{enumerate}

 \end{prop}

\subsubsection{Metaplectic quantum approach}
Although it is more complicated, it is also natural and interesting to use the quantum formula \eqref{ABCDintro} as in
\cite{Saf,SV}, and to compare the results with the classical approach. 
  To some extent, we adapt the notation and arguments of \cite{Saf,SV} on the quantum mechanical approach to the  jump behavior of $Q_{\zeta}(\lambda)$ to our setting.  However, even the spectral measure formula \eqref{SPMEAS}  does not seem to be mentioned in \cite{Saf, SV}, so our approach is not the same.
The  results in the real domain and complex domain are to some degree analogous, but  there are significant differences.
In the real domain, for analytic metrics,  rather than the metaplectic unitary
  quantization $W_{J_{\zeta}} \;(D_{\zeta} g^{n T(\zeta)})$ of the Poincar\'e map of a closed geodesic, one has a nonlinear first return map on directions
  of loops at a point $p \in M$, whose quantization is an operator on half-densities on $S_p^* M$. In the complex domain, 
  we have only the  periodic orbit
  of $\zeta$,  its linear Poincar\'e map,  and the metaplectic quantization of the Poincar\'e map.  The classical mechanical
  approach does not have a simple analogue in the real domain.

Although the exponential map $\exp: {\mathfrak s } {\mathfrak p }(n, \R) \to Sp(n, \R)$ from the symplectic Lie algebra to the
symplectic group is not surjective, we will assume with a rather small loss of generality 
 that $S = e^{i H}$ where $H \in  {\mathfrak s } {\mathfrak p }(n, \R) $.  The exponent $i H$ is not unique,  and in particular is only defined up to 
 the addition by $2 \pi \Z$ times the identity. However, all of the choices of logarithms will given the same results.  With   an abuse of
 notation we denote  the inverse exponential map by
 $\rm{arg}:  Sp(n, \R) \to  {\mathfrak s } {\mathfrak p }(n, \R) $ on the image of $\exp$, so that $e^{i \rm{arg}(S)} = S$.
Thus, as  in \cite[Proposition 1.8.12]{SV},  $\rm{arg}(W_J(S)) $ is a self-adjoint operator so that
$W_J(S) = \exp i \rm{arg}(W_J(S))$. 

In the quantum mechanical approach, we
use the spectral decomposition of the quadratic Hamiltonian
 $\rm{arg}(W_{J_{\zeta}} (S_{\zeta}))$. 
 We denote the possible pure point eigenvalues/eigenfunctions of $\rm{arg}(W_{J_{\zeta}} (S_{\zeta}))$ by $\{s_{\ell}\}$, resp.  $\{v_{\ell}\}_{\ell =1}^{\infty}$. Then,  
 \begin{equation} \label{sellDEF} W_{J_{\zeta}}(S_{\zeta}) v_{\ell} = e^{i s_{\ell}} v_{\ell}. \end{equation}
 Of course, $W_{J_{\zeta}}(S_{\zeta}) $ often has continuous spectrum as well (or, only has continuous spectrum). 
 For $\zeta$ such that there exist $s_{\ell}$ satisfying  \eqref{sellDEF}, we define \begin{equation} \label{LAMBDAjell} \Lambda_{\ell, j} = \frac{2 \pi j + s_{\ell}}{T(\zeta)}. \end{equation}

 We observe that $ \langle W_{J_{\zeta}}(S_{\zeta})^n  \Omega_{J_{\zeta}}, \Omega_{J_{\zeta}}\rangle$ is, by definition,  the $n$th moment
of the spectral measure $d\mu_{\zeta} $ on $S^1 = \{z: |z|=1\}$ of the unitary operator $W_{J_{\zeta}}(S_{\zeta})$ with 
respect to the ground state $ \Omega_{J_{\zeta}}$, i.e. 
\begin{equation} \label{SPMEAS} \langle W_{J_{\zeta}} (S_{\zeta})^n  \Omega_{J_{\zeta}}, \Omega_{J_{\zeta}}\rangle = \int_{S^1} \;e^{i n \theta} d\mu_{\zeta}. \end{equation}
The utility of \eqref{SPMEAS} depends on the extent to which the properties of the spectral measure $d\mu_{\zeta}$ can be determined.
 We denote by $\pi_{\Omega_{\zeta}} : = \Omega_{J_{\zeta}} \otimes \Omega_{J_{\zeta}}^*$ the orthogonal projection in the osculating  Bargmann-Fock space $\hcal^2_{\zeta}$ onto the ground state.
 
\begin{prop}\label{QzetaPROP}  With the above notation and conventions, 
\begin{enumerate}

\item In terms of the spectral measure \eqref{SPMEAS},
$$Q_{\zeta}(\lambda) =  \frac{1}{ T(\zeta)} \int_0^{2 \pi}   \{\theta  + \lambda T(\zeta) -\pi\}_{2 \pi} d\mu_{\zeta}. $$
%\item 
% $$Q_{\zeta}(\lambda) =  \frac{1}{ T(\zeta)} \langle  \{ \rm{arg} (W_{J_{\zeta}}(S_{\zeta})) + \lambda T(\zeta) -\pi\}_{2 \pi} %\Omega_{\zeta}, \Omega_{\zeta} \rangle, $$
%where the operator $ \{ \rm{arg} (W_{J_{\zeta}}(S_{\zeta})) + \lambda T(\zeta) - \pi\}_{2 \pi} $ is defined  by the spectral %theorem. \bigskip

\item 
$Q_{\zeta}$ is uniformly continuous if and only if $d\mu_{\zeta}$ is an absolutely continuous measure. 
%$\pi_{\Omega_{\zeta}} v_{\ell} =0$ for all 
% discrete eigenfunctions of $W_{J_{\zeta}}(S_{\zeta})$.
 \bigskip
 
 \item The atoms of $d\mu_{\zeta}(\theta)$ occur at the eigenvalues $e^{i s_{\ell}}$ of $W_{J_{\zeta}}(S_{\zeta})$, and 
 $\mu_{\zeta}(\{e^{i s_{\ell}}\}) = | \pi_{\Omega_{\zeta}}  v_{\ell} |^2. $
 % If $\{v_{\ell}\}$ is the subspace spanned by eigenfunctions  for which 
  %$\pi_{\Omega_{\zeta}} v_{\ell} \not= 0$, and  $e^{is_{\ell}}$ are the corresonding eigenvalues,  then 
% $$Q_{\zeta}(\lambda) = \frac{1}{T(\zeta)}  \sum_{\ell}  \{ s_{\ell}  +\lambda T(\zeta) -\pi\}_{2 \pi} ||\pi_{\Omega_{\zeta}} v_{\ell} ||^2,$$ 
\bigskip

\item 
$Q_{\zeta}$ has jumps at the points 
\eqref{LAMBDAjell} with $ \pi_{\Omega_{\zeta}}  v_{\ell}  \not=0$.

\end{enumerate}

\end{prop}  

See \cite[Theorem 1.8.17]{SV} for the corresponding statement in the real domain.

In Section \ref{GBSECT}, $Q_{\zeta}$ is calculated in the case where $\zeta$ generates a non-degenerate elliptic closed geodesic whose Poincar\'e map has eigenvalues $e^{i \alpha_j}$ with frequencies $(\alpha_1, \dots, \alpha_{m-1})$  independent,
together with $\pi$, over ${\mathbb Q}$.   In this case, $W_{J_{\zeta}}(S_{\zeta}) = \exp i \rm{arg}(W_{J_{\zeta}} (S_{\zeta})) $ is the unitary operator
generated by a Harmonic oscillator Hamiltonian. The spectrum of $\rm{arg}(W_{J_{\zeta}} (S_{\zeta}))$ is of the form $\{s_{\ell}\} = \{\sum_{j=1}^{m-1} 
\alpha_j (k_j + \half) \}_{\vec k \in {\mathbb N}^{m-1}}$  in the notation above. In this case, $\Omega_{\zeta}$ is itself
the ground state eigenfunction corresponding to $\vec k = 0$ and all other eigenfunctions are orthogonal to it. Hence,
in Proposition \ref{QzetaPROP} there is a single $\ell$, $||\pi_{\Omega_{\zeta}} v_{\ell}||^2_{L^2} =1$ and $s_{\ell} = \half  \sum_j  \alpha_j$. See also Proposition \ref{ELL}  for a general result in the elliptic case.
\subsection{Sup-norms of Husimi distributions and dynamics }

%The following theorem is a precise statement of most of Theorem \ref{ATTAINED} and incorporates the observations above. 
 We have now assembled enough background to state the main result:
Combining Theorem \ref{SHORTINTSa} with Proposition \ref{CLJUMPPROP} and Proposition \ref{QzetaPROP} gives the following,
\begin{theo} \label{ATTAINED}  Among real analytic Riemannian manifolds $(M,g)$ for which $D_{\zeta} g^t_{\tau}$ is semi-simple for all periodic 
points $\zeta \in \partial M_{\tau}$, 
the universal sup norm upper bound  bound of Theorem \ref{PWintro}(2)   is attained by $(M, g, \zeta, \{\phi_{j_k}\})$   only if:

\begin{enumerate}
\item  $\zeta$ is a
periodic orbit point of the geodesic flow $g^t_{\tau}$ of some period $T(\zeta)>0$, \bigskip

\item $S_{\zeta}: = D_{\zeta} g_{\tau}^{T(\zeta)}$ is an elliptic  semi-simple symplectic matrix, i.e. the orbit of $\zeta$ is an elliptic closed geodesic. \bigskip

\item With  $\Lambda_{\ell, k}$ as in \eqref{LAMBDAjell},   there  exist  $\epsilon_{\ell_k} \to 0$,  and a subsequence $j_{\ell_k}$ so that
$\left| \lambda_{j_{\ell_k}} - \Lambda_{\ell, k} \right| < \epsilon_{\ell_k},$
and $$
 \sum_{j:\Lambda_{\ell_k} - \epsilon_{\ell_k}
\leq  \lambda_j \leq 
 \Lambda_{\ell, k+1} - \epsilon_{\ell_k}} e^{- 2 \tau \lambda_j} |\phi_j^{\C} (\zeta)|^2 
=2 \pi \Lambda_{\ell_k}^{\frac{m-1}{2}} ||\pi_{\Omega_{\zeta}} v_{\ell}||^2_{L^2} + o(\Lambda_{\ell, j}^{\frac{m-1}{2}}).$$
Under these conditions, there   exist eigenvalues of $\sqrt{\Delta}$  lying in shrinking neighborhoods of jump points in $\jcal(\zeta)$.
\end{enumerate}

\end{theo}

The assumption of semi-simplicity in Theorem \ref{ATTAINED} is to simplify the discussion of the symplectic normal forms of symplectic matrices. 
In the Jordan normal form decomposition, the semi-simple and nilpotent parts need not be symplectic in general. Hence, we assume
for simplicity that all Poincar\'e type maps are semi-simple. This is an open-dense condition on symplectic matrices \cite{Gutt}. A simple example where
it is not satisfied is the flat torus, but it is easy to see (and proved  in Section \ref{TORUS}) that the universal sup norm bound is not attained in this
case either. There are more general examples of manifolds without conjugate points which are not covered by Theorem \ref{ATTAINED}, and which
doubtless do not attain the universal upper bound, but we omit these for the sake of brevity.

It follows from Theorem \ref{ATTAINED} that if a sequence of Husimi distributions $U_j^{\tau}$  \eqref{HUSIMI} attains the maximal 
$L^{\infty}$ bound at $\zeta$, then it obtains the bound on the whole closed geodesic  (Hamiltonian
orbit) $\gamma_{\zeta}$  with initial data $\zeta$. 
%  It would be interesting to relate this result to  C. Sogge's result  \cite[Proposition 3.1]{Sog11} that if $\gamma $ is not an arc of a periodic geodesic, then the %%configuration space Kakeya-Nikodym norms are not achieved. We expect that Sogge's result has an analogue for phase
%space  Kakeya-Nikodym norms. 
%Eigenfunctions cannot decay faster than a Gaussian in the transverse directions to
%$\gamma$ in configuration space $M$. The corresponding statement for the Husimi distributions is more complicated, since analytic continuation
%adds directions in which the complexified eigenfunctions may grow exponentially. But  one expects that the $L^{\infty} $  bound over the $\lambda^{-\half}$-tube around $\%gamma_{\zeta}$
%to saturate microlocal Kakeya-Nikodym norms. The relation between the norms is currently under investigation. See \cite[Section 3.3]{ZCBMS} and the Problems in %%\cite[Problem 10.3]{ZCBMS}.

\begin{rem} \label{SURFREV}   An interesting question is whether  Theorem \ref{SHORTINTSa} (2) and Theorem \ref{ATTAINED}(2) may be improved in some situations so  that
 they   imply an extremal lower bound on the sup-norm of a Husimi distribution  $\lambda_j \in [\nu_k - \epsilon_k, \nu_k + \epsilon_k]$ when $(M,g)$ possesses an elliptic closed geodesic.
In the general elliptic case, (b) only reflects the existence of a Gaussian
 beam quasi-mode, not an actual eigenfunction with of Gaussian beam type. Such eigenfunctions exist on  convex surfaces of revolution (see Section \ref{GBSECT}) . 
\end{rem}

%\subsection{The $Q_{\zeta}$ function}

The sup norm estimate of Theorem \ref{PWintro} is obtained from the jump of \eqref{TCXSPM}  at an eigenvalue. 
In the next Corollary, we equate
the jumps on the two sides of (1) of Theorem \ref{SHORTINTSa} and implies the result of Theorem \ref{PWintro} and the main input into the
results of Theorem \ref{ATTAINED}.
\begin{cor} \label{SUPNORMCOR} For any $(M, g)$, for fixed $\tau$  and for $\zeta \in \partial M_{\tau}$,
\begin{equation} \label{phiJUMP} \begin{array}{lll}  \sum_{j: \lambda_j = \lambda} e^{- 2 \tau \lambda_j} |\phi_j^{\C} (\zeta)|^2 & = &
  P^{\tau}_{[0, \lambda_j + 0]}(\zeta, \bar{\zeta}) - P^{\tau}_{[0, \lambda_j - 0]}(\zeta, \bar{\zeta}).
%\\ && \\
%& = &  (2\pi)^{-m}  \left(\frac{\lambda_j}{\sqrt{\rho}} \right)^{\frac{m - 1}{2}} \left(
%    Q_{\zeta}(\lambda_j + 0)  - Q_{\zeta}(\lambda_j - 0) +  o(1)  \right);
\end{array} \end{equation}
If $Q_{\zeta}(\lambda)$ is uniformly continuous in $\lambda$,  i.e. has no jumps, then 
$$ e^{- 2 \tau \lambda_j} |\phi_j^{\C} (\zeta)|^2 = o_{\tau} (\lambda_j^{\frac{m - 1}{2}}). $$
Hence, if  $Q_{\zeta}$ is uniformly continuous in $\lambda$ for all $\zeta$, then
$$\sup_{\zeta \in \partial M_{\tau}} e^{- 2 \tau \lambda_j} |\phi_j^{\C} (\zeta)|^2 = o_{\tau}(\lambda_j^{\frac{m - 1}{2}}). $$

The analogous results for \eqref{HUSIMI} follow from Lemma \ref{L2LEMintro}.

\end{cor}

 We observe that  jumps in spectral functions arise in two ways in Corollary \ref{SUPNORMCOR} : \begin{itemize}

\item (i) \;
Jumps $ P^{\tau}_{[0, \lambda_j + 0]}(\zeta, \bar{\zeta}) - P^{\tau}_{[0, \lambda_j - 0]}(\zeta, \bar{\zeta})
$  in the Weyl function at eigenvalues of $\sqrt{\Delta}$, which are non-zero as long as $\phi^{\C}_j(\zeta) \not=0$ for
some $j$ with $\lambda_j = \lambda$;  \bigskip

\item (ii) Jumps $Q_{\zeta}(\nu_k + 0) - Q_{\zeta}(\nu_k - 0)$ in  $Q_{\zeta}$ function  at its jump discontinuities $\nu_k$; \bigskip

\end{itemize}

If $Q_{\zeta}(\lambda)$ has jump discontinuities at points $\nu_k$ with $Q_{\zeta}(\nu_k + 0) - Q_{\zeta}(\nu_k -0) \geq C_1 > 0$,
then there exists a sequence $\epsilon_k \to 0$ such that $$ P^{\tau}_{[0, \nu_k + \epsilon_k]}(\zeta, \bar{\zeta}) - P^{\tau}_{[0, \nu_k - \epsilon_k]]}(\zeta, \bar{\zeta}) \geq C_2 \nu_k^{\frac{m-1}{2}}. 
$$ When the fixed point sets are non-degenerate (and necessarily elliptic in the jump case), one may take $\epsilon_k = \frac{1}{k}$. This implies
that there exist eigenvalues $\lambda_{j_k}$ such that $|\lambda_{j_k} - \nu_k| < \nu_k^{-\half}$. There are several ways that such eigenvalues
can arise. First, $(M,g)$ might be a Zoll manifold, all of whose geodesics are closed. In that case, the spectrum of $\sqrt{\Delta}$ occurs in 
clusters of width $k^{-1}$ around an arithmetic progression $\{\nu_k\}_{k=1}^{\infty}$; see Theorem \ref{Zoll} and Section \ref{ZOLL} for
precise statements. In this case, $\epsilon_k = O(k^{-1})$ is the minimal possible size for the lower bound above, because the clusters  centered
at the points $\nu_k$ have widths $O(k^{-1})$ (see \cite{DG}). In intervals  $\lambda \in I_k$ outside the union of the clusters, there are no jumps in 
$Q_{\zeta}(\lambda)$ and $P^{\tau}_{[0,\lambda]) }(\zeta, \bar{\zeta}) $ is constant. 
A second scenario is illustrated by a  generic convex  surface of revolution $(S^2, g)$. In this case,  the spectrum of $\sqrt{\Delta}$ is evenly distributed in intervals
$[\lambda, \lambda +1]$, so there is no clustering of its eigenvalues; hence, jump behavior is not a spectral invariant but  is  due to the existence of  special eigenfunctions (Gaussian beams) centered along elliptic closed geodesics $\gamma$ whose Husimi measures attain maximal size. If $\gamma$ is
the orbit of $\zeta$, then  $Q_{\zeta}(\lambda)$ exhibits jumps along a dynamically defined  arithmetic progression $\{\nu_k\}$  (see Section \ref{Qzeta}), and 
there exist eigenvalues $\{\lambda_{j_k}\}$ close to $\nu_k$ at which $P^{\tau}_{[0,\lambda]) }(\zeta, \bar{\zeta}) $ has the jumps above. These associated
eigenvalues can be constructed by the elliptic quasi-mode construction \cite{Ral82, BB91} and from this construction one can see that $|\lambda_{j_k} - \nu_k| \leq k^{-\half}.$

\subsection{\label{OUTLINE} Outline of the proofs}

%The proof in the general case  uses the Tauberian theorem of Safarov et al \cite{SV}, which applies in situations where the%singularities
%of the Fourier transform of the $d P_{[0, \lambda]} (\zeta, \overline{\zeta})$ \eqref{SPPROJDAMPED} at certain $t \not= 0$ are as %great as the singularities at $t = 0$. This produces an oscillatory
%secoond term. 
%Whether the term $Q_{\zeta}(\lambda) \lambda^{m-1}$ is continuous or has jumps depends on the decay
%rate of the coefficient $  \langle e_{\Lambda_{n T(\zeta)}}, e_{\Lambda} \rangle$. As mentioned below, it
%is continuous if $(M, g)$ is a manifold without conjugate points but is discontinuous for a Zoll manifold. 

The proof of Theorem of \ref{SHORTINTSa}  is  based on Fourier Tauberian arguments relating  tempered spectral projection
measures 
  \begin{equation} \label{SPPROJDAMPED} d_{\lambda} P_{[0, \lambda]
  }^{\tau}(\zeta, \bar{\zeta}) = \sum_j \delta(\lambda -
 \lambda_j) e^{- 2 \tau \lambda_j} |\phi_j^{\C}(\zeta)|^2,  \;\;\; (\tau = \sqrt{\rho}(\zeta)) \end{equation}
to their Fourier transforms. Note that \eqref{SPPROJDAMPED}
  is a  temperate distribution on $\R$ for each $\zeta$  satisfying $\sqrt{\rho}(\zeta)
 \leq \tau. $ 
%\begin{equation}\label{CXDSP}  P^{\tau}_{I}(\zeta, \bar{\zeta}) =
% \sum_{j: \lambda_j \in I}  e^{- 2 \tau \lambda_j}
%|\phi_j^{\C}(\zeta)|^2,
%\end{equation}
%\edit{\begin{equation}\label{TCXSPM}   P_{ I_{\lambda}}^{\tau}(\zeta, \bar{\zeta}) =
%\sum_{j: \lambda_j \in  I_{\lambda}} e^{-2 \tau \lambda_j}
%|\phi_{j}^{\C}(\zeta)|^2, \;\; (\sqrt{\rho}(\zeta) \leq
%\tau),
%\end{equation}}
%When we set $\tau = \sqrt{\rho}(\zeta)$ we omit the
% $\tau$ and put
%  \begin{equation} \label{SPPROJDAMPEDz} d_{\lambda} P_{[0, \lambda]
 % }(\zeta, \bar{\zeta}) = \sum_j \delta(\lambda -
 %\lambda_j) e^{- 2 \sqrt{\rho}(\zeta) \lambda_j}
 %|\phi_j^{\C}(\zeta)|^2.
 %\end{equation}
%%and also
%\begin{equation}\label{CXDSPa}  P_{I}(\zeta, \bar{\zeta}) =
% \sum_{j: \lambda_j \in I}  e^{- 2 \sqrt{\rho}(\zeta)\lambda_j}
%|\phi_j^{\C}(\zeta)|^2,
%\end{equation}

We study analytic continuations of eigenfunctions, as in \cite{ZJDG}, using the Poisson kernel,
\begin{equation} \label{PTKER} 
P^{\tau}(\zeta, y) := \sum_j e^{- \tau \lambda_j} \phi_j^{\C}(\zeta) \phi_j(y), 
\end{equation}
which has the property,
\begin{equation} \label{EIGCX}  P^{\tau} \phi_j (\zeta) = e^{- \tau \lambda_j} \phi_j^{\C}(\zeta). \end{equation}
%\begin{rem} Although we do not do so here, it would be interesting to study the mapping properties of \eqref{EIGCX} on $L^{\infty}$,
%which would bound the sup norm of the Husimi distribution  by the $L^1$ norm of the Poisson kernel and the $L^{\infty}$ norm of the eigenfunction
%in the real domain. \end{rem} 
To obtain the Weyl asymptotics for \eqref{TCXSPM} we study the singularities in $t$ of the  Fourier transform  of \eqref{SPPROJDAMPED},
\begin{equation} \label{CXWVGP} \begin{array}{lll} U_{\C} (t + 2 i \tau, \zeta, \bar{\zeta}) &:= & \fcal_{\lambda \to t}d_{\lambda} P_{[0, \lambda]
  }^{\tau}(\zeta, \bar{\zeta}) =   \sum_j
e^{(- 2 \tau + i t) \lambda_j} |\phi_j^{\C}(\zeta)|^2,
%\\ && \\ & =
%& \int_{\R} e^{i t \lambda} d_{\lambda} P_{[0, \lambda]}^{\tau}
%(\zeta, \bar{\zeta}).
 \end{array} \end{equation}
 whose properties may be deduced from those of \eqref{PTKER}.
Here, the wave kernel $U(t, x, y)$ is the kernel of $e^{i t \sqrt{\Delta}}$ and \eqref{CXWVGP}  is the Poisson
wave kernel obtained by analytically continuing the wave kernel in time and in space. The asymptotics are obtained by constructing a
 parametrix for \eqref{CXWVGP} as
a Fourier integral Toeplitz operator (or dynamical Toeplitz operator)  in 
Proposition \ref{MAIN}.  We  used such a construction in  \cite{ZJDG} in studying analytic
continuation of eigenfunctions.

\begin{prop} \label{LINEAR} For each $\zeta$, \eqref{CXWVGP} is a homogeneous
Lagrangian (Fourier integral) distribution with complex phase in $t$ which is singular at $t = 0$ and at the
periods $t = n T(\zeta)$ of $\zeta$ if it is a periodic point (Definition \ref{POINC}). The principal symbol of 
$t \to U_{\C} (t + 2 i \tau, \zeta, \bar{\zeta})$ at a period $n T(\zeta)$
 is given by $\gcal_n(\zeta) =  \langle W_{J_{\zeta}} \;(D g^{n T({\zeta})}_z{\zeta}) \;\Omega_{J_{\zeta}}, \Omega_{J_{\zeta}} \rangle$ \eqref{ABCDintro}.
%=  (\det (A + D + i B - i C)^{-\half},$$
%where $D g^{n T(\zeta)}_{\zeta}$ is the symplectic matrix $\begin{pmatrix} A & B \\ & \\ C & D %\end{pmatrix}.$

\end{prop}

To prove the Proposition, we construct the wave group in the complex domain (as in \cite{ZJDG}) 
 as a dynamical Toeplitz operator of the form, 
\begin{equation} \label{CXWAVEGROUPintro} V^t_{\tau}: = \Pi_{\tau} g_{\tau}^t \sigma_{t, \tau} \Pi_{\tau},
\end{equation} 
where $\Pi_{\tau}$ is the Szeg\"o kernel of $\partial M_{\tau}$,  $g^t_{\tau}$ is translation by the geodesic
flow \eqref{gttau} and $\sigma_{t, \tau}$ is a certain  symbol, designed to make \eqref{CXWAVEGROUPintro}
a unitary group.  In  Proposition \ref{MAIN} it is shown (roughly speaking)  that 
\begin{equation} \label{UandV} U_{\C}(t + 2 i \tau, \zeta, \overline{\zeta}) = V_{\tau}^t ( \zeta, \overline{\zeta}). \end{equation}

%\begin{theo} \label{SCALINGTHEO} Let $\zeta \in \partial M_{\tau} $ be a periodic point of $g^t_{\tau} $.  Let $W_{J_{\zeta}} (D_{\zeta} \gtc{T_{\zeta}})$ denote %the metaplectic representation for complex structure
%$J_{\zeta}$  \eqref{WJ} applied to the symplectic linear transformation $ D_{\zeta} \gtc{T_{\zeta}} $ \eqref{DGn}, and let  $\Omega_{J_{\zeta}}$ be the ground %%state of the osculating Bargmann-Fock representation 
%at $\zeta$ \eqref{GSJ}.   Then, for $ \chi  \in \scal(\R)$ with $\hat \chi \in C^\infty_0(\R)$,   $\chi* d P^{\tau}_{[0, \lambda]} (\zeta, \bar{\zeta}) $ 
%and $ \Pi_{\chi, \tau}(\lambda,  \zeta, \bar{\zeta})$
%admits a complete asymptotic expansions as $\lambda \to \infty$,  with leading order term,
%$$  \chi* d P^{\tau}_{[0, \lambda]} (\zeta, \bar{\zeta}) =  \left\{ \begin{array}{ll}  \frac{\lambda^{m-1}}{\tau^{3-m}}  \hat{\chi}(0) \gcal_0(\zeta)(1+ O(\lambda^{-1})), %&\zeta \notin \pcal\\ &\\ \frac{\lambda^{m - 1}}{\tau^{3 - m}} \sum_{n \in \Z}  \hat{\chi}(n T_{\zeta})
%\;\gcal_n(\zeta) e^{ - i T_{\zeta}(n) \lambda}  + O(\lambda^{m-3/2}), &  \zeta \in \pcal. \end{array} \right.,
%$$
% where $\gcal_n$ is defined in \eqref{ABCD}.
%\end{theo}

To determine the singularities in Proposition \ref{LINEAR}, we  convolve with a suitable test function $\chi$ and  determine
the leading term as $\lambda \to \infty$ of  the smoothed temperate sums,
\begin{equation} \label{SMOOTH} \chi* d P^{\tau}_{[0, \lambda]} (\zeta, \bar{\zeta})= \int_{\R} \hat{\chi}(t) e^{i
\lambda t} U_{\C} (t + 2 i \tau, \zeta, \bar{\zeta}) dt.  \end{equation}
The asymptotics of 
\eqref{SMOOTH} are  determined by substituting this expression into \eqref{CXWAVEGROUPintro}, 
using the Boutet de Monvel-Sj\"ostrand parametrix for $\Pi_{\tau}$ and then employing  the stationary phase method in the complex domain in \S \ref{TWOTERM},
to obtain,

\begin{theo} \label{SCALINGTHEO}   Let  $m = \dim M$.  Let $\zeta \in \partial M_{\tau} $ be a periodic point of $g^t_{\tau} $. 
% Let $W_{J_{\zeta}} (D_{\zeta} \gtc{T_{\zeta}})$ denote the metaplectic representation for complex structure
%$J_{\zeta}$  \eqref{WJ} applied to the symplectic linear transformation $ D_{\zeta} \gtc{T_{\zeta}} $ \eqref{DGn}, and let  $\Omega_{J_{\zeta}}$ be the ground %state of the osculating Bargmann-Fock representation 
%at $\zeta$ \eqref{GSJ}. 
  Then, for $ \chi  \in \scal(\R)$ with $\hat \chi \in C^\infty_0(\R)$,   there exist positive universal dimensional constants $C_m, C_m'$ so that $\chi* d P^{\tau}_{[0, \lambda]} (\zeta, \bar{\zeta}) $ 
%and $ \Pi_{\chi, \tau}(\lambda,  \zeta, \bar{\zeta})$
admits a complete asymptotic expansions as $\lambda \to \infty$,   satisfying
$$  \chi* d P^{\tau}_{[0, \lambda]} (\zeta, \bar{\zeta}) =  \left\{ \begin{array}{ll} C_m  \lambda^{\frac{m-1}{2}} +  \ocal(\lambda^{\frac{m-3}{2} }),  &\zeta \notin \pcal, \\ & \\
C_m  \lambda^{\frac{m-1}{2}} +
C_m'    \lambda^{\frac{m-1}{2}} \Re \sum_{n = 1}^{\infty} \hat{\chi}(n T(\zeta))  e^{ - i\lambda n T(\zeta)} \gcal_n(\zeta)
 + \ocal(\lambda^{\frac{m-3}{2} }), & \zeta \in \pcal. \end{array} \right. $$
 
\end{theo}

The main difficulty in the proof lies  in interpreting the Hessian determinants
in the stationary phase expansion explicitly in geometric terms, which
is necessary in understanding the convergence of the $Q_{\zeta}(\lambda)$
function.
As is proved in Lemma \ref{DAULEM}, the principal term
  on the right side of Theorem \ref{SCALINGTHEO} is the Gaussian integral,

 \begin{align}  \ip{W_{J_{\zeta}}\left(S_{\zeta}\right) \left( \Omega_{J_{\zeta}} \right)}{\Omega_{J_{\zeta}}} = 
    \int_{\C^{m - 1}}^{}e^{ -\frac{1}{\tau}\left( |u|^2_{L(\zeta)} + |S_{\zeta}(u)|^2_{L(\zeta)} \right)} du. \label{Gaussian Integral} 
\end{align}
where $S_{\zeta} = D_{\zeta} g_{\tau}^{T_{\zeta}} $ and where $L_{\zeta}$ is the Levi metric at $\zeta$.
In principle, the  principal symbol  could  be calculated using the Boutet de Monvel-Guillemin
symbol calculus for Toeplitz operators \cite{BoGu}, whose purpose
 is to explicitly evaluate Hessian determinants
in the stationary phase formulae in a metaplectic way. But their
calculus involves a somewhat abstract comparison to a Grushin system of harmonic oscillators
on $\R^n$. Instead,  we use a simpler and more natural approach  in the complex setting of combining the   Boutet de Monvel-Sj\"ostrand parametrix
and  the   `osculating Bargmann-Fock' representation. A novelty is that we use a geometric construction of Phong-Stein \cite{PhSt1,PhSt2} of
a certain foliation to construct representations of the relevant oscillatory integrals (see below).
%The foliation of $\partial M_{\tau}$ introduced by  Phong-Stein was used in their analysis of singular Radon transforms on strictly pseudo-convex domains
%in \cite{PhSt1,PhSt2}. 
It is the codimension one foliation $\Im \psi(\zeta,w) = 0$ of $\partial M_{\tau}$ associated to the phase $\psi(x, y)$  of the
Boutet-de-Monvel-Sj\"ostrand parametrix.  Expressing oscillatory integrals as integrals over the  leaves allows one to compute the principal term of the asymptotics  in a geometrically transparent fashion. The Phong-Stein leaves are introduced in Section \ref{PSSECT} and their application to the computation
is given in Section \ref{PSLOCAL}.

With no additional effort we prove the analogue of Theorem \ref{SCALINGTHEO} for a  `purely dynamical'
operator kernel in the Grauert tube setting, namely the spectral projections kernel  \begin{equation} \label{chilambda} \Pi_{\chi, \tau}(\lambda) : =   \chi(\Pi_{\tau} D_{\sqrt{\rho}} \Pi_{\tau} - \lambda) = \Pi_{\tau} \int_{\R} \hat{\chi}(t) e^{- i t \lambda} e^{it\Pi_{\tau} D_{\sqrt{\rho}} \Pi_{\tau}} dt, \end{equation}
of the  Toeplitz differential operator  \begin{equation} \label{DDEF} \Pi_{\tau} D_{\sqrt{\rho}} \Pi_{\tau} : H^2(\partial M_{\tau}) \to H^2(\partial M_{\tau}) \end{equation} on $L^2(\partial M_{\tau})$ where $D_{\sqrt{\rho}} $ is $\frac{1}{i} \Xi_{\sqrt{\rho}}$, where $\Xi_{\sqrt{\rho}}$ is the Hamilton vector field of the Grauert tube function $\sqrt{\rho}$ acting as a differential operator.   Here,
 $\chi \in \scal(\R)$ (Schwartz space) with $\hat{\chi} \in C_0^{\infty}(\R)$; the extra $\Pi_{\tau}$ is needed to define the unitary group
\begin{equation} \label{WCALintro}  \wcal_{\tau}(t): = \Pi_{\tau}  e^{it\Pi_{\tau} D_{\sqrt{\rho}} \Pi_{\tau}} : H^2(\partial M_{\tau}) \to H^2(\partial M_{\tau}) \end{equation} We note that \eqref{WCALintro} is  very similiar to   \eqref{CXWAVEGROUPintro} and that \eqref{chilambda}
is very similar to \eqref{SMOOTH}. The proof of Theorem \ref{SCALINGTHEO} applies with no essential change to \eqref{chilambda},  and
Theorem  \ref{SHORTINTSa} is also valid for \eqref{DDEF}. We introduce
\eqref{WCALintro} because it is the natural Koopman dynamics in the Grauert tube setting, defined entirely in terms of the complex structure and
the geodesic flow. Its spectral theory  seems of independent interest. We denote the Schwartz kernel of $\Pi_{\chi, \tau}(\lambda) $ on the diagonal by $\Pi_{\chi, \tau}(\lambda, \zeta,
\bar{\zeta})$.

\begin{theo} \label{PichilambdaTHEO}  Let  $m = \dim M$.  Let $\zeta \in \partial M_{\tau} $ be a periodic point of $g^t_{\tau} $. 
% Let $W_{J_{\zeta}} (D_{\zeta} \gtc{T_{\zeta}})$ denote the metaplectic representation for complex structure
%$J_{\zeta}$  \eqref{WJ} applied to the symplectic linear transformation $ D_{\zeta} \gtc{T_{\zeta}} $ \eqref{DGn}, and let  $\Omega_{J_{\zeta}}$ be the ground %state of the osculating Bargmann-Fock representation 
%at $\zeta$ \eqref{GSJ}. 
  Then, for $ \chi  \in \scal(\R)$ with $\hat \chi \in C^\infty_0(\R)$,   there exist positive universal dimensional constants $C_m, C_m'$ so that $\chi* d P^{\tau}_{[0, \lambda]} (\zeta, \bar{\zeta}) $ 
%and $ \Pi_{\chi, \tau}(\lambda,  \zeta, \bar{\zeta})$
admits a complete asymptotic expansions as $\lambda \to \infty$,   satisfying
$$ \Pi_{\chi, \tau}(\lambda, \zeta, \bar{\zeta}) =  \left\{ \begin{array}{ll} C_m  \lambda^{m-1} +  \ocal(\lambda^{m-3 }),  &\zeta \notin \pcal, \\ & \\
C_m  \lambda^{m-1} +
C_m'  \lambda^{m-1} \Re \sum_{n = 1}^{\infty} \hat{\chi}(n T(\zeta))  e^{ - i\lambda n T(\zeta)} \gcal_n(\zeta)
 + \ocal(\lambda^{m-3}), & \zeta \in \pcal. \end{array} \right. $$
 
\end{theo}

The only significant difference between Theorem \ref{SCALINGTHEO} and Theorem \ref{PichilambdaTHEO} lies in the
power of $\lambda$, reflecting that the two operators have a different power of $|\xi|$ in their principal symbols. The origin 
of this difference lies in the operator $A_{\tau}$ in   Lemma \ref{OLD}, which does not arise in \eqref{chilambda} and which
lowers the order of $P^{\tau}_{[0, \lambda]}$ relative to \eqref{chilambda}.  Since $P^{\tau}_{[0, \lambda]}$ is the main focus
of this article, we will carry out the analysis in more detail for this kernel and then explain the very simple modifications necessary
to deal with \ref{chilambda}. 
A comparison of the two types of operator kernels in  Theorem \ref{SCALINGTHEO},
and a comparison with \szego\; kernel asymptotics on line bundles,  is given in  Section \ref{COMPARISON}.

To complete the proof of Theorem \ref{SHORTINTSa}, we use the  Fourier Tauberian method of Safarov \cite{Saf, SV}   (see \S \ref{TAUB}).
The singularities  of
$U_{\C}(t + 2 i \tau, \zeta, \bar{\zeta})$ at   periods $t \not= 0$ of the 
periodic orbit through $\zeta$  are all of the same degree (strength) as the singularity at $t =
0$. That is why one needs to sum over periods. As discussed above, they give rise to an oscillating second term or possibly a discontinuous middle term
depending on the continuity of the function $Q_{\zeta}(\lambda)$ \eqref{Q}.
\bigskip

\subsection{\label{RELATED} Related problems and results}

Q-functions in the real domain were   introduced by Safarov  (see \cite{Saf,SV}) to obtain two-term Weyl laws.
The pointwise
asymptotics are sharper in the complex domain  than in the real domain, and resemble the two term pointwise quasi-Weyl asymptotics of Safarov et al \cite{Saf,SV}.
%; we review these asymptotics in \S \ref{Saf}.
%The estimate of the remainder in the case where $\zeta$ is not a periodic point %is analyzed in detail in 
%\S \ref{INTREM}.
An interesting aspect of Theorem \ref{SHORTINTSa}  is that the formula for $Q_{\zeta}$ is valid even if the closed geodesic
through $\zeta$ is degenerate as a closed
geodesic in the real domain. For instance, it is valid on a sphere, Zoll manifold or flat torus (Section \ref{ZOLL}). 
The analogous two-term formula in the real domain involves an integration
over the set $\lcal_x \subset S^*_x M$ of loop directions. In the real domain, $\lcal_x$ might be dense in $S^*_x M$
and additionally might fail to have the kind of cleanliness or transversality properties that are required for application
of the stationary phase method. In contrast, in the complex domain there is a single critical point when $\zeta$ is
a periodic point and the stationary phase method is always applicable.

As mentioned above, it follows from Theorem \ref{ATTAINED} that if a sequence of Husimi distributions $U_j^{\tau}$  \eqref{HUSIMI} attains the maximal 
$L^{\infty}$ bound at $\zeta$, then it obtains the bound on the whole closed geodesic  (Hamiltonian
orbit) $\gamma_{\zeta}$  with initial data $\zeta$.   It would be interesting to relate this result to  C. Sogge's result  \cite[Proposition 3.1]{Sog11} that if $\gamma $ is not an arc of a periodic geodesic, then the configuration space Kakeya-Nikodym norms are not achieved. We expect that Sogge's result has an analogue for phase
space  Kakeya-Nikodym norms. 
Eigenfunctions cannot decay faster than a Gaussian in the transverse directions to
$\gamma$ in configuration space $M$. The corresponding statement for the Husimi distributions is more complicated, since analytic continuation
adds directions in which the complexified eigenfunctions may grow exponentially. But  one expects that the $L^{\infty} $  bound over the $\lambda^{-\half}$-tube around $\gamma_{\zeta}$
to saturate microlocal Kakeya-Nikodym norms. The relation between the norms is currently under investigation. See \cite[Section 3.3]{ZCBMS} and the Problems in \cite[Problem 10.3]{ZCBMS}.

Another comparison in terms of  techniques and results  in the real domain is to compare the complexification
techniques of this article to the use of Gaussian beam decompositions  in the work of Canzani-Galkowski (see  \cite{CG19, CG19b})
and the use of defect measures in \cite{G19}.  
As mentioned above, it would be interesting to study the relation between $L^{\infty}$ norms of Husimi distributions \eqref{HUSIMI} and
the
microlocal  Kakeya-Nikodym norms defined above. It would also be interesting to compare these norms to the ones in 
 \cite{BlS17} and explore the relations to $L^p$ norms on $M$.

The  Husimi distributions \eqref{HUSIMI} and pointwise Weyl asymptotics for $ P_{ I_{\lambda}}^{\tau}(\zeta, \bar{\zeta})$ can be pushed-forward to $M$ by
integrating over the fibrers of the natural projections $\pi: \partial M_{\tau} \to M$ (essentially the cosphere bundle of radius $\tau$). 
This is straightforward but postponed to a latter occasion.  Recalling that Wigner distributions and other `microlocal lifts' of eigenfunctions push
forward to the squares of the eigenfunctions, it would be interesting to determine the pushforwards of the Husimi distributions. The interesting aspect is that only $\zeta \in \pi^{-1}(x)$ which lie on closed
geodesics contribute sub-principal terms to the asymptotics, whereas sup-norm bounds on eigenfunctions reflect the measure
of all closed geodesic loops (see \cite{SoZ16}).

\subsection{Acknowledgements}  Thanks to Peng Zhou for collaboration on the related article \cite{ZZ18} and to Yaiza Canzani,  Robert Chang   and Abe Rabinowitz for detailed comments and corrections.

\section{\label{EXAMPLESECT} Examples illustrating the continuity properties of  $Q_{\zeta}(\lambda)$ }

Let us give simple examples where $Q_{\zeta}(\lambda)$ and others where it has jumps.

\begin{prop} \label{HYPLEM}If $\zeta$ generates a hyperbolic closed geodesic,  i.e. if   $ D_{\zeta} g^{n T(\zeta)}$ \eqref{DGn}  is a real hyperbolic symplectic map, then $Q_{\zeta}(\lambda)$ is uniformly continuous in $\lambda$. \end{prop}
%  If the closed geodesic is  elliptic, then  $Q_{\zeta}(\lambda)$ has jumps
%as $\lambda$ varies.
Indeed, in this case $P = (S^T S)^{\half}$  is diagonalizable with real positive eigenvalues
that come in pairs $e^{\mu_j}, e^{-\mu_j}$.  Since  $ D_{\zeta} g^{n T(\zeta)} =
 ( D_{\zeta} g^{ T(\zeta)})^n$,  with eigenvalues  $e^{\pm n \mu_j}$, it is not hard to 
prove that the series \eqref{Q}  convergences. 
Hence,  $Q_{\zeta}(\lambda)$ is
continuous in $\lambda$ for manifolds with only hyperbolic geodesics,
such as negatively curved manifolds.

As is easily proved by elementary calculations,   $Q_{\zeta}(\lambda)$  is also uniformly continuous on flat tori (Section \ref{TORUS}). It is natural
to conjecture that  $Q_{\zeta}(\lambda)$ is continuous  at all periodic points on a
general $(M, g)$ without conjugate points.

As a further  example with conjugate points, all geodesics through an umbilic point $u$ of a tri-axial ellipsoid  $\ecal \subset \R^3$
are geodesic loops at $u$ of period $2 \pi$, but only one direction $\zeta$ at $u$ (up to time reversal) gives a smoothly closed hyperbolic geodesic. If
$\zeta \in S^*_u \ecal$, then  $Q_{\zeta} =0$  unless $\zeta$ is the closed geodesic direction. Since it is hyperbolic, there are no jumps in $Q_{\zeta}$
for such $\zeta$. The ellipsoid does contain elliptic closed geodesics (namely, the elliptical slices by coordinate hyperplanes $x_1 =0, $ or $x_3 = 0$,
so it is not clear that the universal sup norm bounds are not attained by a sequence of eigenfunctions of $\ecal$. The results of \cite{SoZ16} disqualify
all points of $\ecal$ for attainment of  the universal sup-norm bound \eqref{REALSUP} in the real domain.

\subsubsection{\label{JUMPS}  Examples where $Q_{\zeta}$ has jumps}

  In contrast to Proposition \ref{HYPLEM},  $Q_{\zeta}(\lambda)$ has
jump discontinuities of a rigid kind in the case of the standard sphere (Section \ref{SPHERESECT})  or Zoll manifolds (Section \ref{ZOLL}).
The existence of possible jumps explains why one must write the two-term
asymptotics as the inequalities \eqref{INEQ}  rather than as an asymptotic equality as in the continuous case. An alternative which is available in pure Zoll cases is to prove asymptotics for $P_{I_{\lambda}}^{\tau}$ where the intervals are 
adapted to the eigenvalue clusters of the Zoll manifold (Theorem \ref{Zoll}).  We now illustrate
the jump with an elementary calculation on the circle $S^1$. In Section \ref{ZOLL} we study Zoll examples.

 In the real domain,
the   asymptotics of the spectral projection kernels are 
constant since $S^1$ acts by isometries. There is a single closed geodesic, of periodic $2 \pi$, and 
$$\begin{array}{lll} \Pi_{[0, \lambda]}(x,x) = \sum_{k \in \Z, |k| \leq \lambda}
1 = 2 (
 \lambda  -  \{\lambda \} ) + 1, \end{array} $$
 where, as above,  $\{\lambda\}$ is the fractional
 part of $\lambda$; it is periodic with  jump discontinuities at
$\lambda \in \N$.  

Now consider  the  long interval damped
spectral projections \eqref{TCXSPM}.    We fix $\zeta \in
S^1_{\C} = \C/ \Z$ and assume that $\Im \zeta  =
 \tau 
> 0$.   Then,
\begin{equation} \label{S1} \begin{array}{lll} P^{\tau}_{[0, \lambda]}(\zeta, \bar{\zeta}) & = &  \sum_{k: |k| \leq \lambda} e^{
 - 2 \tau|k|} e^{2   k  \Im \zeta } \\&&\\ &=&  1+ \sum_{k: 0 < k \leq \lambda} 1 +  \sum_{k <0: |k| \leq \lambda} e^{
 - 4 \tau|k|} \\ &&\\
 & = &  \sum_{0 < k  \leq \lambda} 1 + C(\tau) +o(1) , \;\;\; C(\tau
 ) := 1+  \sum_{k=1}^{\infty} e^{- 4 \tau|k|} 
\\ & &\\ &=&
 \lambda - \{\lambda\} + 
C(\tau) +  o(1), \end{array} \end{equation}  
   We note that there is a single
geodesic of period $T= 2 \pi$, and  $W_J(S) = I$ has a single eigenvalue with $s_{\ell}=0$. It has a single eigenfunction
$\Omega_{J_{\zeta}}$ which is not orthogonal to the ground state.  Proposition \ref{QzetaPROP} asserts that
 $Q_{\zeta}(\lambda) = \frac{1}{ T} \{ \lambda T -\pi\}_{2 \pi}    = \{ \lambda - \half \}_{2\pi}$. 
One may compare this simple asymptotic with   Theorem \ref{SHORTINTSa}(3), when $m=1$.

The discontinuity of the subprincipal term is reflected by the
nature of the singularities at $t \not= 0$ of the ``complex wave group'' (see \eqref{CXWVGP}  below),
\begin{equation} \label{UTTAUINTRO} U_{\C}(t + 2 i \tau, \zeta, \bar{\zeta}) = \sum_{k \in \Z} e^{(i t
 - 2 \tau)|k|} e^{2 \langle  k \Im \zeta  \rangle} = \sum_{k \geq 0} e^{i t k} + R(t, \zeta, \bar{\zeta}),
\end{equation}
 where $R(t)$ is analytic in $t$. By the Poisson summation
 formula, the singularities correspond to closed geodesics of
 $S^1$. 
 
\subsection{Zoll manifolds}
 
 One higher dimensional generalization of a circle is a Zoll manifold, all of whose geodesics are closed. 
As reviewed in \S \ref{APPENDIX}, 
the spectrum  of $\sqrt{\Delta}$ on a Zoll manifold consists of a union of small eigenvalue
clusters surrounding a certain  arithmetic progression $\{(k + \frac{\beta}{4})\}_{k = 0}^{\infty}$, where $\beta$ is the 
common Morse index of the periodic orbits. If the interval $I_k$ is chosen to cover the kth  cluster and not
to intersect any other cluster, then the cluster projection $\Pi_{I_k}$ is a Riesz projector  (i.e. a contour integral
of the resolvent) and its analytic continuation $P^{\tau}_{I_k}$ has a complete asymptotic expansion.
\begin{theo} \label{Zoll} Suppose that $(M, g)$ is a real analytic Zoll metric. Then,
if $\sqrt{\rho}(\zeta) \geq \frac{C}{k}$, there exist geometic coefficients $R_j(\zeta)$ such that $ P^{\tau}_{I_k} (\zeta, \bar{\zeta}) $ admits a complete asymptotic expansion as $k \to \infty$ of the form,
$$\begin{array}{l} P^{\tau}_{I_k} (\zeta, \bar{\zeta})  =
\left(\frac{(k + \frac{\beta}{4})}{\sqrt{\rho}(\zeta)} \right)^{
\frac{m - 1}{2}} \left(1 + \sum_{j = 1}^{\infty} R_j(\zeta)  (k + \frac{\beta}{4}) ))^{-j} \right)   \end{array}
  $$

\end{theo}
Here, we have avoided the inequalities in  Theorem \ref{SHORTINTSa} (3) by choosing to locate $\lambda$ at the
center of each eigenvalue cluster. More precisely, to obtain the Weyl sum we simply sum the above result in $k$.
(The notation $P^{\tau}_{I_{\lambda}}$ is discussed below \eqref{TCXSPM}.)

 Although Zoll manifolds are not a primary focus
in this article, they are the  Riemannian analogues of the unit circle bundles associated to positive line bundles over K\"ahler manifolds,  and as in the K\"ahler case  we obtain complete asymptotic
expansions if  the spectral intervals $I_{\lambda}$ contain  precisely one eigenvalue cluster (see Section \ref{ZOLL}).    
A brief review of  spectral asymptotics in the real domain,
and their dependence on the periodicity of the geodesic flow, is given in \S \ref{APPENDIX}.

As mentioned in Remark \ref{SURFREV}, $Q_{\zeta}$ will also have jumps for convex surfaces of revolution if
$\zeta$ is  closed geodesic between the poles. This is an example where the spectrum is uniformly distributed modulo
$1$, i.e. does not come in clusters.

\subsection{\label{TORUS} Parabolic closed geodesics and flat tori }
As a simple  example to illustrate the terminology,  the setting and the results, we work out the details for product eigenfunctions  on a flat torus  $M =\R^m/ 2 \pi \Z^m$.
The real eigenfunctions are the exponentials $e_k(x) = e^{i
\langle x, k \rangle}$ with $k \in \Z^m$.   Since $M_{\C} = \C^m/ 2 \pi \Z^m = \R^m_x /
2 \pi \Z^m \times \R^m_{\xi}$ we see that $\sqrt{\rho}(\zeta) = |\xi|$, and  for $\zeta = x + i \xi \in \partial M_{\tau}$, their complexifications
are the complex exponentials $e_k^{\C}(x + i \xi) = e^{i \langle x
+ i \xi, k \rangle}, $ and clearly $$| e^{i \langle x
+ i \xi, k \rangle}|^2  = e^{- 2 \langle \xi, k \rangle}. $$
Hence,
$$ P^{\tau}_{[0, \lambda] }(\zeta, \bar{\zeta})  = \sum_{k \in \Z^m: |k| \leq \lambda} e^{- 2 |\xi| |k|}
 e^{- 2 \langle \xi, k \rangle}, \;\;\; (\zeta = x + i \xi). $$

The point  $\zeta \in \partial M_{\tau}$ is a periodic point if and only if $ \Im \zeta =  \frac{k}{|k|} \tau $ for some lattice point $k
\in \Z^m$. Then $T(\zeta) = |k|$, and
\begin{equation}\label{FLAT}  Q_{\zeta}(\lambda) \sim    \sum_{n =1}^{\infty}  \frac{\sin  ( n \lambda |k|)}{n  |k|} \left( \frac{\lambda}{|nk| + 2 i  \tau}  \right)^{ \frac{m-1}{2}}, \;\;   \end{equation}
  For $m \geq 2$, the
function $Q_{\zeta}(\lambda)$ is a bounded uniformly continuous function. If $\sqrt{\zeta}= \tau$,
 $$ P^{\tau}_{[0, \lambda] }(\zeta, \bar{\zeta}) \sim C_m \lambda^{\frac{m + 1}{2}} \left(1 +  Q_{\zeta}(\lambda) \lambda^{-1} +  o(\lambda^{-1})  \right);.$$
 
To check Lemma \ref{L2LEMintro}, we note that
 the $L^2$ norm-square  of  $e_k^{\C}(x + i \xi) $ over the Grauert tube $\partial M_{\tau}$, or equivalently
  over the co-sphere bundle $S^*_{\tau} \R^m/\Z^m$ is, by the usual steepest descent  calculation of asymptotics of Bessel functions,

 \begin{equation} \label{L2TORUS} \begin{array}{lll} ||e_k^{\C}(x + i \xi) ||^2_{L^2(\partial M_{\tau})} & =& \int_{S_{\tau}^{m-1}} e^{- 2 \langle \xi, k \rangle}
 d\mu_{\tau}(\xi) \\ &&\\
 & = & \tau^{m-1} \int_{S^{m-1}} e^{- 2 |k| \tau  \langle \xi, \frac{k}{|k|}  \rangle} d\mu(\xi) 
 \\ &&\\
 & = & \tau^{m-1} \int_{S^{m-1}} e^{- 2 |k| \tau \langle \xi, e_1 \rangle} d\mu(\xi) \\ &&\\
 & \sim & C_m e^{2\tau  |k| } \tau^{m-1} \left(\det (2 |k| \tau I_{m-1} \right)^{-\half} = C_m |k| ^{-\frac{m-1}{2}}  e^{2\tau  |k| } \tau^{\frac{m-1}{2}}
 , \end{array} \end{equation}

The Husimi functions,  $$\frac{| e^{i \langle x
+ i \xi, k \rangle}|^2}{ ||e_k^{\C}(x + i \xi) ||^2_{L^2(\partial M_{\tau})} }   = C_m \tau^{-\frac{m-1}{2}} |k| ^{\frac{m-1}{2}}  \frac{e^{- 2 \langle \xi, k \rangle}}{ e^{2\tau  |k| }}, $$
attain their maximum when $\xi = - \tau \frac{\vec k}{|\vec k|} $ and at that point take the value $C_m \tau^{-\frac{m-1}{2}} |k| ^{\frac{m-1}{2}} .$

\section{\label{GRAUERTSECT} Grauert tubes: Geometry and Analysis}

In this section, we review  geometry and analysis on   Grauert tubes. The relevant geometry and analysis have already
been presented in the prior articles \cite{ZPSH1} and \cite{ ZJDG}. To avoid duplication, we only briefly introduce the basic
objects and results and refer to these articles for further background.

  \subsection{\label{GRB} Grauert tubes, the Hamilton flow of $\sqrt{\rho}$ and the complex geodesic flow}

  In this section, we briefly review the basic objects regarding
  Grauert tubes and establish notation.  There is considerable overlap
with the exposition in  \cite{ZPSH1} and we refer there for many details.

The Grauert tube function \eqref{RHOFORM} induces a distinguished 1-form,
%(known as a pseudo-Hermitian
%structure \cite{Lee}), 
\begin{equation} \label{alpha} \alpha = \frac{1}{i} \partial \rho|_{\partial M_{\tau} }\end{equation}
on $\partial M_{\tau}$. It induces the K\"ahler form
\begin{equation} \label{rho} \omega =  d \alpha = i \ddbar \rho \end{equation}
with $\rho$ as K\"ahler potential, and also the volume form
\begin{equation}  \label{CONTACTVOL} d\mu_{\epsilon} : = \alpha \wedge \omega^{m-1}, \;\; (m = \dim M). \end{equation}

For generic analytic Riemannian metrics $g$,  there is a finite maximal
radius $\tau_{\max}$ of the Grauert tubes, which is finite for all but a few real analytic Riemannian metrics. 
The eigenfunctions are known to extend holomorphically to the maximal open
Grauert tube but do not extend further. For instance, the maximal Grauert tube radius
for hyperbolic space of constant curvature $-1$ is $\frac{\pi}{2}$.  As a result, the Grauer tube Weyl laws are
only valid for $\tau < \tau_{\max}$ and should blow up when $\tau = \tau_{\max}$.
We refer to \cite{LS1,Szo} and to \S \ref{GRB}
for background.

\subsection{\label{CDSECT} The diastasis}  As with any real analytic \kahler potential, we may consider $\rho(z)$ as a function of $(z, \bar{z})$ and analytically
extend it to a function $\rho(z,w) $ on the complexification
$M_{\tau} \times \overline{M_{\tau}}$ of $M$,  as a function holomorphic in $z$ and $\bar{w}$. Thus, the  analytic extension
of $\rho(z)$ has the form $\rho(z, w) = f(z, \bar{w})$ where $f$ is holomorphic in both variables and  satisfies
$f(z, w) = f(w, z), \overline{f(z,w) } = f(\bar{z}, \bar{w}). $

 The defining function of the Grauert tube $M_{\tau}$ is $\sqrt{\rho} -2  \tau$, or equivalenty  $(\rho(z,z) - 4 \tau^2)$.  Thus, $\rho < 4 \epsilon^2$ in $M_{\epsilon}$,  $d \rho \not= 0$ when $\rho = 4 \epsilon^2$,
and the Levi matrix $\begin{pmatrix} \partial^2 \rho/\partial z_j \partial \bar{z}_k \end{pmatrix}$ is positive Hermitian non-degenerate. 
Indeed, $i \ddbar \rho= \omega_{\rho}$ is a K\"ahler form on $M_{\epsilon}$.    

Following the notational conventions of \cite[Proposition 1.1]{BoSj},
\begin{equation} \label{psi}  \psi_{\tau}(\zeta, w) =  \frac{1}{i}(\rho(\zeta,w) - 4 \tau^2) ,\;\;\; (\zeta, w \in \partial M_{\tau}).\end{equation} 
%When restricted to $\partial M_{\tau} \times \partial M_{\tau}$. 
%$$ - 4  \psi(y, x) = r_{\C}^2(y, x)  - 4\tau^2 = r_{\C}^2(x,y) - 4 \tau^2,
%=  -  \overline{\psi(x, y)} = - \overline{r_{\C}^2(x, y) } - 4\tau^2
%$$
Thus, $\psi_{\tau}(z,w)$ is holomorphic
in $z$, anti-holomorphic in $w$ and satisfies  $$\psi(z, w) = - \overline{\psi(w,z)}.$$
The imaginary part of
\eqref{psi} is minus   the Calabi diastasis function,
\begin{equation} \label{CD} \begin{array}{lll} D(z,w) : & = &  - \left(\rho(z,w) + \rho(w,z) - \rho(z,z) - \rho(w,w) \right) 
% \\ &&\\&&
%= f(z, \bar{z}) + f(w, \bar{w}) - f(z, \bar{w}) - f(w, \bar{z})
.  \end{array}\end{equation}

Near the diagonal, \eqref{psi} admits the Taylor expansion, \begin{equation} \label{TAYLOR} \frac{1}{i} (\psi(x,y) + \psi(y,x) - \psi(x,x)  - \psi(y,y)) = L_{\rho}(x - y) + O(|x-y|)^3. \end{equation}
where $L_{\rho}$ is the Levi form (or K\"ahler form), and so as a function on $\Omega \times \Omega$,
\begin{equation} \label{IM} \Im \psi(z,w) \geq C \left( d(z, \partial \Omega) + d(w, \partial \Omega) + |z - w|^2 \right) + O(|z -w|^3).  \end{equation}

\begin{rem}The  notational convention of \cite{BoSj}
of putting $i^{-1}$ in front of $\rho - 4 \tau^2$ disagrees with that of Phong-Stein \cite{PhSt1} (page 96), who use the notation $\psi$
rather than $i \psi$ for the complex phase.
\end{rem}

\subsection{\label{KNCSECT}\kahler normal coordinates}

Since $(M_{\tau}, i \ddbar \rho)$ is a \kahler manifold, we may define  \kahler normal coordinates around any  point $p \in M_{\tau}$.  Local holomorphic coordinates $(z_1,\dots,z_m)$ in a
	neighborhood $U$ of $p$ are called \kahler normal coordinates centered at $\zeta \in \partial M_{\tau}$ with $z(\rho) = 0$ if  		 the \kahler potential $\rho$ takes the form,
	\begin{equation} \label{phiK} \rho(z) = 4 \tau^2 +  |z|^2 + \sum_{ J K} a_{ J K} z^J \bar{z}^K, \;\;{\rm with}\; |J| \geq 2, |K| \geq 2, \end{equation}
so that the \kahler form is locally given by,
\begin{equation} \label{omega} \omega = \omega_0 + \sum_{i j k \ell} R_{i \bar{j} k \bar{\ell}} z_i \bar{z}_j dz_k \wedge d \bar{z}_{\ell} + \cdots,
\;\; \omega_0 = \sum_j dz_j \wedge d \overline{z_j}. 
\end{equation}
where $R_{i \bar{j} k \bar{\ell}} $ is the curvature. 

\begin{lem} \label{CDKNC} 
In \kahler normal coordinates centered at $\zeta \in \partial M_{\tau}$, so that $z =0$ denotes the point $\zeta$ the diastasis has the form,
$$D(\zeta, z) =  |z |^2 + O(|z|^3).$$
%where $O^k$ denotes a quantity which vanishes to order $k$ at $z =  0 (= \zeta)$.
\end{lem}

\begin{proof} Set $w = \zeta$ in \eqref{CD} with $\rho(\zeta, \zeta) = 4 \tau^2$, to get $\rho(z,w) = 4 \tau^2 + z \bar{w} + O^3 $. Then,
$D(z,w) = - (z \bar{w} + w \bar{z}   - |z|^2 - |w|^2 +  O^3 ) = |z-w|^2 + O^3.$

\end{proof}

%We may adapt the \kahler coordinates centered at  $z \in \partial M_{\tau}$, so that
%the  \kahler normal  coordinates $(\zeta, \zeta_{n + 1})$ satisfy, for $w \in \partial M_{\tau}$, 
% \begin{equation} \label{PSEXP} \rho(w) = \Im \zeta_{n + 1} - |\zeta|^2 + \epsilon +
% O(|\zeta_{n + 1}| |\zeta| + C\; |\zeta|^2
%+ |\zeta|^3),  \;\; (w \in \partial M_{\tau}) \end{equation}
%as in  \cite[Page 93]{PhSt1}. \edit{Explain and correct remainder}
%The coordinates
%$$(\zeta, t, \rho) = (\zeta, \Re \zeta_{n + 1}, \rho(w)) \in \partial M_{\tau} \times \R_+, \; (\zeta = \zeta(z,w), \;\; t = t(z,w)), $$
%are referred to as  {\it standard Heisenberg coordinates} for $w$ near $z$. We also
%write,
%$$\Theta(z,w) = (\zeta, t), \;\; \zeta = pr_{\C^n} \Theta(z,w), \;\; t = pr_{\R} \Theta(z,w). $$
%\edit{Define $\C^n$}
%They satisfy
%$$\zeta(z,w) = - \zeta(w,z) + \bar{O}^2, \;\; t(z,w) = - t(w,z) + O^3. $$

\subsection{\label{CRSECT} CR geometry of $\partial M_{\tau}$}

Let us denote by $J$ the adapted complex structure on $M_{\tau}$ arising
from the complexification of $M$. 

As a real hypersurface of the complex manifold $M_{\tau}$, $\partial M_{\tau}$ 
has a 
  {\it  CR structure\/}, i.e.
 a real  $J$-invariant horizontal symplectic hyperplane bundle defined by \begin{equation}
\label{H} H = \ker \alpha \subset
T \partial M_{\tau}, \;\; J: H \to H, \;\; H = J T\partial M_{\tau} \cap
T \partial M_{\tau}.   \end{equation} Then
\begin{equation} \label{SP} T \partial M_{\tau} = H \oplus \R T \end{equation}  where 
 $T$ is the characteristic vector field satisfying $$\alpha (T) = 1, \;\;  d \alpha(T, \cdot) = \omega_{\rho}(T, \cdot 0) = 0 \; \mbox{on}\; \partial M_{\epsilon}. $$
  $T = \Xi_{\sqrt{\rho}}$ is  the just the Hamilton vector field of $\sqrt{\rho}$ on $M_{\tau}$ with
respect to $\omega $ \eqref{rho}.
After complexifying each horizontal space $H_{\zeta} \otimes \C=  H_{\zeta}^{1,0} \oplus H_{\zeta}^{0,1}$,  we have the decomposition
   \begin{equation} \label{SPLITc} T_{\C} \partial M_{\tau} = H^{1,0} \oplus H^{0,1} \oplus \C T, \end{equation} We  define the  boundary Cauchy-Riemann operator  $\bar{\partial}_b$
operator by  $\bar{\partial}_b = df|_{H^{1,0}}$.

\subsection{Geodesic and Hamiltonian flows}

We denote by  $\Xi_H$
the Hamiltonian vector field of a Hamiltonian $H$ and its flow by
$\exp t \Xi_H$.  Given the symplectic form \eqref{rho}, we define the Hamiltonian flow 
\begin{equation} \label{gtdef} 
g^t = \exp t \Xi_{\sqrt{\rho}}, \;\; \rm{on}\;\; M_{\tau}. \end{equation}
We  denote the restriction of the Hamilton flow on the right side to the energy
surface $\partial M_{\tau}$ by,
\begin{equation} \label{gtau} g_{\tau}^t: \partial M_{\tau} \to \partial M_{\tau}.\end{equation}
We denote by $G^t = \exp t \Xi_{|\xi|}$ the homogeneous Hamiltonian flow on $T^*M \backslash \{0\}$. 
We also define the 
exponential map  $\exp_x: T^*M \to M$ of $g$;
as usual in geometry, the exponential map is defined by the Hamiltonian flow of $|\xi|^2$ rather than $|\xi$. 

The analytic continuation of the exponential map to imaginary time defines a diffeomorphism \eqref{EXP}
%\begin{equation}\label{E} E: B^*_{\tau} \to M_{\tau}, \;\; E(x, \xi) = \exp_x i \xi, \end{equation}
satisfying  $E^* \sqrt{\rho} = |\xi|$
\cite{GS1,LS1}. It follows that $E^*$ conjugates the geodesic flow
on $B^*M$ to the Hamiltonian flow \eqref{gtdef}, i.e.
%$\exp t \Xi_{\sqrt{\rho}}$ of
%$\sqrt{\rho}$ with respect to $\omega$, i.e.
\begin{equation} \label{gt} E(G^t(x, \xi)) = g^t(\exp_x i \xi).  \end{equation}
We often restrict \eqref{EXP} to the unit cosphere bundle $S^*_{\tau} M$
of radius $\tau$ and we then \eqref{EXP}- \eqref{gtau}  become, 
\begin{equation} \label{Etau} E_{\tau}: S^*_{\tau} \to \partial M_{\tau}, \;\;\; g^t_{\tau} : = E_{\tau} G^t E_{\tau}^{-1}. \end{equation} 

\subsection{\label{PSSECT} Phong-Stein leaves as symplectic transversals }

In dealing with the geodesic flow on $\partial M_{\tau}$ it is very useful to introduce symplectic transversals to the flow, and to define
the time coordinate by the flow-time to a transversal. A very nice set of transversals (or, leaves) were introduced in 
 \cite{PhSt1}  (see also  \cite[Example 4]{PhSt2}  on a general strictly pseudo-convex domain $D$. In the Grauert tube setting, they
 are defined by,
\begin{equation} \label{LEAF} \mcal_{\zeta} =
% \{w \in \partial M_{\tau}: \; \Im \; i \psi(\zeta,w) = 0\} =
 \{w \in \partial M_{\tau}:  \Re \psi(\zeta, w) = 0
%\rho(\zeta,w) - 4 \tau^2 = 0
\},\end{equation}
where $\psi$ is defined in \eqref{psi}.
The Phong-Sturm leaves $\mcal_{\zeta}$ are not complex submanifolds of $M_{\tau}$ in general, but at the point $z$ they are tangent to $H_{\zeta}^{*(1,0)} M_{\tau}$. We summarize the result  in \cite{PhSt1}:

\begin{lem} $T_{\zeta} \mcal(\zeta) = H^{1,0}_{\zeta} $, hence  $\mcal_{\zeta} $ is  transversal to $T = \Xi_{\sqrt{\rho}}$ at $\zeta$, hence is a local symplectic    to the flow $g^t_{\tau}$ (by \eqref{SPLITc}).   %If  $\zeta$ is a periodic point of period $T(\zeta)$ for $g^t_{\tau}$.  
 \end{lem}

\subsection{\label{HEISCOORDSECT} Heisenberg coordinates}

We will be linearizing phases of integrals by taking Taylor expansions in special coordinate systems, generalizing the \kahler normal coordinates in
the line bundle setting and the Heisenberg coordinates in \cite{ShZ02} on boundaries of unit co-disc bundles. Since $\mcal_{\zeta}$ is not a complex
hypersurface in $M_{\tau}$ it is afortiori not a \kahler manifold and it does not make sense to introduce \kahler normal coordinates on it.  However,
there exist \kahler normal coordinates on $M_{\tau}$ (see Section \ref{KNCSECT}) and one may define  `Heisenberg  coordinates'  as follows:

\begin{defn}\label{SODEF} 
We  define `slice-orbit'   coordinates on $\partial M_{\tau}$ near $\zeta$  as
the (locally defined) inverse of the slice-orbit parametrization,
\begin{equation} \label{SO} (w', t) \in \mcal_{\zeta} \times \R \to g^t_{\tau}(w'). \end{equation}
By slice-orbit \kahler\; normal coordinates we mean slice orbit coordinates with  $w'$ defined by the restriction of \kahler normal coordinates centered at $\zeta$
on $\mcal_{\zeta}$.

\end{defn}

\subsection{Return times and return maps}
The map \eqref{SO}  is well defined for all $t$ but is not one to one, and has many
local inverses. 
The inverses will be important below and are best described in terms of the return times and  (non-linear)  local Poincar\'e first return map for the transversal 
near $\zeta$. For points $z \in \mcal_{\zeta}$ near $\zeta$, the orbit $g^t_{\tau}(z)$
will return to $\mcal_{\zeta}$ at some minimal time $T(z)$ near $T(\zeta)$. 

\begin{defn} \label{POINC}
 Denote by 
$$\Phi_{\zeta}(z) = g^{T(z)}(z) : \mcal_{\zeta}\to \mcal_{\zeta}$$  first return map  to $\mcal_{\zeta}$. 
We further define the $n$th return time $T_n$ so that $T_n(\zeta) = n T(\zeta)$. The $n$th
return map $\Phi^{(n)}(z)$ is defined to equal $g_{\tau}^{nT(z)}(z)$.
\end{defn}

When $\zeta$ is a periodic point, we obviously have
$$(s, \zeta) \simeq (s + n T(\zeta), \zeta) $$
in the sense that the two sides get taken to the same point under the slice-orbit parametrization.

The nth return map is only well-defined for $z$ sufficiently close to $\zeta$, but this
is sufficient for the proof of the pointwise Weyl laws. 
We then consider $z \in \mcal_{\zeta}$ near $\zeta$ and use slice orbit coordinates for the $n$th return
$\Phi_{\zeta}^n(z) \in \mcal_{\zeta}$.  For $t$ near $T_n(\zeta)$ and $w$ sufficiently close to $\zeta$,  we have the following equivalence relation on slice-orbit coordinates:
\begin{equation} \label{EQUIV} (t, z) \simeq  (n T(\zeta)  + (t - n T(\zeta)), \Phi_{\zeta}^n(z))
\end{equation}

\subsection{\label{COMPARISON} Comparison to the line bundle setting}
In this section, we compare the geometry of the Grauert tube setting to that of the line bundle setting. 
Grauert tubes  are the Riemannian analogue of (co-) disc bundles $D^* \subset L^*$  of positive  Hermitian line bundles
$L \to M$ over \kahler manifolds.  
The line bundle setting is that of a positive Hermitian holomorphic line bundle $(L, h) \to (M, \omega)$ over a \kahler manifold, where $i \ddbar \log h = \omega$.  Let $L^*$ be the dual line bundle,  let $h^*: L^* \to \R$ be the dual Hermitian metric and let $X_h= \{h^* = 1\}$ be the boundary
of the unit co-disk bundle. It is a strictly pseudo-convex  CR hypersurface in the complex manifold $L^*$ and is the analogue of $\partial M_{\tau}$. 
The Reeb vector field $\frac{\partial}{\partial \theta}$ is the analogue of the Hamilton vector field $\Xi_{\sqrt{\rho}}$ on $\partial M_{\tau}$, but
has two significantly simpler properties: First, it generates an $S^1$ action ($S^1 = \R/\Z$), which in the Riemannian case is only true on a Zoll
manifold; second, and more important, it generates a holomorphic action which complexifies to a $\C^*$ action. In the Riemannian setting,
the geodesic flow is never holomorphic, and (except for Zoll manifolds),  the orbits of the geodesic flow almost never form a fiber bundle over a quotient space and there is no analogue of $M$.

 %However, there are also many significant differences. The  main one is that the Reeb characteristic flow is the geodesic flow in the Grauert setting, which is usually not %periodic, 
%whereas it is the circle $S^1$ action on the fibers of the line bundle $L \to M$ in the line bundle setting. 
 There are also significant differences in the behavior
of \szego \; kernels and the linearization of other relevant kernels to osculating Heisenberg spaces. We make substantial use of the Phong-Stein foliation and
their analysis of adapted coordinates in \cite{PhSt1,PhSt2} in the linearization procedure.    Moreover, this article is fundamentally about analytic
continuations of eigenfunctions from $M$ to the Grauert tube, while in \cite{ZZ18} the results pertain to holomorphic sections of powers of $L$ that have no underlying
real structure. But we are able to refer to  \cite{ZZ18} for many of the details on  local Bargmann-Fock-Heisenberg approximations and linearizations of Fourier
integral Toeplitz operators. We refer to Section \ref{SZEGOCOMP} for further comparisons to the line bundle case.
The pointwise Weyl law in the line bundle setting was proved by the author and P. Zhou in  \cite{ZZ18}, and there are many over-laps
in this article and \cite{ZZ18}. 

\section{\label{BFHSECT} Linear and Heisenberg models}

In this section, we review the linear Bargmann-Fock models. As in \cite{ZZ18} and elsewhere, we often reduce calculations in the nonlinear setting to osculating Bargmann-Fock and Heisenberg models on the tangent spaces. 
Much of the exposition below repeats that in \cite{ZZ18}; proofs are generally omitted unless they are short, and the reader is referred to
\cite{ZZ18} for further details.
		
\subsection{\label{HEISSECT} Heisenberg group}
The space $\C^m \times S^1$ can be identified with the reduced Heisenberg group $\H^m_{red}$, where the group multiplication is given by 
\[ (z, \theta) \circ (z', \theta') = (z+z', \theta+\theta' + \Im( z \bar z')). \] We repeat some background from \cite{ZZ18}.

The generators of the Heisenberg group action are   contact vector fields on the Heisenberg group generated by a linear Hamiltonian function $H: \C^m \to \R$.
For any $\beta \in \C^m$, we define a linear Hamiltonian function on $\C^m$ by
$ H(z) = z \bar \beta + \beta \bar z. $
The Hamiltonian vector field on $\C^m$ is  
$ \Xi_H = - i \beta \pa_z + i \bar \beta \pa_{\bar z}, $
and it lifts to a contact vector field on $\H^m_{red}$,
\[ \h \Xi_H = - i \beta \pa_z + i \bar \beta \pa_{\bar z} -\half( z \bar \beta + \beta \bar z) \pa_\theta, \]
with respect to the  contact form $\alpha = d\theta + \frac{i}{2} \sum_j (z_j d\bar z_j - \bar z_j dz_j)$,
and generates the flow,
\[ \h g^t(z,\theta) = (-i \beta t, 0) \circ (z,\theta) = (z-i\beta t, \theta - t\Re(\beta \bar z)). \]

%These vector fields are contact vector fields in the sense that
%the  contact form $\alpha = d\theta + \frac{i}{2} \sum_j (z_j d\bar z_j - \bar z_j dz_j)$ on $\H^m_{red}$ is invariant under the left multiplication
%\[ L_{(z_0, \theta_0)}: (z, \theta) \mapsto (z_0, \theta_0)  \circ (z, \theta) = (z + z_0, \theta + \theta_0 + \frac{ z_0 \bar z-  \bar z_0 z }{2i} ). \]

%\bpf
%\[ (L_{(z_0, \theta_0)}^* \alpha)|_{(z, \theta)} = d (\theta + \theta_0 + \frac{\bar z z_0 - \bar z_0 z}{2i} ) + \frac{i}{2} \sum_j ((z_j + z_{0j}) d \bar z_j - (\bar z_j+\bar %z_{0j}) d z_j) = \alpha|_{(z, \theta)}.\]
%\epf

	\subsection{\label{METASECT} Symplectic Linear Algebra} We refer to \cite{deG,Gutt,Gutt2,Gutt3} for background on symplectic linear algebra and symplectic
	normal forms. 
 
 Let $(V, \sigma)$ be a real symplectic vector space of dimension $2m$ and
let $J $ be a compatible complex structure on $V$. There exists a symplectic basis in which  $V \simeq \R^{2m} $,  $\sigma$ takes the standard form $ \omega = 2 \sum_{j=1}^m dx_j \wedge d y_j$ and   $J$ has the standard form,  $J_0 = \begin{pmatrix}  0 & - I \\ & \\
I & 0 \end{pmatrix}.$
We generally identify $(V, \sigma)$ with the standard real symplectic vector space 
 $\R^{2m}, \omega = 2 \sum_{j=1}^m dx_j \wedge d y_j$.

We denote the symplectic group of $(V, \sigma)$ by $\rm{Sp}(m, \R) $\footnote{It is sometimes denoted $\rm{Sp}(2m, \R)$}, and its Lie algebra
by ${\mathfrak s} {\mathfrak p}(m, \R)$.  The group $\rm{Sp}(m, \R)$ consists of linear transformation $S: \R^{2m} \to \R^{2m}$, such that $S^*\omega = \omega$, and as in \cite{F89}  it may be expressed in block form as,
\begin{equation} \label{SBLOCK}  \bma x' \\ y' \ema = S \bma x \\ y \ema = \bma A & B \\ C & D \ema \bma x \\ y \ema. \end{equation}
The symplectic  Lie algebra consists of skew symplectic matrices; it  may be identified with the Poisson algebra of  quadratic Hamiltonians.

 %We freely use the standard facts about $\rm{Sp}(n, \R)$, for instance its   $\bar{N} D N $ decomposition
% \cite[Proposition 4.9]{F89} and  it
The maximal compact subgroup of $\rm{Sp}(m, \R)$ is the unitary group $K =
U(m) := \rm{Sp}(m, \R) \cap O(2n, \R)$ of $(\R^{2m}, J)$. It is the group of orthogonal matrices $U$ on $\R^{2m}$ satisfying $U J = J U$.
% is its maximal compact subgroup, $D$ is the real diagonal subgroup and $N$ (resp. $\bar{N}$)  are the upper (resp. lower) triangular matrices with $I$ on the %diagonal. 
 One has
%By  \cite[Proposition 22]{deG}.
$$U = \begin{pmatrix} A & -B \\ & \\
B & A \end{pmatrix},\;\; A B^t = B^t A, \;\; A A^t + BB^t = I, \;\; U^{-1} = \begin{pmatrix} A^t &B^t \\ & \\
-B^t & A^t \end{pmatrix} = U^t. $$

A symplectic matrix admits a polar decomposition $S = U \hat{P}_S$ where $\hat{P}_S = (S^*S)^{\half}$ is  a non-negative symplectic matrix  and where $U \in U(m)$; see    \cite[Proposition 4.3]{F89}.

If $S \in Sp(m, \R)$, then its  transpose $S^t = J S^{-1} J^{-1}$ also lies in  $Sp(m, \R)$ and   $S J = J (S^t)^{-1}.$
$S  \in \rm{Sp}(m, \R) $ is 
a {\it symmetric symplectic matrix} if it satisfies $S^t = S$, and then $S J = J S^{-1}$.
We say that $S  \in \rm{Sp}(m, \R) $ is a normal symplectic matrix if $[S, S^t] = 0$.  In the polar decomposition $S = U \hat{P}_S$  of a normal  $S \in \rm{Sp}(m, \R)$, one has $ \hat{P}_SU = U \hat{P}_S$.
A symplectic matrix $S$ is symmetric positive definite if and only if $S = e^X$ with $X \in \mathfrak s \mathfrak p(m)$
and $X = X^t$. See for instance   \cite[Proposition 4.7]{F89}.

\subsection{Semi-simple symplectic matrices} Let $S \in \rm{Sp}(m,\R)$. An element $S \in \rm{Sp}(m,\R)$ is called {\it semi-simple} if $\R^{2m} \otimes \C = \C^{2m}$ is the direct 
sum of the eigenspaces $E_{\lambda}$  of $S$. The eigenvalues of $S$ arise in quadruples $[\lambda] = \{\lambda, \lambda^{-1},
\bar{\lambda}, \bar{\lambda}^{-1}\}$.  

A semi-simple matrix $M$ is one that may be diagonalized over $\C$, i.e. if $M \otimes \C$ is  conjugate in $GL(2m, \C)$ to a diagonal matrix. In general,
it need not be conjugate in $\rm{Sp}(m, \R)$ to a diagonal matrix. If $M \in \rm{Sp}(m, \R)$ is diagonalizable 
over $\R$, then $\R^{2m}$ admits a symplectic basis consisting of eigenvectors of $M$, hence is symplectically diagonalizable; equivalently 
there exists $U \in U(m)$ so that $U^t S U = \Lambda$ is the diagonal matrix $$\Lambda = \rm{diag} (\lambda_1, \dots, \lambda_m; \lambda_1^{-1}, \dots, \lambda_m^{-1}).$$
Indeed  if $e_1, \dots, e_n$ are orthonormal eigenvectors of $S$ corresponding to the eigenvalues $\lambda_1, \dots, \lambda_n$
then since $S J = JS^{-1}$, 
$$S J e_k = J S^{-1} e_k = \frac{1}{\lambda_k} J e_k. $$
Hence $\pm J e_1, \dots, \pm J e_n$ are orthonormal eigenvectors. The basis $\{e_j, J e_k\}$ is symplectic,

Symplectic matrices such as $U \in U(m)$ with complex eigenvalues are not conjugate to their diagonalizations in $\rm{Sp}(m, \R)$.
According to \cite{Gutt, Gutt2,Gutt3}, the  conjugacy classes of  $A \in \rm{Sp}(m, \R)$  are determined by three
types of data: (i) the eigenvalues of $A$, (ii) $\dim (\ker (A - \lambda)^r)$ for $r \geq 1$ for one eigenvalue in each $[\lambda]$; and for
$\lambda = \pm 1$, the rank and signature data of an associated  quadratic form. To avoid excessive technicalities, we only consider semi-simple
symplectic matrices.

 Let $$W_{[\lambda]}  = E_{\lambda} \oplus E_{\lambda^{-1}} \oplus E_{\bar{\lambda}} \oplus E_{\bar{\lambda}^{-1}}. $$
Then $W_{[\lambda]} = V_{[\lambda]} \otimes \C$ where $V_{[\lambda]}$ is a real symplectic subspace, and one has symplectic orthogonal decompositions, 
\begin{equation} \label{WDECOMP} \R^{2m} = \bigoplus_{j=1}^k V_{[\lambda_j]}, \;\; \C^{2m} = \bigoplus_{j=1}^k  W_{[\lambda_j]}, \end{equation}
where $k$ is the number of distinct 4-tuples of eigenvalues. Note that eigenvalues $\pm 1$ are special, since then $E_{\lambda} = E_{\lambda^{-1}}$,
and the 4-tuple collapses to a singleton.

\begin{itemize}
\item We say that $S \in \rm{Sp}(m, \R) $ is a positive definite symmetric symplectic matrix (or, a real  hyperbolic  symplectic matrix)  if it is  conjugate in $Gl(2m, \R)$ to the diagonal matrix  $\Lambda $ with $\lambda_j > 0$. In this case, there  exists $U \in U(m)$ so that $S = U^t \Lambda U. $ \bigskip

\item $S$ is complex hyperbolic if it is semi-simple and its decomposition \eqref{WDECOMP} contains eigenvalue quadruples $\lambda \in \C$ with
$\Re \lambda \not=0$.  \bigskip

\item $S$ is elliptic if all $\lambda$ in \eqref{WDECOMP} satisfy $|\lambda | = 1$, $\lambda \not= \pm 1$.\bigskip

\item $S$ is degenerate elliptic if all $\lambda$ in \eqref{WDECOMP} satisfy $|\lambda | = 1$, and there exists $\lambda$ with  $\lambda = \pm 1$.\bigskip

\end{itemize}

\subsection{\label{CXSTRUCTSECT} Complex structures} Let  $V$ be a real symplectic vector space and let 
  $V_{\C} = V \otimes \C$ be its  complexification. Given a complex structure $J$ on $V$, let $H^{1,0}_J$ resp. $H^{0,1}_J$, denote the $\pm i$ eigenspaces of $J$ in $V \otimes \C$. 
The projections onto these supspaces are denoted
by
\begin{equation} \label{PJ} P_J = \half(I - i J): V \otimes \C \to  H^{1,0}_J, \;\; \bar{P}_J = \half(I + i J): V \otimes \C \to H^{0,1}_J. \end{equation}

In complex coordinates $z_i = x_i + i y_i$, we have then 
\[ \bma z' \\ \bar z' \ema = \bma P & Q \\ \bar Q & \bar P \ema \bma z \\ \bar z \ema =: M\; \bma z \\ \bar z \ema, \]
where (in the block notation of \eqref{SBLOCK}) $M =  \wcal^{-1} S  \wcal$, i.e.
\begin{equation} \label{PQDEF} M:=  \bma P & Q \\ \bar Q & \bar P \ema = \wcal^{-1} \bma A & B \\ C & D \ema \wcal, \quad \wcal = \frac{1}{\sqrt 2} \bma I & I \\ -i I & iI \ema. \end{equation}
This conjugate group  of $M$ is denoted by  $ \rm{Sp}_c(m, \R) \subset GL(2n,\C)$ in \cite{F89}, and one has the identities   \cite[Prop. 4.17]{F89},

 $$ \begin{pmatrix}  P & Q \\ \bar Q & \bar P \end{pmatrix}^{-1} =\begin{pmatrix} P^* & -Q^t \\ -Q^* & P^t \end{pmatrix}.$$
 %= K  M^* K,$$
 %where $K =  \begin{pmatrix} I & 0 \\ 0 & -I \end{pmatrix}.$
%\\
%(2) $ PP^* - QQ^* = I$ and $P Q^t = Q P^t$. \\
%(3) $P^*P - Q^t \bar Q = I$ and $P^t \bar Q = Q^* P$. 
By \eqref{PJ}, the upper left block $P$ is given in terms of the real blocks \eqref{SBLOCK} of $S$ by, \begin{equation} \label{PDEF} P_J S P_J =  P = \half(A+D + i (C-B)). \end{equation}

We will also need to use polar decomposition $S = U \hat{P}_S$ where $U \in U(m)$ and where $\hat{P}_S$ is positive symmetric (cf \cite{F89}). \footnote{The notation $ \hat{P}_S$ for a positive matrix should not be confused with the block $P= P_J S P_J$.  } In the following, we assume
$S$ is a normal symplectic matrix, i.e. that $S$ commutes with $S^*$.
\begin{lem} \label{DetFORM}  Let $S$ be a normal symplectic matrix, and let 
 $S = U  \hat{P}_S $ be the   polar decomposition of $S$. Then,  $\det P_J S P_J =  \det P_J U P_J  \cdot \det P_J  \hat{P}_S P_J$,
 with $|\det P_J U P_J|  =1$.
 \end{lem}
 
 \begin{proof} 
 
  For $U \in U(m)$, $UJ = JU$, so by \eqref{PJ} $U P_J = P_J U$. 
 
Since $S$ is normal, $P_J U = U P_J$, and since $P_J$ is a projection,    $P_J S P_J =  P_J   U \hat{P}_S P_J = (P_J U P_J)  (P_J \hat{P}_S P_J)$.
Hence, it suffices to prove that $\det P_J U P_J = 1$. But if we use \eqref{PQDEF} to conjugate $U \in O(2m) \cap \rm{Sp}(m, \R)$ to $ \rm{Sp}_c(m,\R)$
it is a block diagonal matrix with $Q =0$,  $P = \hat{U}$ unitary on $\C^m$. It follows that $P_J U P_J = \hat{U}$ and $\det P_J U P_J \in S^1$.

\end{proof}

\subsection{\label{HEISMETSECT} Heisenberg and metaplectic representations}

The metaplectic representation is a representation of the double cover of $\rm{Sp}(m,\R) \simeq  Sp_c(m, \R)$ on Bargamann-Fock space. 
 The Bargmann-Fock space of a symplectic
vector space $(V, \sigma)$ with compatible complex structure $J \in \jcal$ is the Hilbert space,
$$\hcal_{J} = \{f  e^{-\half \sigma(v, J v)} \in L^2(V, dL), f\; \mbox{is\; entire \; J-\;holomorphic}\}. $$ 
Here,    \begin{equation} \label{GSJ} \Omega_{J} (v) := e^{-\half \sigma(v, J v)}
\end{equation} is  the `vacuum state' and $d L$ is normalized Lebesgue measure (normalized
so that square of the symplectic Fourier transform is the identity).
The orthogonal projection onto $\hcal_J$ is denoted by $P_J$ in \cite{Dau} but we
denote it by $\Pi_J$ in this article. Its Schwartz kernel relative to $dL(w)$ is denoted
by $\Pi_J(z,w)$. 
\bigskip

%\noindent{\bf Example}
%When $V = \C^n$  we write $v = Z$, $J Z = i Z$,  and $\sigma(Z,W) = \Im \overline{Z} \cdot W$. 
%Then $\Omega_J(Z) = e^{- \half |Z|^2}.$
%\bigskip

The Heisenberg group
acts on Bargmann-Fock space  by phase space translations, 
\begin{equation} \label{W} W(a):  \hcal_{J} \to \hcal_{J} \end{equation} for  $a \in V$, defined by
$$(W(a) \psi)(v) = e^{i \sigma(a, v)} \psi(v - a). $$
% Translations $\Omega_{J_0}^w = W(w) \Omega_{J_0}$ of the ground state are known as
%{\it coherent states} and are given explicitly by
%$$\Omega_J^w(z) = e^{ i \Im z \cdot \bar{w}} e^{- \half \sigma(z - w, J (z-w))} $$As noted in (3.1) of \cite{Dau}, one has
%\begin{equation} \Pi_J \psi(z) = \langle \Omega_J^z, \psi \rangle = \int_{\C^n} \psi(v)
%\overline{\Omega_J^z}(v) dv,\end{equation} i.e.
%\begin{equation}\label{PIBF} \Pi_J(z,w) = \overline{\Omega_J^z}(w) = e^{ i \Im z \cdot \bar{w}} e^{- \half \sigma(z - w, J (z-w))}  =  e^{i \Im z \bar{w}} e^{- \half (|z %- w|^2)}. =  e^{z \bar{w}} e^{- \half (|z |^2 + |w|^2)}. \end{equation}

The (double cover) $Mp(m,\R)$ of  $Sp(m,\R)$ acts on the Bargmann-Fock  space $\hcal$  by unitary integral operators.  
Following \cite{Dau},  we denote by $W_J(S)$ the 
unitary operator associated to $S \in Mp(m,\R)$ but for simplicity we view $S$ as an element of $Sp(m, \R)$.   Let
$U_S$ be the unitary translation operator  on $L^2(\R^{2n}, d L)$ defined  by $U_S F(x, \xi): = F(S^{-1}(x, \xi))$. The metaplectic representation of $S$ on $\hcal_J$
is given by (\cite{Dau},(5.5) and (6.3 b)) 
\begin{equation} \label{eta}W_J(S) = \eta_{J,S} \Pi_J U_S \Pi_J,
%=  \beta_{J, SJS^{-1}} \Pi_J \circ U_S \Pi_J
\end{equation} where \begin{equation} \label{ETAJS} \begin{array}{lll}
% \beta_{J, SJS^{-1}}  &= & 2^{-n/2} [\det (SJ + JS) ]^{1/4}
%\\&&\\
\eta_{J,S} & = & 2^{-n}  \det (I - i J) + S (I + i J)^{\half}.
 \end{array} \end{equation} 
 % and where $\Pi_J$ is the Bargmann-Fock
%Szeg\"o projector.
Given $M$ as in 
%$M= \bma P & Q \\ \bar Q & \bar P \ema \in Sp_c$, 
\eqref{PQDEF},  the Schwartz kernel of $W_J(M)$ is given by
\begin{equation} \label{METASCHWARTZ}   \kcal_{M}(z,  w) = (\frac{1}{2 \pi})^m (\det P)^{-1/2} \exp \left\{ \half \left( z \bar{Q} P^{-1} z + 2 \bar{w} {P}^{-1} z
- \bar{w} P^{-1} Q \bar w \right)  \right\}, \end{equation}
where the ambiguity of the sign the square root $(\det P)^{-1/2}$ is determined by the lift to the double cover.
%% When $ \acal=Id$, then $\kcal_{k, \acal}(z, \bar w) = \Pi_k(z, \bar w)$. Similarly, we have the kernel upstairs on $X$ as 
%\be \h \kcal_{k, \acal}(\h z, \h w) = \kcal_{k,M}(z, \bar w) e^{k(i\theta_z -|z|^2/2) + k(-i\theta_w - |w|^2/2)}. \label{hatK}\ee
The following Proposition  ties together
the dynamical Toeplitz formula  \eqref{eta} for $W_J(S)$ with the kernel formula \eqref{METASCHWARTZ}.

\bpp \label{toep-met}
Let $ M$ be a linear symplectic map \eqref{PQDEF}. Then, 
%$M$ lifts to a contact transformation $\hat{M}$  of $\H^m_{red}$, and  
 the Schwartz kernel \eqref{METASCHWARTZ} of its metaplectic quantization 
may be expressed in terms of the \szego\; projector $\Pi_J$ onto $\hcal_J$ by,
%$\acal =  \bma P & Q \\ \bar Q & \bar P \ema$, and let $\h  \acal: X \to X$ be the contact lift that fixes the fiber over $0$, then 
\[   \kcal_{M} ( z,  w) = (\det P^*)^{1/2} \int_{\C^m}  \Pi( z,   M  u) \h \Pi( u,  w) d u ) \]
\epp
We only use this formula as motivation for the dynamical Toeplitz quantization of geodesic flows on Grauert tubes. Proposition
\ref{toep-met}  is extensively
discussed in   \cite{Dau, Z97, ZZ18}, and we refer there for further discussion.

\subsection{\label{DETS} Matrix elements and determinants} In this section, we review determinant formulae from \cite{Dau} which relate the determinant $\det P$ in \eqref{METASCHWARTZ}
and inn Proposition \ref{toep-met} with $\eta_{J,S}$ in \eqref{ETAJS} and \eqref{eta}. The same determinant arises in the
 principal symbol in Proposition \ref{LINEAR} and in Theorem \ref{SCALINGTHEO}, and  is the origin
of  $\gcal_n(\zeta)$ \eqref{ABCDintro} - \eqref{ABCD}  in the  $Q_{\zeta}(\lambda)$ function in Definition \ref{QDEF}  Theorem \ref{SHORTINTSa}.

The symplectic form $\omega$ induces a notion of determinant of 
a linear transformation $T$ by $\det_{\omega} T: = \frac{T^*(\omega^n)  }{\omega^n}. $
A choice of symplectic basis identifies $(V, \omega)$ with $(\R^{2n}, \omega_0)$
(the standard symplectic form), and then $\det T$ is the standard determinant. 
Given $S \in Sp(2n, \R)$, the polar decomposition of $S$ has the form
$S = P Q$ where $P = (S^* S)^{\half}$ is the polar part  and $Q  = P^{-1} S$ is the orthogonal
part.

Given $J \in \jcal$ and $S \in Sp(V, \omega)$ we define (see 
  \cite{Dau} (6.1) and (6.3a)), 
\begin{equation} \label{BETAJS} \begin{array}{lll}
 \beta_{J, SJS^{-1}}  &= & 2^{-n/2} [\det (SJ + JS) ]^{1/4}
%\\&&\\
%\eta_{J,S} & = & 2^{-n}  \det (I - i J) + S (I + i J)^{\half}

 \end{array} \end{equation}
The determinants \eqref{ETAJS} and \eqref{BETAJS}   are related by,$|\eta_{J,S}| = \beta_{J, SJS^{-1}}$. In fact (see \cite{Dau}, above (6.3a), and (B6))
$$|2^{-n}  \det (I - i J) + S (I + i J)^{\half}| = [\det (SJ + JS) ]^{1/2} = 2^n \beta^2_{J, SJS^{-1}}. $$
%A useful Gaussian integral formula (\cite{Dau}, (A3')) is
%\begin{equation} \label{GAUSS} \int_{\C^n}  e^{- \sigma(W, B W)/2} e^{- i \sigma (W, C W)/2}
%dL(W) = 2^n \left( \det(B + i C)\right)^{-\half}. \end{equation}

We further record the identities,
\begin{equation}\label{ID2}
\det (SJ + JS) = \det (I + J^{-1} S^{-1} J S) = \det (I + S^* S).
\end{equation}

%A key identity is

%\begin{lem} \label{KEYID2} (see \cite{Dau}, p. 1388)

%\begin{equation} \begin{array}{lll} \eta_{JS} \beta^{-2}_{J, SJS^{-1} } 
%& = & (\eta_{J S})^{*-1} = \eta_{JS}\; 2^n  (\det (I + S^* S))^{-\half} 
%, \\&&\\\eta_{JS}\; 2^n  (\det (I + S^* S))^{-\half}
% & = & \eta_{JS}^{*-1}
.% \end{array}\end{equation}
%\end{lem}

%\begin{proof} The first equality is proved on p. 1388 of \cite{Dau}. The second asserts that
%\begin{equation} \label{BETA} \beta_{J, SJS^{-1} } = 2^{-n/2} (\det (I + S^* S))^{\frac{1}{4}}, \end{equation}
%which follows from \eqref{ETAJS} and identity (ii) above.

%\end{proof}

If we express $S$   as a block  matrix \eqref{SBLOCK},
%$$S = \begin{pmatrix} A & B \\ & \\ C & D \end{pmatrix},\;\; S^{-1}   = \begin{pmatrix} D^* &  - B^* \\ & \\ - C^* & A^* \end{pmatrix} $$
%in a $J$-symplectic basis, 
 then (cf. \cite{Dau}, p. 1388,
\begin{equation}\label{DAUB}  \begin{array}{l}(\eta^*_{J, S})^{-1} = \det ((I + i J) + S (I - i J)) = 2^n \det(A + D + i (B - C)).
  \end{array}  \end{equation} 
  This is the origin of the determinant \eqref{ABCD}.

%\subsection{\label{TQ} Toeplitz quantization of linear symplectic maps}

The following explains the relations between the determinants \eqref{ABCD}  and the matrix elements in the ground state in \eqref{ABCDintro}. 
\begin{lem}\label{DAULEM} ( \cite{Dau}  p. 1388 (Above Appendix C); see \eqref{ETAJS})  Let $(V, \omega) $ be a real symplectic vector space,
 and let $\det = \det_{\omega}$. Let $S \in Sp(V, \omega)$ be as in \eqref{SBLOCK},  let $W_J(S)$ be as in \eqref{eta}, and  let $\Omega_J$ be as in \eqref{GSJ}. Then,
%$$W_J(S) = 2^{-n} (\det (I - i J) + S (I + i J)]^{\half} \int db |\Omega_J^{S b}) (\Omega_J^b | . $$
$$\langle \Omega_J, W_J(S) \Omega_J) =  2^{n/2} \det(A + D + i (B - C))^{-\half}. $$
\end{lem}

Although the proof is well-known (and is given in \cite{ZZ18}) we include it here for the reader's convenience.

\begin{proof}

The following identities are proved on p. 1388 of \cite{Dau}
\begin{equation} \label{KEYID} \begin{array}{lll} \langle \Omega_J, W_J(S) \Omega_J
\rangle & = & \eta_{JS} \langle \Omega_J,  \Omega_{SJS^{-1}} 
\rangle \\ && \\
& = & \eta_{JS} \beta^{-2}_{J, SJS^{-1} }= (\eta_{J S})^{*-1} \\ & &\\
& = & 2^n ( \det \left( I + i J + S(I - i J) \right))^{-\half}\\&&\\
& = &2^{n/2} (\det (A + D + i B - i C)^{-\half}. \end{array} \end{equation}

\end{proof}

\section{\label{LINSECT} Osculating Bargmann-Fock-Heisenberg space and Heisenberg coordinates  }

We now tie  together the linear theory of \S \ref{BFHSECT} with the nonlinear CR
setting of Grauert tubes \S \ref{CRSECT} by defining the osculating
Bargmann-Fock space at a point $\zeta \in \partial M_{\tau}$. As mentioned
in the introduction and in \S \ref{GRB}, the maximal radius  $\tau_{\max}$ of the Grauert
tubes is generally finite, and so the data defining the osculating Bargmann-Fock
space should become singular at this radius. In particular, the identification map
$E_{\tau}$ \eqref{Etau} become and therefore $g_{\tau}^t$ become singular.

We recall from \S \ref{CRSECT}  that  $(H, \omega_{\rho} |_H)$ is a real $J$-invariant symplectic vector space \eqref{H}
of dimension $2m - 2\; (m = \dim_{\R} M)$. $T$ is   Hamilton vector field of $\sqrt{\rho}$ on $M_{\tau}$ with
respect to $\omega_{\rho} $ \eqref{rho}. Since $Dg^t$ preserves $T$ and $\alpha$ it also preserves $H$ and the splitting
\eqref{SP}.  

Thus $Dg^t$  induces a linear symplectic  Poincar\'e map  \begin{equation} \label{Dgtdef} D g^t: H_z \to H_{g^t z}.  \end{equation}
The complexification of $H$ is invariant under the complex structure
$J$ of $M_{\tau}$ and we have a splitting \begin{equation}
\label{CRr} H \otimes \C=H^{1,0} \oplus H^{0,1} \end{equation} into the
$\pm i$ eigenspaces of $J$.  
 Thus we have the complex  decomposition,
   \begin{equation} \label{SPLIT} T_{\C} \partial M_{\epsilon} = H^{1,0} \oplus H^{0,1} \oplus \C T, \end{equation} 
Extending by scalars, we also have  $$(D_{\zeta} g^t): H^{1,0}_{\zeta} \oplus H^{0,1}_{\zeta}
\to H^{1,0}_{\zeta} \oplus H^{0,1}_{\zeta}. $$
$D g^t$ never commutes with $J$ in the Riemannian setting, and the extended $D g^t$
 never   perserves $H^{1,0} $. In the following, we use the notation and terminology of Section \ref{HEISMETSECT}.
%\begin{lem} \label{ID} The map%
%$$X \in H \to \Pi_{1, 0}(X): =  (X - i J X) \in H^{1,0} $$
%is an isomorphism of complexl vector spaces. 
%\end{lem}

\begin{defn} \label{OSCBFDEF} 

Given a  point $\zeta \in M$, we define the {\it osculating Bargmann-Fock
space} at $\zeta$ to be the Bargmann-Fock space of $(H_{\zeta}, J_{\zeta}, \omega_{\zeta})$
and denote it by  $\hcal_{J_{\zeta}, \omega_{\zeta}}$. \end{defn}

 If $\zeta$  is a periodic point for $g^t$, 
let $\gamma = \bigcup_{0 \leq s \leq t} g^s \zeta$ be the corresponding
closed geodesic.   
and we may apply the metaplectic representation to define $W_{J_{\zeta}}((D_{\zeta} g_{\tau}^t))$
as a unitary operator on $(H_{\zeta}, J_{\zeta}, \omega_{\zeta})$.   There is a square root ambiguity
which can be resolved as in \cite{Dau} but for our purposes it is irrelevant.

%As in \S \ref{DETS}, the symplectic form $\omega_{\zeta}$ induces
%a notion of determinant for $S \in Sp(N_{\zeta}, \omega_{\zeta}). $ In 
%a system of local coordinates $z_j, \bar{z}_k$ the mixed Hessian
%on $\mcal_{\zeta}$ of a function $A$ is defined by
%$$\det_{\omega_{\zeta}} (\begin{pmatrix} \frac{\partial^2 A(z, \bar{z})}{\partial z_i \partial \bar{z}_j} \end{pmatrix} = \frac{(\ddbar A)^{n-1}}{\omega_{\zeta}^{n-1}}. $$

\subsection{Determinants}

We now apply the results of Section \ref{CXSTRUCTSECT} and Lemma \ref{DAULEM} to $D_{\zeta} g_{\tau}^{n T_{\zeta}}$. 
%We write $V_{\C} = V^{1,0} \oplus V^{0,1}$ when $J$ is the standard complex structure.
% Denote  the projection to the `holomorphic component' by
%\begin{equation} \label{pi10} \pi^{1,0} : V \otimes \C \to V^{1,0}. \end{equation}
Relative to a  symplectic basis $\{e_j, J e_k\}$   of $H_{\zeta}(\partial M_{\tau})$ in which $J$ assumes the standard form $J_0$,   the matrix of $D_{\zeta} g_{\tau}^{n T(\zeta)}$ has the form \eqref{SBLOCK},
\be  \label{ABCDn} D_{\zeta} g^{n T_{\zeta}}_{\tau} := S^n: =  \begin{pmatrix} A_n  
 & B_n\\ & \\ C_n & D_n \end{pmatrix} \in Sp(m, \R). \ee
If we conjugate to the complexifcation 
$T_z M \otimes \C$ by the natural map $\wcal$  defined in \eqref{PQDEF}, then \eqref{ABCD} conjugates to $$ \bma P_n & Q_n \\ \bar Q_n & \bar P_n \ema \in Sp_c(m). $$ Then, by \eqref{PDEF} 
\begin{equation} \label{PDEFn} P_n =  \left(A_n + D_n+ i (-B_n + C_n) \right) =  P_J  S^n \; P_J: 
 H^{1,0}_J \to  H^{1,0}_J.\end{equation}

We then obtain a formula for the  leading order symbol in Proposition \ref{LINEAR} from Lemma \ref{DAULEM} and \eqref{PDEFn}.
%We have the following useful formula (cf.  \cite[(0.9)]{ZZ18}). 
%Let $\wcal \Xi_{\sqrt{\rho}}$ be the image of the Hamilton vector field  $\Xi_{\sqrt{\rho}}$ in $T_{\zeta} M_{\tau}  \otimes \C$. Let $\alpha = \pi^{1,0} \wcal %\Xi_{\sqrt{\rho}}$, let $\bar{\alpha} \in \pi^{0,1} \wcal \Xi_{\sqrt{\rho}}$,

%\edit{What is $\langle P_n^{-1} \alpha, \overline{\alpha} \rangle$ in the Grauert setting. It may not arise because we are working on a symplectic
%transversal $\mcal_{\zeta}$, and $\Xi$ is transversal to it. In \cite{ZZ18}, by comparison, $\xi_H$ is tangent to $M$ and $\frac{\partial}{\partial 
%\theta}$ does not show up. }

\begin{lem} Let $P_n$ be as in \eqref{PDEF} and $\gcal_n(\zeta)$ as in \eqref{ABCDintro}.  Then (as stated in \eqref{ABCD}),

 \begin{equation} \label{gcalndef} 
 \gcal_n(\zeta): =  \langle W_{J_{\zeta}} \;(D_{\zeta} g^{n T(\zeta)}) \;\Omega_{J_z}, \Omega_{J_z}  \rangle=  (\det P_n)^{-\half}. \end{equation}
\end{lem}

%$S$ is diagonalizable over $\C$. Let $[\lambda_1], \dots, [\lambda_k]$ be its 4-tuples of eigenvalues, and let  $V_{[\lambda]} \subset \R^{2m}, $ resp. %$W_{[\lambda_j]} \subset \C^{2m}$
%be the spaces in \eqref{WDECOMP}. Each $V_{[\lambda]}$ is a symplectic subspace invariant under $J$, so the statement reduces to the case
%of one $[\lambda]$. 

\section{Dynamical Toeplitz operators  and the spectral projections $\Pi_{\tau, \chi}(\lambda)$}
To prove Theorem \ref{SCALINGTHEO}, we construct a parametrix for the smoothed spectral projectors \eqref{SMOOTH} - \eqref{chilambda} as ``dynamical
Toeplitz operators' of a type deployed in \cite{ZJDG} (and elsewhere). To prepare for the parametrix construction,
%\chi* d P^{\tau}_{[0, \lambda]} (\zeta, \bar{\zeta})
%begin{equation} \label{SMOOTH} \chi* d P^{\tau}_{[0, \lambda]} (\zeta, \bar{\zeta})= \int_{\R} \hat{\chi}(t) e^{i
%\lambda t} U_{\C} (t + 2 i \tau, \zeta, \bar{\zeta}) dt.  \end{equation}
  we briefly  review the definition and properties of the \szego \; kernel for a Grauert tube in  section \ref{SZEGOKSECT}. We then review dynamical Toeplitz
operators in Sections \ref{DYNTOEP} - \ref{WTAUTOEP}.   The  spectral projections
\eqref{chilambda} 
%\begin{equation} \label{chilambda}  \chi(\Pi_{\tau} D_{\sqrt{\rho}} \Pi_{\tau} - \lambda) = \Pi_{\tau} \int_{\R} \hat{\chi}(t) e^{- i t \lambda} e^{it\Pi_{\tau} %D_{\sqrt{\rho}} \Pi_{\tau}} dt, \end{equation}
%where $\chi \in \scal(\R)$ (Schwartz space) with $\hat{\rho} \in C_0^{\infty}(\R)$. Here, the extra $\Pi_{\tau}$ is needed to compress the unitary
%group $\hat{\chi}(t) e^{- i t \lambda} e^{it\Pi_{\tau} D_{\sqrt{\rho}} \Pi_{\tau}} $ to $H^2(\partial M_{\tau})$.
 are analogous to the  Fourier (or, spectral) decompositions of the \szego \; projector in the  line bundle setting. We include a 
brief comparison between the Fourier components in the line bundle and the Grauert tube setting in Section \ref{SZEGOCOMP}, since the analogy 
has guided the definition of \eqref{chilambda}.  The proof of Theorem \ref{SCALINGTHEO} requires further preparations in the following two
sections, and is given in
Section \ref{SCTHSECT}.
 
 We use the following notation:  for any manifold $X$, let   $\Psi^s(X)$ denote  the class of pseudo-differential operators of
order $s$ on $X$.

\subsection{\label{SZEGOKSECT} Szeg\"o kernel  }

The semi-classical asymptotics of Theorem \ref{SHORTINTSa}  in the complex domain
is based on the microlocal construction of the Szeg\"o kernel $\Pi_{\tau}$ of
$\partial M_{\tau}$ and of the wave group \eqref{CXWAVEGROUPintro}. The leading
order term of the asymptotics is tantamount to calculating the principal symbol of
the wave group \eqref{UandV}  at  singular times $t$. But the Szeg\"o kernel and
the wave group are Fourier integral operators with complex phase and the symbol calculus
for such operators is rather complicated and not well developed.  
We therefore calculate the asymptotics directly from the Boutet-de-Monvel-Sj\"ostrand parametrix for
$\Pi_{\tau}$ without using either symbol calculus.
We also quote some results of the symplectic spinor symbol calculus  of \cite{BoGu} 
to identify the principal symbol of the wave group \eqref{CXWAVEGROUPintro}, essentially
because the calculation was already done in \cite{ZJDG}. When it comes to
calculating the asymptotics of \eqref{UandV} we find that it is simpler to work by hand
with the parametrix for $\Pi_{\tau}$ and the Phong-Stein foliation. 

%We denote by $\sigma_A$ the principal symbol
%of a pseudodifferential operator $A$. 

  We denote by $\ocal^{s +
\frac{n-1}{4}}(\partial M _{\tau})$ the Sobolev spaces of CR
holomorphic functions on the boundaries of the strictly
pseudo-convex domains $M_{\tau}$, i.e.
\begin{equation} \label{SOBSP} {\mathcal O}^{s +
\frac{m-1}{4}}(\partial M_{\tau}) = W^{s + \frac{m-1}{4}}(\partial
M_{\tau}) \cap \ocal (\partial M_{\tau}), \end{equation}  where
$W_s$ is the $s$th Sobolev space and where $ \ocal (\partial
M_{\tau})$ is the space of boundary values of holomorphic
functions. The inner product on $$\ocal^0 (\partial M _{\tau} )= : H^2(\partial M_{\tau})$$ is
with respect to the Liouville measure or contact volume form \eqref{CONTACTVOL}.
%\begin{equation} \label{LIOUVILLE} d\mu_{\tau} = (i \ddbar
%\sqrt{\rho})^{m-1} \wedge d^c \sqrt{\rho}. \end{equation}

The study of norms of complexified eigenfunctions is intimately
related to the study of the \szego\; kernels $\Pi_{\tau}$ of
$M_{\tau}$, namely the orthogonal projections

\begin{equation} \label{PitauDEF} \Pi_{\tau}: L^2(\partial M_{\tau}, d\mu_{\tau}) \to H^2(\partial M_{\tau},
d\mu_{\tau}) \end{equation}  onto the Hardy space of boundary
values of holomorphic functions in $M_{\tau}$ which belong to $
L^2(\partial M_{\tau}, d\mu_{\tau})$. The \szego \;projector
$\Pi_{\tau}$ is a complex Fourier integral operator with a
positive complex canonical relation. The
 real points of its canonical relation form the graph
$\Delta_{\Sigma}$ of the identity map on the symplectic one
 $\Sigma_{\tau}
\subset T^*
\partial M_{\tau}$ defined by the spray \begin{equation} \label{SIGMATAU} \Sigma_{\tau} =
\{(\zeta, r d^c \sqrt{\rho}(\zeta): r \in \R_+\} \subset T^*
(\partial M_{\tau})
\end{equation}  of the contact form $d^c \sqrt{\rho}$. There exists a symplectic equivalence  \begin{equation} \iota_{\tau} : T^*M - 0 \to
\Sigma_{\tau},\;\; \iota_{\tau} (x, \xi) = (E(x, \tau
\frac{\xi}{|\xi|}), |\xi|d^c \sqrt{\rho}_{E(x, \tau
\frac{\xi}{|\xi|})} ).
\end{equation}

The well-known parametrix construction of  Boutet-de
Monvel-Sj\"ostrand \cite[Theorem 1.5]{BoSj} for \szego\; kernels of strictly
pseudo-convex domains applies to the Grauert tube setting, and we
have
$$\Pi_{\tau} (\zeta, \zeta') \sim  \int_0^{\infty} e^{ i \sigma \psi_{\tau}  (\zeta,\zeta') } s(\zeta,\zeta',\sigma) d\sigma,$$
where the phase $\psi_{\tau}$ is defined in \eqref{psi}.

Also, the
symbol $s \in S^{n}(\partial M_{\tau} \times \partial M_{\tau}
\times \R^+)$ is of the classical type and of order $m$, 
$$s(\zeta, \zeta', \sigma) \sim \sum_{k=0}^{\infty} \sigma^{m-k} s_k(\zeta, \zeta').$$

\subsection{\label{DYNTOEP} Dynamical Toeplitz operators} 

Let $\Pi_{\tau}$ be as in \eqref{PitauDEF}, and let $g^t_{\tau}$ be as in \eqref{gtau} (see also \eqref{DGn}).
The time evolution of $\Pi_{\tau}$ under the flow $g^t_{\tau}$ is defined by
\begin{equation} \label{PITAUT} \Pi_{\tau}^t = g_{\tau}^{-t}  \Pi_{\tau}  g_{\tau}^t. \end{equation}   
It is another Szeg\"o  projector adapted  to the graph of $g^t_{\tau}$ on the symplectic cone $\Sigma_{\tau}$;
since $g^t_{\tau}$
is not a family of holomorphic maps in general, $\Pi_{\tau}^t$ is
associated to a new CR (complex) structure and  translation by $g^t_{\tau}$   does not commute with $\Pi_{\tau}$. 
But $\Sigma_{\tau}$ is invariant under the flow and $g^t_{\tau}$ clearly commutes with the identity map on $\Sigma_{\tau}$.
The change in the range of $\Pi_{\tau}^t$ under $t$ is encoded to leading order by a pullback by a canonical transformation
on $T^* \partial M_{\tau}$ which is the identity along $\Sigma_{\tau}$ but whose derivative rotates the Lagrangian
subspace $\Lambda$ defining the ground state. The details were worked out in \cite{Z97} in the line bundle setting, but
much of the analysis generalizes to Grauert tubes (see \cite{ZPSH1, ZJDG}) . We summarize the results in this section.

Under $Dg^t_{\tau}$, the Lagrangian $\Lambda$  goes to a new
Lagrangian $\Lambda_t$ and $\sigma_{\Pi_{\tau} ^t}$ is 
a rank one projector onto a ground state $e_{\Lambda_{\tau}^t}$  depending on $t$.
As in \cite{Z97, ZPSH1, ZJDG}, we define the symbol,
\begin{equation} \label{sigmataut} \sigma^0_{\tau, t} =  \langle e_{\Lambda_t}, e_{\Lambda} \rangle^{-1} \end{equation}
to be the (inverse of) the inner product of the ground states with respect to $\Lambda$ and
$\Lambda_t$. In the linear model, this inner product is calcuated in Lemma \ref{DAULEM}
and the formula for the linear quantization \eqref{eta}  implies that 
$$\Pi_{\tau} \sigma_{t \tau} (g_{\tau}^{-t})^* \Pi_{\tau} $$
is unitary in the Bargmann-Fock setting. In the nonlinear setting, it is unitary modulo
compact operators (i.e. a Toeplitz operator of order $-1$). 
To see this,  we observe that the  composite symbol is  \begin{equation}
\label{SYMBPIT}  \sigma(\Pi_{\tau} \Pi_{\tau}^t \Pi_{\tau}) = | \langle e_{\Lambda_{\tau}} , e_{\Lambda_{\tau}^t} \rangle |^2
\sigma_{\Pi_{\tau}}. \end{equation}
  In the linear case, it is the matrix element
given in Lemma \ref{DAULEM} and \eqref{eta}.
% In the nonlinear case it is shown 
%in  \cite{ZID} that the same formula is valid.

The change of $\Pi_{\tau}$  under $g^t_{\tau}$ reflects the change in complex structure.
Let $J$ be the complex structure of $M_{\tau}$.
% and denote the Szeg\"o projector with respect to this
%complex structure by $\Pi_{\tau, J}$.
 Under $g_{\tau}^t$ it is moved to a different complex structure $$J_t : = g^t_{\tau *} J
: = D g_{\tau}^t J Dg_{\tau}^{-t} . $$   The Szeg\"o projector  with respect to this deformed
complex structure is $\Pi_{\tau}^t$ above. 
%We further denote by $\ocal^0_{J_t}(\partial M_{\tau})$ the Hardy space with respect
%to $J_t$, i.e. the range of $\Pi_{\epsilon, J_t}$. 
%The operator $$\ucal_{J_t, J}: = \Pi_{\tau} \sigma^0_{t \tau} (g_{\tau}^{-t})^* \Pi_{\tau}  (g_{\tau}^{t})^*$$ is a unitary operator modulo operators of %order $-1$ and the complete symbol
%$\sigma_{t \epsilon}$ 
%can be arranged to make it unitary. The original operator is
%$$ \Pi_{\tau} \sigma_{t \tau} (g_{\tau}^{-t})^* \Pi_{\tau}  = \ucal_{J_t, J} \circ (g_{\tau}^{-t})^*. $$

\subsection{\label{WTAUTOEP} $\wcal_{\tau} (t)$ as a dynamical Toeplitz operator} 

\begin{prop} \label{WCALPROP} The unitary group  $\wcal_{\tau}(t)$ of \eqref{WCALintro} is a one-parameter group of unitary dynamical Toeplitz operators.\end{prop}

\begin{proof} We deploy the ingenious argument of \cite[Lemma 12.2]{BoGu}. According to this Lemma (but in the notation and setting
of this article) given any first order  Toeplitz operator
$ \Pi_{\tau} P \Pi_{\tau} $ with $P \in \Psi^1(\partial M_{\tau})$, there exists a first order pseudo-differential operator $Q$ on $\partial M_{\tau}$ such 
that $[Q, \Pi_{\tau}] = 0$ and such that $\Pi_{\tau} P \Pi_{\tau} = \Pi_{\tau} Q \Pi_{\tau}. $  \footnote{$\Psi^k(X)$ denotes the space of kth order poly-homogeneous pseudo-differential operators
on a manifold $X$. See \cite{Ho} for background. }  Let $Q_{\sqrt{\rho}}$ be this operator $Q$ in the case where $P = D_{\Xi_{\sqrt{\rho}}}$. 
Then, 
$$\wcal_{\tau}(t) = \Pi_{\tau}  e^{it Q_{\sqrt{\rho}}}. $$
Indeed, both sides solve the evolution equation,
$$\left\{ \begin{array}{l} \frac{d}{i dt} W(t) = Q_{\sqrt{\rho}} \wcal_{\tau}(t) = \Pi  D_{\Xi_{\sqrt{\rho}}} \Pi W(t), \\ \\
W(0) = \Pi_{\tau}. \end{array} \right. $$

Now, $e^{i t Q_{\sqrt{\rho}}}$ is a unitary group of Fourier integral operators on $L^2(\partial M_{\tau})$ by standard Fourier integral operator
theory (see e.g. \cite[Volume IV]{Ho}). By the composition theorem for the composition of the Fourier integral operator $e^{i t Q_{\sqrt{\rho}}}$ 
with the \szego\; projection $\Pi_{\tau}$ of \cite{BoGu} (or, alternatively, for Fourier integral operators with positive complex phases of \cite{MeSj}),
$\wcal_{\tau}(t)$ is a Toeplitz Fourier integral operator adapted to the Hamilton flow of $\Xi_{\sqrt{\rho}}$ on $\Sigma_{\tau}$ (see \cite[Appendix]{BoGu})
for adapted Fourier integral Toeplitz operators).

On the other hand, for any zeroth order pseudo-differential operator $\sigma(x, D)$ on $L^2(\partial M_{\tau})$, the dynamical Toeplitz operator
$\Pi_{\tau} \sigma(t, x, D) (g_{\tau}^t)^* \Pi_{\tau}$ is also a Fourier integral Toeplitz operator or, equivalently, a Fourier integral operator with complex
phase that commutes with $\Pi_{\tau}$. Therefore, $\wcal_{\tau}(t)$ and $\Pi_{\tau} \sigma(x, D) (g_{\tau}^t)^* \Pi_{\tau}$ are both Fourier integral
Toeplitz operators with the same canonical relation, i.e. adapted to the geodesic flow on $\Sigma_{\tau}$. We now choose $\sigma(t,x, D)$
so that the principal  symbols of $\Pi_{\tau} \sigma(t, x, D) (g_{\tau}^t)^* \Pi_{\tau}$  and of $\wcal_{\tau}(t)$ coincide, i.e so that $\sigma(t, x, \xi) |_{\Sigma_{\tau}} $ equals \eqref{sigmataut}.  By 
induction  on the order of the symbol, one can improve $\sigma(t, x, D)$ so that its complete symbol (restricted to $\Sigma_{\tau}$) agrees with
that of $\wcal_{\tau}(t)$. Then, $\Pi_{\tau} \sigma(t, x, D) (g_{\tau}^t)^* \Pi_{\tau} -\wcal_{\tau}(t)$ is a smoothing operator, and the Proposition follows.

\end{proof}

\begin{rem} The operator $Q$ may be thought of as Wick normal-ordering $P$. It would be interesting to construct $Q$ explicitly when
$P = D_{\Xi_{\sqrt{\rho}}}$. \end{rem}

\begin{prop} \label{MAIN1} There exists a poly-homogeneous  pseudo-differential operator $\hat{\sigma}_{t \tau}(w, D_{\Xi_{\sqrt{\rho}}})$ on $\partial M_{\tau}$ with complete symbol of the classical form
$$\sigma_{t \tau}(w, r) \sim \sum_{j = 0}^{\infty} \sigma_{t, \tau, j}(w) r^{-j} $$
and with $\sigma_{t, \tau, 0} = \sigma_{t, \tau}^0$, so that 
for $\zeta \in \partial M_{\tau},$ modulo smoothing Toeplitz operators,
\begin{equation} \label{WCALPIFORM} \wcal_{\tau} (t, \zeta, \bar{\zeta}) \simeq \int_{\partial M_{\tau}} \Pi_{\tau}(\zeta, w) \hat{\sigma}_{t, \tau} \Pi_{\tau}(g_{\tau}^t w, \bar{\zeta}) d\mu_{\tau}(w). \end{equation}
\end{prop}

The symbol $\sigma_{t \tau}(w, r)$
is  a zeroth order  polyhomogeneous function on $\Sigma_{\tau}$, i.e. a classical symbol of order zero.

\subsection{\label{SZEGOCOMP} Comparison of  dynamical Toeplitz operators in the Grauert and line bundle settings}
In this section, we extend the comparisons in Section \ref{COMPARISON} between the CR geometry in the Grauert and line bundle
settings to the spectral theory of dynamical Toeplitz operators.
The \szego\; projector in the line bundle setting is the orthogonal projection $\hat{\Pi}: L^2(X_h) \to H^2(X_h)$ where $H^2$ is the Hardy
space of $L^2$ CR holomorphic functions. Under the $S^1$ action on $L^*$ it has the Fourier decomposition $\hat{\Pi} = \sum_{N=0}^{\infty}
\hat{\Pi}_N$, and there is a canonical lift, $s \to \hat{s}$ from $H^0(M, L^N) \to \rm{Range}(\hat{\Pi}_N)$, from holomorphic sections of $L^N$
to equivariant functions on $X_h$,  which conjugates $\hat{\Pi}_N$ with
the standard Bergman-\szego \; kernels $\Pi_{h^N}$ on $H^0(M, L^N)$. It follows that the Fourier decomposition of $\hat{H}$ is the same 
as the spectral decomposition of $D_{\theta} : = \frac{1}{i} \frac{\partial}{\partial \theta}$ on $H^2(X_h). $

The analogue of this spectral decomposition in the Grauert tube setting is that of the Toeplitz operator $\Pi_{\tau} D_{\sqrt{\rho}} \Pi_{\tau}$
where $D_{\sqrt{\rho}} = \frac{1}{i} \Xi_{\sqrt{\rho}}$ is the differential operator induced by the Hamiltonian
 vector field of $\sqrt{\rho}$.  One might also consider $\Pi_{\tau} (g_{\tau}^t)^* \Pi_{\tau}$, the compression to $H^2(\partial M_{\tau})$
  of the pullback (or, composition)
operator with the Hamilton flow, but as mentioned above, $[D_{\sqrt{\rho}}, \Pi_{\tau}] \not=0$ and unitary group generated by $\Pi_{\tau} D_{\sqrt{\rho}} \Pi_{\tau}$ on $H^2(\partial M_{\tau})$ is not quite the same as the 1-parameter family $\Pi_{\tau} (g_{\tau}^t)^* \Pi_{\tau}$, which is not unitary and
not a group. 

$\Pi_{\tau} D_{\sqrt{\rho}} \Pi_{\tau}$ is an elliptic Toeplitz operator with a discrete spectrum, which is very close to that of $\sqrt{-\Delta}$
in the sense that, after the identifications discussed in \cite{ZJDG},  the two operators have the same principal symbol. The analogue of the Fourier decomposition of $\hat{\Pi}$ in the line bundle setting is, then, the spectral decomposition of $\Pi_{\tau} D_{\sqrt{\rho}} \Pi_{\tau}$  on $H^2(\partial M_{\tau})$.  In the line bundle case, the spectrum lies in $\Z_+$ and the eigenvalues have large multiplicities.
In the Grauert analogue, one may expect that the spectrum of $\Pi_{\tau} D_{\sqrt{\rho}} \Pi_{\tau}$  is quite irregular and, generically, the
eigenvalues have multiplicity; we will not prove this here, but it is a simple Toeplitz analogue of well-known theorems in the Riemannian
setting (the Helton clustering theorem and the Uhlenbeck generic simplicity of the spectrum of $\Delta_g$; see \cite{BoGu} for background).

%In  the analogy to line bundles, the geodesic flow (transported to $\partial M_{\tau})$ plays the role of the circle $S^1$ action
%on the unit co-disk bundle of the line bundle. But the $S^1$ action is holomorphic and periodic, while the geodesic flow is rarely periodic and
%never holomorphic, so the analogy only goes so far. It is argued in Section \ref{SCALINGSECT} that the closest analogue of $\Pi_{h^k}(z,w)$
%in the Riemannian Grauert setting are the spectral localizations \eqref{chilambda}. 

In the setting of line bundles $L \to M$,  the semi-classical \szego \; kernels $\Pi_{h^k}(z,w)$ are Fourier components, $$\Pi_{h^k}(z,w) 
= \frac{1}{2 \pi}  \int_0^{2 \pi} \Pi_h (x, r_{\theta} y) e^{- i k \theta} d \theta $$
of the \szego \; projector $\Pi_h(x, y): L^2(\partial D_h^*) \to H^2(\partial D_h^*)$ onto boundary values of holomorphic functions in  the strictly pseudo-convex domain $D_h^* \subset L^*$ in the dual line bundle $L^*$ where $D_h^*$ is the dual unit
disk bundle $\{\lambda \in L^*: h^*(\lambda) < 1\}$.  One obtains their asymptotic expansions as $k \to \infty$ 
by applying using a Boutet de Monvel - Sj\"ostrand parametrix for  $\Pi_h$ and by applying
a complex stationary phase argument. In the setting of Grauert tubes one also has a Boutet de Monvel - Sj\"ostrand parametrix for  $\Pi_{\tau}$
and can try to adapt the argument to obtain analogous asymptotics for  $\Pi_{\chi, \tau}(\lambda) $
 \eqref{chilambda}. The direct analogue would apply to the integral, 
 \begin{equation} \label{SDEF} S_{\lambda, \chi, \tau} (x, y): = \int_{\R} \hat{\chi}(t) e^{- i \lambda t} \Pi_{\tau} (x, \hat{g}^t y) dt, \end{equation}
 for some $\hat{\chi} \in C^{\infty}_0(\R)$. This is not quite the right analogue, however, because unlike the $S^1$ action, $\hat{g}^t$ does not
 act holomorphically, hence composition with $\hat{g}^t$ does not commute with $\Pi_{\tau}$, and therefore $S_{\lambda, \chi, \tau}(x, y)$ fails
 to be CR holomorphic in the $y$ variable.  Indeed, $S_{\lambda, \chi, \tau}$ is the Schwartz kernel of the operator
 $\int_{\R} \hat{\chi}(t) \Pi_{\tau} \circ \hat{g}^{t *}dt$.  Using the  Boutet de Monvel - Sj\"ostrand parametrix for $\Pi_{\tau}$, this one obtains
 $$S_{\lambda, \chi, \tau}(x, y) = \int_{\R} \int_0^{\infty} \hat{\chi}(t) e^{- i \lambda t}  e^{\theta \psi(x, \hat{g}^t y) }s(x, \hat{g}^t y, \theta) dt d \theta, $$
 where $s(x, y, \lambda)$ is a semi-classical symbol of order $m $. 
By stationary phase, one finds that  if $\rm{supp} \hat{\chi}$ is close to $0$, then the only critical point occurs at $\theta =1, t = 0$ and 
 $$S_{\lambda, \chi, \tau}(x, y) \simeq \lambda^m e^{\lambda \psi(x, y) } \wt s(x, y, \lambda), $$
 where $s(x, y, \lambda)$ is classical symbol of order zero. 
 
 For our problem, $\hat{g}^t$ is not holomorphic and it is necessary to work with the more complicated operator  \eqref{chilambda}. 
 The pointwise values on the anti-diagonal of  $\int_{\R}\Pi_{\tau}  \hat{\chi}(t) \Pi_{\tau} \circ \hat{g}^{t *}dt$ are quite different from
 those of \eqref{SDEF}, as the next result shows.

 Theorem \ref{SCALINGTHEO} in the Riemannian Grauert setting is somewhat analogous to  \cite[Theorem 0.9]{ZZ18}  in the line bundle setting (see  Section \ref{RELATED}). However, there are significant differences in the two settings and the analogy only goes so far. The
most obvious difference is that, in the line bundle setting, there are two Hamiltonians: the generator $\frac{\partial}{\partial \theta}$ of rotations
in the fibers of the line bundle $L \to M$ and an independent  Toeplitz Hamiltonian $\hat{H}_k$ whose spectrum is the main object of study. In
the Riemannian setting, there is just one operator,  $\Pi_{\tau} D_{\sqrt{\rho}} \Pi_{\tau} $  \eqref{DDEF}, or alternatively (and essentially
equivalently)     $\sqrt{\Delta}$.  As mentioned above, $\Pi_{\tau} D_{\sqrt{\rho}} \Pi_{\tau} $ is the analogue of $\frac{\partial}{\partial \theta}$, but
is spectrum is the main object of study in the Grauert tube setting, and it simultaneously plays the role of  $\frac{\partial}{\partial \theta}$ and
of $\hat{H}_k$.

\section{Analytic continuation of the Poisson  kernel}

For the remainder of the article, we analyze the  Laplacian and associated operators. In the next two sections, we build up enough background
to   show that $U(t + i \tau, \zeta, \bar{\zeta})$
\eqref{UTTAUINTRO} 
is also a dynamical Toeplitz operator of the same type as $W_{\tau}(t, \zeta, \bar{\zeta})$ in Section \eqref{WTAUTOEP}. Much of this statement is proved in  
 \cite{ZPSH1,ZJDG}, using  the analytic continuation of the Poisson-wave kernel, and we review that material in this section.
 % as a
%preparation for studying the kernel $U(t + i \tau, \zeta, \bar{\zeta})$ \eqref{UTTAUINTRO}.
We state the result in the language of adapted Fourier
integral operators of the Appendix of \cite{BoGu}, where only the real points of canonical
relations are considered. 
 We  use a slight extension of the notion of adapted Fourier integral operator, 
in which the homogeneous symplectic map may be a symplectic embedding rather than
a symplectic isomorphism. All of the composition results of \cite{BoGu} extend readily to this case.  
For the definitions of  Hermite Fourier integral operators, and 
operators ``adapted" to the graph of the Hamiltonian flow of $\sqrt{\rho}$ on the symplectic cone $\Sigma_{\tau} $
we refer to the Appendix of \cite{BoGu}.

The wave group of $(M, g)$ is the unitary group $U(t) = e^{ i
 t \sqrt{\Delta}}$. Its kernel $U(t, x, y)$ solves the  `half-wave equation',
\begin{equation} \label{HALFWE} \left(\frac{1}{i} \frac{\partial }{\partial t} -
\sqrt{\Delta}_x \right) U(t, x, y) = 0, \;\; U(0, x, y) =
\delta_y(x). \end{equation}  It is well known \cite{Ho,DG} that
$U(t, x, y)$ is the Schwartz kernel of a Fourier integral
operator,
$$U(t, x, y) \in I^{-1/4}(\R \times M \times M, \Gamma)$$
with underlying canonical relation $$\Gamma = \{(t, \tau, x, \xi,
y, \eta): \tau + |\xi| = 0, G^t(x, \xi) = (y, \eta) \} \subset T^*
\R \times T^*M \times T^*M. $$
%This means that there exists  a microlocal parametrix for $U(t, x,
%y)$ with phase parameterizing $\Gamma$. We only need one for small
%times. For $|t|$ small and $(x, y)$ near the diagonal, there
%exists a parametrix of the form,
%\begin{equation} \label{PARAONE} U(t, x, y) = \int_{T^*_y M} e^{ i
%t |\xi|_{g_y} } e^{i \langle \xi, \exp_y^{-1} (x) \rangle} A(t, x,
%y, \xi) d\xi
%\end{equation} where $|\xi|_{g_x} $ is the metric norm function at
%$x$, and where $A(t, x, y, \xi)$ is a polyhomogeneous amplitude of
%order $0$ which is supported near the diagonal.
\subsection{Poisson wave kernel}

The Poisson-wave kernel is the analytic continuation $U(t + i \tau, x, y)$  of the wave kernel  with respect to time,  $ t \to t + i \tau\in \R \times \R_+$. For $t = 0$ we
obtain
the  Poisson semi-group
$U(i \tau) = e^{- \tau \sqrt{\Delta}}$ on $L^2(M)$.  
For general $t + i \tau$ we define the Poisson-wave kernel in the real domain by  the eigenfunction expansion for $\tau > 0$,
\begin{equation}\label{POISEIGEXP}  U ( i
\tau, x, y) = \sum_j e^{i (t + i \tau) \lambda_j} \phi_j(x)
\phi_j(y).
\end{equation}
As discussed in \cite{ZPSH1}, this kernel is globally  real analytic
on $M \times M$ for any $\tau
> 0$.
The Poisson-wave  kernel $U(t + i \tau, x, y)$  admits an analytic
continuation $U_{\C}(t + i \tau, \zeta, y)$ in the first variable
to  $M_{\tau} \times M$.   When the real time $t=0$,  the operator kernel $U_{\C}(i \tau, \zeta, y)$ $P^{\tau}$ defines the  operator
\begin{equation} \label{PTDEF} P^{\tau}: = \Pi_{\tau} \circ U_{\C} (i \tau): L^2(M)
\to H^2(\partial M_{\tau}) \end{equation} 
with Schwartz kernel \eqref{PTKER}. The \szego \; kernel is not needed here, since $U_{\C}(i \tau, \zeta, y)$ is holomorphic in $\zeta$, but
is put in to emphasize that point.
We also define the adjoint operator  $P^{\tau *}: H^2(\partial M_{\tau})
\to L^2(M) $
which has the Schwartz kernel

\begin{equation} \label{PTKER*} 
P^{\tau *}(y, \bar{\zeta}) = \sum_j e^{- \tau \lambda_j} \overline{\phi_j^{\C}(\zeta)} \phi_j(y), \;\; y \in M, \zeta \in \partial M_{\tau}.
\end{equation}

The following result was  stated   by
Boutet de Monvel (and given a detailed proof in three recent articles \cite{ZPSH1,L18}.

\begin{theo}\label{BOUFIO}   For sufficiently small $\tau$, $P^{\tau}: = \Pi_{\tau} \circ U_{\C} (i \tau): L^2(M)
\to H^2(\partial M_{\tau})$ is a   Fourier integral
operator with complex phase in the sense of \cite{MeSj} of order $- \frac{m-1}{4}$  adapted to the canonical
relation
$$\Gamma = \{(y, \eta, \iota_{\tau} (y, \eta) \} \subset T^*M \times \Sigma_{\tau}.$$
Moreover, for any $s$,
$$P^{\tau} = \Pi_{\tau} \circ U_{\C} (i \tau): W^s(M) \to {\mathcal O}^{s +
\frac{m-1}{4}}(\partial  M_{\tau})$$ is a continuous isomorphism.
\end{theo}

 Theorem \ref{BOUFIO} readily extends  to $U_{\C}(t + i \tau)$. Referring to \eqref{PitauDEF},

\begin{prop}

 \label{REALPTAU}$P^{\tau} \circ U_{\C} (t) : C_c(\R \times M) \to H^2(\partial M_{\tau})$ is a Fourier integral operator with complex phase  of order 
 $- \frac{m-1}{4}$ adapted to the canonical relation 
\begin{equation} \label{CR} \{(t, E, \chi_{\tau, t}(y, \eta), y, \eta): E + |\eta| = 0\} \subset T^* \R \times
\Sigma_{\tau} \times T^*M, \end{equation} where $\chi_{\tau, t}$ is the symplectic isomorphism 
\begin{equation} \label{chitDEF} \chi_{\tau, t}  (y, \eta) = \iota_{\tau} (G^t(y, \eta), y, \eta ) : T^*M - 0 \to
\Sigma_{\tau}.
\end{equation}
Equivalently, $P^{\tau} \circ U(t)$ is a   Fourier integral operator of Hermite type
of order 
 $- \frac{m-1}{4}$  associated
to the  canonical relation
\begin{equation} \label{GAMMAtDEF} \Gamma_{\tau} = \{(t,E),  (\iota_{\tau} (G^t(y, \eta), y, \eta ) \} \subset  \Sigma_{\tau}
\times T^* M. 
\end{equation}
\end{prop}

\begin{proof} This follows  from Theorem \ref{BOUFIO} and from the fact  proved in \cite{BoGu}, Theorems 3.4 and 7.5,  that the compositon of a Fourier integral operator
and a Fourier integral operator of Hermite type is also a Fourier integral operator of Hermite type, with a certain
addition law for the orders and a composition law for the symbols. \end{proof}

\subsection{Singular support of   $ U_{\C} (t + 2 i \tau,
\zeta, \bar{\zeta})$ }

In this section, we extend the discussion of the analytic continuation of the Poisson kernel to 
the Poisson wave kernel on the anti-diagonal,  \begin{equation} \label{Uttau} U(t +  2 i \tau, \zeta, \bar{\zeta}) \in
\dcal'(\R \times \bar{\Delta}_{M_{\tau} \times M_{\tau}}). \end{equation}
The main result determines the singularities of \eqref{UTTAUINTRO} for fixed $\zeta$
as a distribution in $t$. 
It shows that  $ U_{\C} (t + 2 i \tau,
\zeta, \bar{\zeta})$ is singular in $t$ only if $\zeta$
corresponds to a point $(x, \xi) \in S^* M$ for which the geodesic
$G^t(x, \xi)$ is periodic and then the singular times are
multiples of the lengths of the corresponding closed geodesic.
This should be compared with the well-known fact (see e.g.
\cite{SV,SoZ}) that in the real domain, $U(t, x, y)$ is singular
at the lengths of all geodesic segments from $x$ to $y$. The same
result will be proved below by a parametrix construction, but it 
is possible to prove this statement just using the results of
the previous section. The parametrix construction is valuable
in computing the leading coefficient, which is not easy to obtain
from the abstract approach.

To analyze the  singularities,  we use the calculus of Hermite
Fourier integral operators adapted to symplectic maps in the framework of \cite{BoGu}.
%We then repeat the analysis using the calculus of Fourier integral operators with
%complex phase in the framework of \cite{MeSj}. This involves the study of the complex
%canonical relation. 

%In both approaches we begin with the following

\begin{prop} \label{CLOSED} For fixed $\zeta \in \partial M_{\tau}$, the singular support of   the
distribution $t \in \R \to U_{\C}  (t + 2 i \tau, \zeta,
\bar{\zeta})$ consists of times  $T$ such that $g^T_{\tau}(\zeta) = \zeta$.
% $\zeta = \exp_y i
%\eta$ where $G^T(y, \eta) = (y, \eta)$.
 If no  $T \not= 0$ exists,
the singular support is $\{0\}$.

\end{prop}

\begin{proof}

The first step is to express \eqref{UTTAUINTRO} as a composition of Fourier integral operators.

\begin{lem} \label{UvsVtilde} We have,

\begin{equation} 
\begin{array}{lll} 
P^{\tau} U(t) P^{\tau*} (\zeta, \overline{\zeta})
& = &
  = \sum_{j, k}  e^{(- 2 \tau + i t) \lambda_j}\int_M  \phi_j^{\C}(\zeta)\phi_j(y)
\overline{\phi_{\lambda_k}^{\C}(\zeta)} \phi_k(y) dV_g(y)
\\&&\\& = & U_{\C} (t + 2 i \tau, \zeta, \overline{\zeta}) 
\end{array}
\end{equation}

\end{lem}

\begin{proof} The identity follows directly from the eigenfunction 
expansions \eqref{PTKER} and \eqref{PTKER*} and orthonormality of $\phi_j(y)$ in the real domain:

\begin{equation} \label{EFORM2}
\begin{array}{lll} U_{\C} (t + 2 i \tau, \zeta, \overline{\zeta}) &  = &
P^{\tau} U(t) P^{\tau*} (\zeta, \overline{\zeta}) \\ &&\\ & = &
  = \sum_{j, k}  e^{(- 2 \tau + i t) \lambda_j}\int_M  \phi_j^{\C}(\zeta)
\phi_j(y)
\overline{\phi_{\lambda_k}^{\C}(\zeta)} \phi_k(y) dV_g(y)
\end{array}
\end{equation}

\end{proof}

By Lemma \ref{UvsVtilde} we can calculate the wave front set of $U(t + 2 i \tau, \zeta, \overline{\zeta})$
by composing wave front sets in the real domain of the adapted Hermite Fourier
integral operators $P^{\tau} U(t)$ and $P^{\tau*}$. 
Proposition \ref{REALPTAU} implies that, for $\sqrt{\rho}(\zeta) = \tau, $ and   $t \in \R$,
 the   singular support of the
distribution $t \in \R \to U_{\C}  (t + 2 i \tau, \zeta,
\bar{\zeta})$ is the set
$$\begin{array}{l} \rm{SingSupp} (t \to U_{\C}  (t + 2
i \tau, \zeta, \bar{\zeta})) \\ \\=    \{t:  \exists (y, \eta) \in T^*M
: |\eta| = \tau,  \; \exp_y(i \eta) =  \zeta, \;\; \exp_y (t + i
\tau) (- \eta) = \bar{\zeta} \}. \end{array}$$

To complete the proof of Proposition \ref{CLOSED}, we  observe that $\exp_y i \eta = \zeta$ implies $\exp_y (- i
\eta) = \bar{\zeta}$. Since   $G_{\C}^{- i \tau}(y, - \eta)$ must
be tangent to  $\Sigma_{\tau}$, the terminal momentum must be $d^c
\sqrt{\rho}$. It follows that
$$(y, \eta) = G^{i \tau}(\bar{\zeta}, d^c \sqrt{\rho}). $$
If $t$ lies in the singular support, then $\exp_y (t + i \tau) (-
\eta) = \bar{\zeta} $ and since  the terminal momentum must again
be tangent to $\Sigma$ we have
$$(y, \eta) = G^{-t - i \tau}(\bar{\zeta}, d^c \sqrt{\rho}),\;\; \rm{hence}\;\;
G^{t}(y, \eta) = (y, \eta). $$

\end{proof}

\section{\label{WG} The wave group in the complex domain as a dynamical Toeplitz operator} 

In this section, we prove the identity  \eqref{UandV}. Consequently,  \eqref{UTTAUINTRO} 
is  a dynamical Toeplitz operator of the same type as $W_{\tau}(t, \zeta, \bar{\zeta})$ in Section \eqref{WTAUTOEP}. We use 
 symbol calculus of  Toeplitz Fourier integral operators to calculate the symbol of \eqref{UTTAUINTRO}, which is apparently
 more complicated than $\wcal_{\tau}(t)$, and prove \eqref{UandV}.
In effect, the main result is proved in    \cite[Proposition 44.]{ZJDG} and we review the relevant background. In Section \ref{KtauSECT}, we introduce the
kernel $K_{\tau}$ and prove Lemma \ref{L2LEMintro}.

As above, let $\frac{1}{i} D_{\Xi_{\sqrt{\rho}}}$ denote the self-adjoint directional derivative 
in the direction of $\Xi_{\sqrt{\rho}}$. 
The directional derivative $D_{\Xi_{\sqrt{\rho}}}$  is elliptic on the kernel of $\dbar_b$,
i.e. its symbol is nowhere vanishing on $\Sigma_{\tau} \backslash \{0\}$. Hence  $\Pi_{\tau}
D_{\Xi_{\sqrt{\rho}}} \Pi_{\tau}$ is an elliptic Toeplitz operator. The symbol $\sigma_{t \tau}(w, r)$
is  a polyhomogeneous function on $\Sigma_{\tau}$. 
Also as above, for any manifold $X$, let   $\Psi^s(X)$ denote  the class of pseudo-differential operators of
order $s$ on $X$.

The next  Proposition  is   \cite[Proposition 44.]{ZJDG} and is analogous to Proposition \ref{MAIN1} for $W_{\tau}(t)$ and Proposition
\ref{SMOOTHCOR1}.

\begin{prop} \label{MAIN} There exists a poly-homogeneous  pseudo-differential operator $\hat{\sigma}_{t \tau}(w, D_{\Xi_{\sqrt{\rho}}})$ on $\partial M_{\tau}$ with complete symbol of the classical form
$$\sigma_{t \tau}(w, r) \sim \sum_{j = 0}^{\infty} \sigma_{t, \tau, j}(w) r^{-\frac{m-1}{2} -j} $$
on $\Sigma_{\tau}$,
and with $\sigma_{t, \tau, 0} = \sigma_{t, \tau}^0$, so that 
for $\zeta \in \partial M_{\tau},$ modulo smoothing Toeplitz operators,
\begin{equation} \label{UPIFORM} U_{\C}(t + 2  i \tau, \zeta, \bar{\zeta}) \simeq \int_{\partial M_{\tau}} \Pi_{\tau}(\zeta, w) \hat{\sigma}_{t, \tau} \Pi_{\tau}(g_{\tau}^t w, \bar{\zeta}) d\mu_{\tau}(w). \end{equation}
%Consequently,
%there exists a semi-classical amplitude,
%$A_{\lambda}(\zeta, \bar{\zeta}, \sigma_1, \sigma_2, t, w)$,  of order  $- \frac{m-1}{2}$ such that

%\begin{equation}\label{SMOOTHED} \begin{array}{lll} \chi * d P_{[0, \lambda]}^{\tau}(\zeta, \bar{\zeta})
%& = &\lambda^{2m}  \int_{\R} \int_{0}^{\infty} \int_0^{\infty} \int_{\partial M_{\tau}} \hat{\chi}(t) e^{i \lambda \Phi(t, \zeta,\bar{\zeta}, w, \sigma_1, \sigma %_2) }
%A_{\lambda}(\zeta, \bar{\zeta}, \sigma_1, \sigma_2, t, w)
% d\sigma_1 d\sigma_2 d w dt. \end{array} \end{equation} 
\end{prop}

The main point of the proof is to show that $U_{\C}(t + i \tau, \zeta, \bar{\zeta})$ may be constructed as the dynamical
Toeplitz operator $V_{\tau}^t$ of \eqref{CXWAVEGROUPintro}. 
The proof consists of a sequence of Lemmas from \cite{ZPSH1, ZJDG}.
\begin{proof}

The following Lemma  is  \cite[Lemma 8.2]{ZPSH1} (see also \cite[Section 3.1]{ZJDG}).
\begin{lem} \label{OLD}  Let $A_{\tau} = (P^{\tau *} P^{\tau})^{-\half}$.  Then, 

\begin{itemize}
\item (i)\;  $A_{\tau} \in \Psi^{ \frac{m-1}{4}}
(M)$, with  principal symbol $|\xi|^{ 
\frac{m-1}{4}}$.
% as a function of $(\zeta, r \alpha_{\zeta})$  on $\Sigma_{\tau}$. 
\bigskip

\item (ii) \;  $U_{\C}(i \tau)^* U_{\C}(i \tau) \in \Psi^{-
\frac{m-1}{2}}(M)$ with principal symbol $|\xi|_g^{- 
\frac{m-1}{2}}.$ \bigskip

%\item (iii) \; $U_{\C}(t + i \tau, \zeta, \bar{\zeta})$ is a 
% \circ U_{\C}(i \tau)^* = \Pi_{\tau}
%A_{\tau} \Pi_{\tau}$.
 \bigskip

\end{itemize}

\end{lem}

We note that $ P^{\tau *} P^{\tau}: L^2(M) \to L^2(M)$. It is proved in     \cite[Lemma 8.2]{ZPSH1} (see also \cite{ZJDG}) that 
$P^{\tau *} P^{\tau} \in \Psi^{-\frac{m-1}{2}} (M)$ with principal symbol $|\xi|^{-\frac{m-1}{2}}$, proving (i). Statement (ii) follows from
 Theorem \ref{BOUFIO}.

To prove \eqref{UandV}, we  introduce a slightly modified version of $P^{\tau} U(t) P^{\tau *}$ from
\cite{ZJDG}.

\begin{defin} \label{VINTROalt}   As above, let $A_{\tau} = (P^{\tau *} P^{\tau})^{-\half}$ and define 
$$
%\left\{\begin{array}{l}
%V_{\tau}^t := P^{\tau} U(t) A_{\tau}^2 P^{\tau *}: H^2(\partial M_{\tau}) \to H^2(\partial M_{\tau}), \\ \\
\tilde{V}_{\tau}^t := P^{\tau} A_{\tau} U(t) A_{\tau} P^{\tau *}: H^2(\partial M_{\tau}) \to H^2(\partial M_{\tau}). 
%\end{array} \right.
$$
\end{defin}
As the notation suggests, $\wt V_{\tau}^t$ can be constructed in the form $V_{\tau}^t$ of \eqref{CXWAVEGROUPintro}. 
The first step is the following Lemma, which is proved in \cite[Proposition 4.4]{ZJDG}.

%\begin{rem} 
%As explained in \cite{ZJDG}, $V_{\tau}^t$ is not a unitary group.  $\tilde{V}_{\tau}^t$ is a slight modification, with the same Fourier integral operator
%properties,   which does define a unitary group. 
%\end{rem}

%\begin{prop} \label{TILDEV} $\tilde{V}_{\tau}^t$ is a unitary group   with %%eigenfunctions
%$$\tilde{V}_{\tau}^t P^{\tau} A_{\tau} \phi_j = e^{i t \lambda_j} P^{\tau} A_{\tau} %\phi_j . $$\end{prop}

%\begin{proof}  By Proposition \ref{UNIT}, 
%$$\begin{array}{lll} \tilde{V}_{\tau}^t \tilde{V}_{\tau}^{t *} & = &  P^{\tau} A_{\tau} U^t A_{\tau} P^{\tau *}  P^{\tau} A_{\tau} U^{-%t} A_{\tau} P^{\tau *} \\ &&\\ && 
%= P^{\tau } P^{\tau *} = \Pi_{\tau}: L^2(\partial M_{\tau}) \to L^2(\partial M_{\tau}), \end{array}$$
%so that $\tilde{V}_{\tau}^t$ is unitary. Also, 
%$$\begin{array}{lll} \tilde{V}_{\tau}^t \tilde{V}_{\tau}^{s} & = &  P^{\tau} A_{\tau} U^t A_{\tau} P^{\tau *}  P^{\tau} A_{\tau} U^{s} %A_{\tau} P^{\tau *} \\ &&\\ && 
%= P^{\tau } A_{\tau} U^{t+s}A_{\tau}P^{\tau *}  =  \tilde{V}_{\tau}^{t + s} : L^2(\partial M_{\tau}) \to L^2(\partial M_{\tau}), %\end{array}$$

%Similarly,
%$$\tilde{V}_{\tau}^t P^{\tau} A_{\tau} \phi_j =   P^{\tau} A_{\tau} U^t A_{\tau} P^{\tau *}  P^{\tau} A_{\tau} \phi_j 
%= e^{i t \lambda_j} P^{\tau} A_{\tau} \phi_j. $$
%
%\end{proof} 

\begin{lem}  \label{VWAVEproptilde}
  $\tilde{V}_{\tau}^t$  is a  unitary Fourier integral operator with positive complex phase
 of Hermite type on $
H^2(\partial M_{\tau}) \subset L^2(\partial M_{\tau})$ 
adapted to the graph of the Hamiltonian flow of $\sqrt{\rho}$ on $\Sigma_{\tau}. $

 \end{lem}

We briefly indicate the proof. 

\begin{proof}  
By Proposition 4.3 of   \cite{ZJDG}, $\tilde{V}_{\tau}^t$ is a unitary group   with eigenfunctions
$$\tilde{V}_{\tau}^t P^{\tau} A_{\tau} \phi_j = e^{i t \lambda_j} P^{\tau} A_{\tau} \phi_j . $$
%The complexified eigenfunctions $  \phi_j^{\C}(\zeta)$ are
%not generally orthogonal in $L^2(\partial M_{\tau}, d\mu_{\tau})$. 
Just like $V_{\tau}^t$, $\tilde{V}_{\tau}^t$ is a composition of Fourier integral operators with complex phase, and 
all are associated to canonical graphs and equivalence relations. Moreover all are operators
of Hermite type.  If follows that the composition is transversal,
so that $\tilde{V}_{\tau}^t$  is also a  Fourier integral operator with complex phase and of Hermite type. It follows that $\tilde{V}_{\tau}^t$ is adapted to the graph of  $E  G^t E ^{-1}= \exp t \Xi_{\sqrt{\rho}}$
on  $\Sigma_{\tau}$ (see \eqref{EXP}).

\end{proof}

The next Lemma is   \cite[Proposition 4.5]{ZJDG}. It  shows that    $\tilde{V}_{\tau}^t(\zeta, \bar{\zeta})$  can be constructed
as a unitary group of Toeplitz dynamical operators $V_{\tau}^t$  \eqref{CXWAVEGROUPintro}.
%V^t_{\tau} = \Pi_{\tau} (g_{\tau}^t)^* \sigma_{t, \tau} \Pi_{\tau},
\begin{lem}  \label{VWAVEprop}There exists
a polyhomogeneous pseudo-differential operator  $\sigma_{t \tau}$ on $\partial M_{\tau}$  so that 
\begin{equation} \label{CXWAVEGROUP} 
\;\;\;\; \tilde{V}_{\tau}^t  = \Pi_{\tau} \sqrt{\sigma_{t, \tau}} (g_{\tau}^t)^*\sqrt{\sigma_{t, \tau}} \Pi_{\tau}.
\end{equation}
Thus, $\tilde{V}_{\tau}^t$ is equivalent to  \eqref{CXWAVEGROUPintro}.

 \end{lem}
 We note that $\tilde{V}_{\tau}^t$ only differs from \eqref{CXWAVEGROUPintro} in the definition of its symbol. One can interchange
 the order of $(g_{\tau}^t)^*\sqrt{\sigma_{t, \tau}}$ by translating the symbol.  Since Lemma \ref{VWAVEprop} is proved in \cite{ZJDG},
 we only briefly sketch the proof for the sake of completeness.

\begin{proof}  Each side of each formula is an elliptic  Toeplitz Hermite Fourier integral operator
adapted to the graph of $g_{\tau}^t$ on $\Sigma_{\tau}$.  In the case of  $\tilde{V}_{\tau}^t$ this follows directly from the definitions and
 by the composition theorem for such operators in \cite{BoGu}. In the 
case of $ \Pi_{\tau} g_{\tau}^t \sigma_{t, \tau} \Pi_{\tau}$ it follows similarly from the fact that $\Pi_{\tau}$
is a Toeplitz operator and from the simple composition with pullback by $g_{\tau}^t$.

By Proposition \ref{VWAVEproptilde}, $\tilde{V}_{\tau}^t$ is unitary. Hence its principal
symbol is unitary.  We also have by the composition calculus of Toeplitz symplectic spinor
symbols (see \eqref{SYMBPIT})  that
$$\sigma_{\Pi_{\tau}} \circ \sigma_{\tau, -t} \sigma_{g_{\tau}^{-t} \Pi_{\tau} g_{\tau}^t} \sigma_{t, \tau}
\circ \sigma_{\Pi_{\tau}} = \sigma_{\Pi_{\tau}} \leftrightarrow  |\sigma_{\tau, t} |^2 
\sigma_{\Pi_{\tau}} \circ \sigma_{g_{\tau}^{-t} \Pi_{\tau} g_{\tau}^t} 
\circ \sigma_{\Pi_{\tau}} = \sigma_{\Pi_{\tau}} .$$

Then
$$\sigma_{\Pi_{\tau}} \circ \sigma_{g_{\tau}^{-t} \Pi_{\tau} g_{\tau}^t} 
\circ \sigma_{\Pi_{\tau}}  = |\langle e_{\Lambda_t}, e_{\Lambda} \rangle|^2. $$
It follows that 
$$ |\sigma_{\tau, t}^0 |^2  =  |\langle e_{\Lambda_t}, e_{\Lambda} \rangle|^{-2}. $$
Thus the principal symbol   can only differ from $\sigma_{t, \tau}^0$ \eqref{sigmataut} by a multiplicative
factor of modulus one.   We can choose the factor to make the principal
coincide with the principal symbol of the linearization on the osculating Bargmann-Fock
space in \eqref{METASCHWARTZ} and Proposition \ref{toep-met}. The symbol then equals \eqref{ABCD}.

Using the composition calculus, one can define the lower order terms recursively 
so that $V_{\tau}^t$ and $\tilde{V}_{\tau}^t$ have the same complete symbol, i.e. differ
by a smoothing Toeplitz operator.

 \end{proof}

 %\begin{rem} The fact that $\langle e_{\Lambda_t}, e_{\Lambda} \rangle$ equals \eqref{ABCDintro} and \eqref{ABCD} was
 %first proved in \cite{Z97} in the line bundle setting. \end{rem}

The final Lemma compares   $\tilde{V}_{\tau}^t(\zeta, \bar{\zeta})$  and  $U_{\C}(t + 2i \tau, \zeta, \bar{\zeta})$. 
\begin{lem}  \label{VWAVEproptildeb} (\cite{ZJDG} Proposition 4.4) 
$\tilde{V}_{\tau}^t$ is a  Fourier integral operator with complex phase
 of Hermite type on $
H^2(\partial M_{\tau}) \subset L^2(\partial M_{\tau})$ 
adapted to the graph of the Hamiltonian flow of $\sqrt{\rho}$ on $\Sigma_{\tau} $.
  $\tilde{V}_{\tau}^t(\zeta, \bar{\zeta})$   has the same canonical relation  as $U_{\C}(t + 2i \tau, \zeta, \bar{\zeta})$
and the same principal symbol multiplied by $ |\xi|^{\frac{m-1}{2}}$.
%$$ \sigma_{\tilde{V}_{\tau}^t} = |\xi|^{\frac{m-1}{4}} \sigma_{P^{\tau}} \circ \sigma_{U^t} \circ %\sigma_{P^{\tau *}}. $$ 

 \end{lem}

Indeed, $\tilde{V}_{\tau}^t$ only differs from $U(t + 2 i \tau, \zeta, \bar{\zeta})$
by the insertion of two $A_{\tau}$ factors, and by Lemma \ref{OLD} this only changes the principal
symbol by the factor $ |\xi|^{\frac{m-1}{2}}$.

%\begin{lem}\label{VT} Let $g^t_{\tau} = \exp t \Xi_{\sqrt{\rho}}$ on $\partial M_{\tau}$ and as in
%\eqref{SYMBPIT}, let $\sigma_{t, \tau}$ be a polyhomogeneous pseudo-differential operator
%on $\partial M_{\tau}$ with principal symbol \eqref{sigmataut}
%\begin{equation} \label{sigmataut2} \sigma^0_{\tau, t} =  \langle e_{\Lambda_t}, e_{\Lambda} %\rangle^{-1}. \end{equation}
 %Then
%\begin{equation} \label{CXWAVEGROUP} V^t_{\tau}: = \Pi_{\tau} g_{\tau}^t \sigma_{t, \tau} %\Pi_{\tau},
%\;\;\;\; \tilde{V}_{\tau}^t  = \Pi_{\tau} \sqrt{\sigma_{t, \tau}} g_{\tau}^t \sqrt{\sigma_{t, \tau}} \Pi_{\tau} 
%\end{equation} \end{lem}

 Combining  Lemma \ref{VWAVEproptilde} and Lemma \ref{VWAVEproptildeb} proves \eqref{UandV} and \eqref{UPIFORM}.
 
\subsection{Completion of the proof of Proposition \ref{MAIN}}

By \eqref{SMOOTH}, we see that $\chi * d  P^{\tau}_{[0, \lambda]}$ is the Fourier transform of $U_{\C} (t + 2 i \tau, \zeta, \bar{\zeta})$, weighted
by $\hat{\chi}$. By Lemma \ref{VWAVEproptildeb}, one has 
a  Toeplitz parametrix for $U_{\C}(t + 2 i \tau, \zeta, \bar{\zeta})$ as in Lemma \ref{VWAVEprop}, hence   in the form   \eqref{UPIFORM} but with the symbol of $\wt{V}_{\tau}^t$ multiplied
by $|\xi|^{- \frac{m-1}{2}}$, accounting for the order of the amplitude.

This completes the proof of Proposition \ref{MAIN}.

 \subsection{\label{KtauSECT} Proof of Lemma \ref{L2LEMintro}} 
In this section,
% we deploy the spectral  kernel  \eqref{RENORMWEYL} which related to the  Husimi distributions \eqref{HUSIMI} and
we prove Lemma \ref{L2LEMintro}. That is, we prove
%\begin{defin} 
%\begin{equation} K_{\tau}(\zeta, \bar{\zeta}'): = \sum_{j: \lambda_j \leq \lambda} \frac{\phi_{j}^{\C}(\zeta) \overline{\phi_{j}^{\C}(\zeta')}}{|| \phi_{j}^{\C}||^2_{L^2(\partial %M_{\tau})}} \end{equation}
%\end{defin}
%Define the self-adjoint operator on $L^2(M, dV)$ by 
%\begin{equation} \label{ATAU} A_{\tau} = (P^{\tau *} P^{\tau})^{-\half}. \end{equation}
%\begin{lem} \label{L2LEMPF} 
$||\phi_j^{\C}||^2_{L^2(\partial M_{\tau}}  \simeq  e^{2 \tau \lambda_j}  \lambda_j^{-\frac{m-1}{2}} (1 + O(\lambda_j^{-1})). $ 
%\end{lem}

\begin{proof} We have,
$$\begin{array}{lll} ||\phi_j^{\C}||^2_{L^2(\partial M_{\tau})}  & = &  e^{2 \tau \lambda_j} ||P^{\tau} \phi_j||^2_{L^2(\partial M_{\tau})} \\&&\\
& = & e^{2 \tau \lambda_j} \langle P^{\tau *} P^{\tau} \phi_j, \phi_j \rangle_{L^2(M)}. 
\end{array}$$
\end{proof}

By  Lemma \ref{OLD} (see also \cite[Proposition 3.6]{ZJDG})  $A_{\tau}$   is an elliptic  self-adjoint  pseudo-differential operator of order $ \frac{m-1}{4}$
with principal symbol $|\xi|^{ \frac{m-1}{4}}$. That is, $(P^{\tau *} P^{\tau})$ is a pseudo-differential
operator of order $-\frac{m-1}{2}$ with principal symbol $|\xi|^{- \frac{m-1}{2}}$. Hence,
$$(P^{\tau *} P^{\tau}) = \Delta^{-\frac{m-1}{4}} + R, \;\; R \in \Psi^{-\frac{m-3}{2}}. $$
It follows that
$$\langle P^{\tau *} P^{\tau} \phi_j, \phi_j \rangle_{L^2(M)} = \lambda_j^{-\frac{m-1}{2}} (1 + O(\lambda_j^{-1})).$$
Hence,
$$\begin{array}{lll} ||\phi_j^{\C}||^2_{L^2(\partial M_{\tau})}  & = & e^{2 \tau \lambda_j}  \lambda_j^{-\frac{m-1}{2}} (1 + O(\lambda_j^{-1})).
\end{array}$$

\end{proof}

%\edit{Out of place. Note that $\Pi_{[0, \lambda]}$ is missing. \begin{proof} As in Lemma \ref{UvsVtilde},  
%$$\begin{array}{lll}  P^{\tau} P^{\tau *}  & = &   \sum_{j: \lambda_j \leq \lambda} e^{- 2 \tau \lambda_j} \phi_{j}^{\C}(\zeta) \overline{\phi_{j}^{\C}(\zeta')}
%\\ &&\\
% & = &   \sum_{j: \lambda_j \leq \lambda}   \lambda_j^{-\frac{m-1}{2}} (1 + O(\lambda_j^{-1})) \frac{\phi_{j}^{\C}(\zeta) \overline{\phi_{j}^{\C}(\zeta')}}{{|| \phi_{j}^{\C}||^2_{L^2(\partial M_{\tau})}}}
% \end{array} $$
% It follows that,
 %$$\begin{array}{lll}  P^{\tau} A_{\tau}^{-1} P^{\tau *}   & = &   \sum_{j: \lambda_j \leq \lambda}   (1 + O(\lambda_j^{-1})) \frac{\phi_{j}^{\C}(\zeta) %\overline{\phi_{j}^{\C}(\zeta')}}{{|| \phi_{j}^{\C}||^2_{L^2(\partial M_{\tau})}}}
% \end{array} $$

%\end{proof}
%}

 \section{\label{SCTHSECT} Proof of Theorems \ref{SCALINGTHEO} - \ref{PichilambdaTHEO} }

We now use the  Boutet de Montel-Sj\"ostrand parametrix for $\Pi_{\tau}$ to construct an oscillatory integral parametrix with complex phase
for  $U_{\C}(t + 2 i \tau, \zeta, \bar{\zeta})$ of \eqref{UPIFORM}. 
The  expression
in Corollary  \ref{MAIN1} can be put in an explicit form as an oscillatory integral with
complex phase.

\subsection{An oscillatory integral parametrix for $\chi * d P^{\tau}_{[0, \lambda]}$}
 In this section, we prove,

\begin{prop} \label{SMOOTHCOR1} Define the phase
\begin{equation} \label{CXPHASE1} \Phi(t, \zeta,w, \sigma_1, \sigma _2) = -  t + \sigma_1 \psi(\zeta, w) + \sigma_2 \psi(g_{\tau}^t w, \zeta).  \end{equation}
Let $\chi \in \scal(\R)$ be an even function with $\hat{\rho} \in C_0^{\infty}.$ Then,
%i  t + \sigma_1 (r_{\C} (\zeta,\bar{w}) - \tau)  + \sigma_2 (r_{\C} (g^t w,\bar{\zeta}) %- \tau) . $$
 There exists a semi-classical amplitude
$A_{\lambda}(\zeta, \bar{\zeta}, \sigma_1, \sigma_2, t, w)$ of order $-\frac{m-1}{2}$ such that 
\begin{equation}\label{SMOOTHED1} \begin{array}{l}  \chi * d P_{[0, \lambda]}^{\tau}(\zeta, \bar{\zeta}) =   \sum_j
e^{(- 2 \tau \lambda_j} \chi(\lambda_j - \lambda) |\phi_j^{\C}(\zeta)|^2,
\\ \\ 
=\lambda^{2m}  \int_{\R} \int_{0}^{\infty} \int_0^{\infty} \int_{\partial M_{\tau}} \hat{\chi}(t) e^{i \lambda \Phi(t, \zeta,\bar{\zeta}, w, \sigma_1, \sigma _2) }
A_{\lambda}(\zeta, \bar{\zeta}, \sigma_1, \sigma_2, t, w)
 d\sigma_1 d\sigma_2 d w dt. \end{array} \end{equation}
 The same type of parametrix exists for $\Pi_{\chi, \tau}(\lambda,  \zeta, \bar{\zeta})$ but with an amplitude of order $0$.
\end{prop}

\begin{proof} We  deploy the Boutet de Monvel-Sj\"ostrand
parametrix for $\Pi_{\tau}$ in Section \ref{SZEGOKSECT} to pass from \eqref{UPIFORM} to \eqref{SMOOTHED1}.
We combine   \eqref{SPPROJDAMPED} and  \eqref{CXWVGP} and \eqref{UPIFORM} 
 with  Proposition \ref{MAIN} to obtain,
\begin{equation}\label{SMOOTHEDb} \begin{array}{lll} \chi * d P_{[0, \lambda]}^{\tau}(\zeta, \bar{\zeta})
& = & \int_{\R} \hat{\chi}(t) e^{- i \lambda t } \Pi_{\tau} \hat{\sigma}_{t, \tau} g^{t}_{\tau} \Pi_{\tau}(\zeta, \bar{\zeta}) d t
\\&&\\
& = &  \int_{\R} \int_{0}^{\infty} \int_0^{\infty} \int_{\partial M_{\tau}} \hat{\chi}(t) e^{- i \lambda t }
e^{ i\sigma_1 \psi(\zeta, w) }e^{ i\sigma_2 \psi (g_{\tau}^t w,\zeta) } \\&&\\&&\tilde{\sigma}_{t, \tau}(w, \sigma_1) s(\zeta, \bar{w}, \sigma_1) s(g_{\tau}^t w, \bar{\zeta}, \sigma_2) d\sigma_1 d\sigma_2 d w dt. \end{array} \end{equation}
Thus, the phase is  $\Phi(t, \zeta,w, \sigma_1, \sigma _2) $ as given in  \eqref{CXPHASE1}.
Here, $\tilde{\sigma}_{t, \tau}(w, \sigma_1)$ is a polyhomogeneous function determined by the complete symbol of $\hat{\sigma}_{t, \tau}$ 
in Proposition \ref{MAIN} and with the same principal term $\sigma_{t \tau}^0$,  and $s$ in $|\xi|^{-\frac{m-1}{2}}$ times the amplitude
of the Szeg\"o kernel.
The order of each factor $s$ of the symbol of the \szego\; kernel   is  $m-1$. Changing variables $\sigma_j \to \lambda \sigma_j$ we obtain an oscillatory integral with large parameter $\lambda$,
with an amplitude of order $\lambda^{2 + 2(m -1)} - \frac{m-1}{2}.$ 
and with the complex phase \eqref{CXPHASE1}. 
The factors of $A_{\tau}$ account for the factor $\lambda^{- \frac{m-1}{2}}$.

The only change in the proof for $\Pi_{\chi, \tau}(\lambda,  \zeta, \bar{\zeta})$ is in the order of the amplitude, which now is $0$.
%\begin{equation}\label{SMOOTHEDb1} \begin{array}{lll}  \Pi_{\chi, \tau}(\lambda,  \zeta, \bar{\zeta}) & = & \int_{\R} \hat{\chi}(t) e^{-i \lambda t } \Pi_{\tau} %\hat{\sigma}_{t, \tau} g^{t}_{\tau} \Pi_{\tau}(\zeta, \bar{\zeta}) d t
%\\&&\\
%& = &  \int_{\R} \int_{0}^{\infty} \int_0^{\infty} \int_{\partial M_{\tau}} \hat{\chi}(t) e^{i \lambda t }
%e^{ i\sigma_1 \psi(\zeta, w) }e^{ i\sigma_2 \psi (g_{\tau}^t w,\zeta) } \\&&\\&&\tilde{\sigma}_{t, \tau}(w, \sigma_1) s(\zeta,g_{\tau}^t \bar{w}', \sigma_1) %s(g_{\tau}^t w, \bar{\zeta}, \sigma_2) d\sigma_1 d\sigma_2 d w dt \\&&\\& = &\lambda^{2m}  \int_{\R} \int_{0}^{\infty} \int_0^{\infty} \int_{\partial M_{\tau}} %\hat{\chi}(t) e^{i \lambda \Phi(t, \zeta,\bar{\zeta}, w, \sigma_1, \sigma _2) }
%A_{\lambda}(\zeta, \bar{\zeta}, \sigma_1, \sigma_2, t, w)
 %d\sigma_1 d\sigma_2 d w dt. 

%. \end{array} \end{equation}

%Here, $\tilde{\sigma}_{t, \tau}(w, \sigma_1)$ is a polyhomogeneous function determined by the complete symbol of $\hat{\sigma}_{t, \tau}$ 
%in Proposition \ref{MAIN} and with the same principal term $\sigma_{t \tau}^0$,  and $s$ is the amplitude
%of the Szeg\"o kernel.
%The order of each factor $s$ of the  amplitude is  $m-1$ (similar to the unit bundle of a line bundle over a K\"ahler manifold of dimension 
%$m-1$). 
%Changing variables $\sigma_j \to \lambda \sigma_j$ we obtain an oscillatory integral with large parameter $\lambda$,
%with an amplitude of order $\lambda^{2 + 2(m -1)}$ 
%and with the complex phase $\psi$.

\end{proof}

 \subsection{Proof of Theorem \ref{SCALINGTHEO} for $\wcal_{\tau}(t, \zeta, \bar{\zeta})$}
 
 We now apply a complex stationary phase method to prove Theorem \ref{SCALINGTHEO}.

\subsubsection{\label{CXCRITSEC} Critical set of the complex phase}
The critical set of \eqref{CXPHASE1} with respect to the integration variables is given by the equations:
\begin{equation} \left\{ \begin{array}{l} \psi(\zeta, w) =  \psi(g_{\tau}^t w, \zeta) = 0;
%r_{\C} (\zeta,\bar{w}) =  \tau = r_{\C} (g^t w , \bar{\zeta});
 \\ \\ \sigma_2 d_t\psi(g_{\tau}^t w, \bar{\zeta}) = 1, \\ \\
d_w ( \sigma_1 (\psi (\zeta, w )   + \sigma_2 \psi (g_{\tau}^t w,\zeta) ) = 0

%d_w ( \sigma_1 (r_{\C} (\zeta,\bar{w}) )   + \sigma_2 (r_{\C} (g^t w,\bar{\zeta})  = 0. 
\end{array} \right. \end{equation}

As  in Proposition \ref{CLOSED}, we have:

\begin{lem} \label{PERCRIT1}  Given $\zeta$, the  critical set of \eqref{CXPHASE1} is emtpy
unless $g_{\tau}^t(\zeta) = \zeta$. It then consists of
$$\{(t, \sigma_1, \sigma_2, w) \in \R \times \R_+ \times \R_+ \times \partial M_{\tau} : w = \zeta,   \sigma_1 = \sigma_2 = 1, t = n  T(\zeta)\}. $$
\end{lem}

It follows that for fixed $\zeta$, the values of $t$ for which one has critical points in the support of $\hat{\chi}$
 are $t = 0$ and $t  $ in the period
set $\pcal(\zeta)$, i.e. the set of $t$ so that $g_{\tau}^t(\zeta) = \zeta$.  Thus $t = n T(\zeta)$ for some $n$, where
$T(\zeta)$ is the primitive period.

Indeed, by  \eqref{IM}  the first equation holds if and only if 
\begin{equation}  w = \zeta, \;\;\; g_{\tau}^t w = \zeta \in \partial M_{\tau} \implies g_{\tau}^t \zeta = \zeta.   \end{equation}
We restrict the second equation to $w = \zeta$ and get $\sigma_2 d_t \psi(g_{\tau}^t \zeta, \zeta) = 1$. Since
the period set of $\zeta$ is discrete the left side equals $\sigma_2 d_t \rho(g_{\tau}^t \zeta, \zeta) |_{t = L(\zeta)}$
where $L$ is a value of $t$ so that $g_{\tau}^t \zeta = \zeta$, and then  $\sigma_2 \partial \rho (\zeta) \cdot \frac{d}{dt}
g_{\tau}^t \zeta = \sigma_2 \alpha_{\zeta} (\frac{d}{dt} g_{\tau}^t \zeta) = \sigma_2 = 1. $ Here we use that
$g_{\tau}^t$ is a contact flow for the contact form $ \alpha$.

We then consider the third equation. We may set $t = nT(\zeta)$ and get 
$\sigma_1 d_w \psi(\zeta, w) + \sigma_2 d_w \psi(w, \zeta) = 0$. But $d_w \psi(\zeta, w)|_{w = \zeta}
= \alpha$ by \eqref{alpha}.
If follows
that $\sigma_2 = \sigma_1 = 1$.

\begin{cor} \label{VALUE} The critical  value of phase  \eqref{CXPHASE1} on critical set of Lemma \ref{PERCRIT1}   is given by  $$\Phi(t, w, \sigma_1, \sigma_2; \zeta)  = n T(\zeta),$$
where $T(\zeta)$ is the primitive period of $\zeta$ and $n$ is the iterate number. 
\end{cor}

\subsection{\label{PSLOCAL} Localization to a Phong-Stein leaf}

%\begin{proof}

%Let $\zeta$ be a periodic point of $g^t$ and let $\mcal_{\zeta}$ be the
%leaf \eqref{LEAF} through $\zeta$.
%, i.e. $\{w \in \partial M_{\epsilon}: \Im \psi(z,w) = 0\}.$

In slice-orbit coordinates (Definition \ref{SODEF}) the phase \eqref{CXPHASE1} takes the form
\begin{equation} \label{SOPHASE} \begin{array}{lll} \Psi(t, \zeta,(s,z), \sigma_1, \sigma _2) 
& =  &-  t + \sigma_1 \psi(\zeta, (s, z)) + \sigma_2 \psi(g_{\tau}^t (s,z), \zeta) \\&&\\
& = & -  t + \sigma_1 \psi(\zeta, (s, z)) + \sigma_2 \psi ((s + t,z), \zeta) 
\end{array}  \end{equation}

Fix $\zeta$ and  $\epsilon >0$ and consider the time $\epsilon$ flow-out of the leaf
in both positive and negative time: 
\begin{equation} \label{FLOWOUT} \mcal_{\zeta}(\epsilon) = \bigcup_{|s| \leq \epsilon}
g^s \; \mcal_{\zeta}. \end{equation}
Let $\theta(s)$ be a smooth cutoff in $|t|$ supported in $[-\epsilon, \epsilon]$  which equals one for $|s| \leq \frac{\epsilon}{2}$. 
As above we parametrize a neighborhood by the slice-orbit coordinates $(z, s) \in \mcal_{\zeta}\times [-\epsilon, \epsilon]
\to g^s z$ \eqref{SO}. Denote the volume density on $\partial M_{\tau}$ in slice-orbit coordinates by $J(w,s)$.

\begin{lem} \label{SMOOTHCORa} With the same notation as in Proposition
\ref{SMOOTHCOR1}. Fix $\zeta$. Then, modulo a rapidly decaying error in 
$\lambda$,

\begin{equation}\label{SMOOTHEDc} \begin{array}{lll} 
%\int_{\R} \int_{0}^{\infty} \int_0^{\infty} \int_{\partial M_{\tau}} \hat{\chi}(t) e^{i \lambda \Phi(t, \zeta,\bar{\zeta}, w, \sigma_1, \sigma _2) }
%A_{\lambda}(\zeta, \bar{\zeta}, \sigma_1, \sigma_2, t, w)
% d\sigma_1 d\sigma_2 d w dt
 \chi * d P_{[0, \lambda]}^{\tau}(\zeta, \bar{\zeta})&
\simeq  & \lambda^{2m} \int_{\R} \int_{0}^{\infty} \int_0^{\infty} \int_{\mcal_{\zeta} \times
[-\epsilon, \epsilon]}
 \hat{\chi}(t) e^{i \lambda \Psi(t, \zeta,\bar{\zeta}, g^s z, \sigma_1, \sigma _2) }\\&&\\&&
A_{\lambda}(\zeta, \bar{\zeta}, \sigma_1, \sigma_2, t, g^s z)
 d\sigma_1 d\sigma_2 J(w,s) d z ds dt, \end{array} \end{equation}
 where the Jacobian factor  $J(w,s)$ is the volume density on $\mcal_{\zeta} \times [-\epsilon, \epsilon]$ and where $A_{\lambda}$ has order $\lambda^{-\frac{m-1}{2}}$. 
 Similarly for $  \Pi_{\chi, \tau}(\lambda,  \zeta, \bar{\zeta})$ with the changes in $A_{\lambda}$ mentioned in Proposition \ref{SMOOTHCOR1}.
\end{lem}

\begin{proof}We have merely localized the integral over $\partial M_{\tau}$ to $\mcal_{\zeta} \times [-\epsilon, \epsilon]$. This is possible, by the standard Lemma of Stationary Phase, i.e. the use
of integration by parts to show that the integral is negligible on the complement of 
any neighborhood of the critical point .  It applies since  $\mcal_{\zeta}(\epsilon)$ covers
a neighborhood of the stationary phase point. By the Lemma of stationary
phase the remaining part of the integral is rapidly decaying. \end{proof}

In the next Lemma we apply stationary phase in the variables $(\sigma_1,
\sigma_2, s, t)$ to reduce the integral to a Phong-Stein leaf. This reduction
is reminiscent of the steepest descent method for an oscillatory
integral with complex phase, where  the contour is deformed to one
on which  the imaginary part of the phase  $\Im i \Phi$  equals zero. We do not deform
contours
but use the stationary phase method to obtain the reduction.

Next we evaluate the phase in Lemma \ref{SMOOTHEDcc1} more explicitly. We retain the notation of that Lemma. 

\begin{lem} \label{SMOOTHEDd1} Fix $\zeta \in \pcal$ and let  $T_n(\zeta)$ be the return time to $\mcal_{\zeta}$ of Definition \ref{POINC}, and let $D$ be the diastasis \eqref{CD}. Then there exists a zeroth order amplitude $B_{\lambda}(\cdot, z)$ supported in an arbitrarily small neighborhood
of $z = \zeta$, so that, 
modulo rapidly decreasing funtions of $\lambda$,
\begin{equation} \begin{array}{lll}
% \int_{\R} \int_{0}^{\infty} \int_0^{\infty} \int_{\partial M_{\tau}} \hat{\chi}(t) e^{i \lambda \Phi(t, \zeta,\bar{\zeta}, w, \sigma_1, \sigma _2) }
%A_{\lambda}(\zeta, \bar{\zeta}, \sigma_1, \sigma_2, t, w)
 %d\sigma_1 d\sigma_2 d w dt \\ \\ 

  \chi * d P_{[0, \lambda]}^{\tau}(\zeta, \bar{\zeta})
& \simeq & 
%(\lambda \tau)^{-1}

 \lambda^{2m - 2 - \frac{m-1}{2} } \sum_{n=0}^{\infty} \hat{\chi}(T_n(\zeta)) e^{-i \lambda T_n(\zeta)} \int_{\mcal_{\zeta}}
 e^{ \lambda (D(\zeta, z) + D(g_{\tau}^{T_n(\zeta) } z, \zeta)) }\\ &&\\ &&
B_{\lambda}(\zeta, \bar{\zeta},1, 1, T_n(\zeta), z)
 J(z,0) d z.\end{array} \end{equation}
 %where, as a function of $z$,  where $B_{\lambda}(\cdot, z)$ is a $0$th order amplitude   supported in an arbitrarily small neighborhood of $z= \zeta$.
 
A similar formula holds for $\Pi_{\chi, \tau}(\lambda)$ \eqref{chilambda}  but without the factor of $\lambda^{-\frac{m-1}{2}}$.

\end{lem}

\begin{proof}

%\begin{lem} \label{SMOOTHCORc1} With the same notation as in Proposition
%\ref{SMOOTHCOR1}. Fix $\zeta$. Then, modulo a rapidly decaying error in 
%$\lambda$,

At a critical point of the phase, $\psi(g_{\tau}^t z, \zeta) =0$ and so $\Im \psi(g_{\tau}^t z, \zeta) = 0$ and also $\psi(\zeta, z) = 0$ so that
$w \in \mcal_{\zeta}$ and also $g_{\tau}^t z \in \mcal_{\zeta}$.  This forces $\zeta = z, g_{\tau} ^t z = \zeta$,
and again we see that  $t = T_n(\zeta)$ for some $ n \in \Z$.  We then introduce Phong-Stein coordinates
using a local inverse to \eqref{SO} for $t$ near$T_n(\zeta)$ in the sense of the equivalence
relation \eqref{EQUIV},
and  consider critical points in $t, s, \sigma_1, \sigma_2$.  If we take $\partial_s$ and consider only the imaginary part of the critical point equation,
we get $\sigma_1 = \sigma_2$. If we take $\partial_t$ and consider only the imaginary part we get $\sigma_2 = 1$. If
we take $D_{\sigma_1}$ and consider only the imaginary part  we get $s = 0$; for $D_{\sigma_2}$ we get $t - T_n(\zeta)= -s = 0$.
Thus, the only critical point occurs at $(s, t, \sigma_1, \sigma_2) = (0,T_n(\zeta),1,1)$.

The Hessian of the  phase in $(\sigma_1, \sigma_2, s, t)$  at
this critical point is
%\begin{equation} \label{HESS}\begin{pmatrix} &  \sigma_1 & \sigma_2 & s & t & \\&&&&\\
%\sigma_1 & 0 & 0   & 1 & 0 \\ &&&&\\
%\sigma_2 & 0 & 0  & 1 & 1 \\ &&&&\\
%s & 1 & 1 &  2i & -2i  \\ &&&&\\
%t & 0 & 1 & -2i & 2 \end{pmatrix}. \end{equation}

\begin{equation} \label{HESS}\begin{pmatrix} &  \sigma_1 & \sigma_2 & s & t & \\&&&&\\
\sigma_1 & 0 & 0   & - 4 i  \tau & 0 \\ &&&&\\
\sigma_2 & 0 & 0  & 4 i \tau  &4 i \tau\\ &&&&\\
s & -4i \tau &4 i \tau  &  * & *  \\ &&&&\\
t & 0 & 4 \tau & *& * \end{pmatrix}. \end{equation}
Here, $T = T_n(\zeta) = n T(\zeta)$ for some $n$.
It is not necessary to calculate the lower right block. 
 Note that by \eqref{alpha},
$$\left\{\begin{array}{l}\partial_s \; \psi(\zeta, (s, z)) |_{s = 0, 
z = \zeta}= \partial_s \; \psi(\zeta, (s, \zeta)) |_{s = 0} = - 4i   \alpha_{\zeta}(\Xi_{\sqrt{\rho}}) = - 4i
\tau  \\ \\ \partial_s  \psi(g_{\tau}^{t + s} z, \zeta)  |_{s = 0,z = \zeta,  t = T_n(\zeta)} =  \partial_s  \psi(g_{\tau}^{T_n(\zeta) + s} \zeta, \zeta)  |_{s = 0} =4 i\tau. \end{array} \right.$$
The $\partial_t$ derivative is the same as the $\partial_s$ derivative in the last line.

Since the determinant is non-singular, we may apply stationary phase in the variables
$(\sigma_1, \sigma_2, s, t)$. Since the  only stationary phase points are
$\sigma_1 = \sigma_2= 1, s = 0, t = T_n(z)$, the integral localizes to  $\mcal_{\zeta}$ and has the phase $  \Psi |_{\sigma_1 = \sigma_2= 1, s = 0, t = T_n(z)}$. Applying stationary phase in $(\sigma_1,
\sigma_2, s, t)$  concludes the proof of the Lemma. Since Hessian is non-degenerate, the integration produces a factor of $\lambda^{-2}$.
We then get,
\begin{equation}\label{SMOOTHEDcc1} \begin{array}{l}

% \int_{\R} \int_{0}^{\infty} \int_0^{\infty} \int_{\partial M_{\tau}} \hat{\chi}(t) e^{i \lambda \Phi(t, \zeta,\bar{\zeta}, w, \sigma_1, \sigma _2) }
%A_{\lambda}(\zeta, \bar{\zeta}, \sigma_1, \sigma_2, t, w)
% d\sigma_1 d\sigma_2 d w dt

% \Pi_{\chi, \tau}(\lambda,  \zeta, \bar{\zeta}) 
\chi * d P_{[0, \lambda]}^{\tau}(\zeta, \bar{\zeta})\\ \\ 
\simeq  \lambda^{2m - 2 - \frac{m-1}{2} } \sum_{n=0}^{\infty} \hat{\chi}(T_n(\zeta))\int_{\mcal_{\zeta}}
 e^{i \lambda \wt \Psi(T_n(\zeta), \zeta,\bar{\zeta},  z, 1, 1) }
B_{\lambda}(\zeta, \bar{\zeta},1, 1, T_n(z), z)
 J(z,0) d z, \end{array} \end{equation}
where $B_{\lambda}(\cdot, z)$ is a $0$th order amplitude   supported in an arbitrarily small neighborhood of $z= \zeta$. The phase $\wt \Psi$ and amplitude
$B_{\lambda}$ are  obtained from the amplitude $A_{\lambda}$ and phase $\Psi$ by the standard stationary
phase method. The factor $\lambda^{-\frac{m-1}{2}}$ in $A_{\lambda}$ is moved outside the integral.  $\wt \Psi = \Psi |_{\sigma_1 = \sigma_2= 1, s = 0, t = T_n(z)}$ is the critical value of $\Psi$. 

%\end{lem}

Since $\Im \psi(\zeta, z) = \Im \psi(g_{\tau}^{T_n(\zeta) } z, \zeta) = 0$ on $\mcal_{\zeta}$, the   value of the phase on the critical set is
\begin{equation} \label{CXPHASEcr}\begin{array}{lll} \Phi(t, \zeta,z, 1,1) & = & -   T_n(\zeta) + \Re  \psi(\zeta, z) + \Re  \psi(g_{\tau}^{T_n(\zeta) } z, \zeta)
\\ &&\\ & = & - T_n(\zeta) + D(\zeta, z) + D(g_{\tau}^{T_n(\zeta) } z, \zeta), \end{array} \end{equation}
where $D(z, w)$ is the Calabi diastasis \eqref{CD}.

The same form is valid for $\Pi_{\chi, \tau}(\lambda,  \zeta, \bar{\zeta}) 
$ with a different amplitude and without the factor $\lambda^{-\frac{m-1}{2}}$. This completes the proof of the Lemma.

\end{proof}

\subsection{\label{SCALINGSECT} Proof of Theorem \ref{SCALINGTHEO} }

To complete the proof of Theorem \ref{SCALINGTHEO}, we need to calculate the integral over $\mcal_{\zeta}$ asymptotically using 
the method of complex stationary phase (steepest descent).
We next deploy the  estimate on the phase from \eqref{IM}. 

\begin{lem} \label{GLOBALDECAY}  In \kahler normal (or, Heisenberg normal)  coordinates centered at $\zeta$, there exists $\epsilon > 0$ and  there exists a positive constant $C_{\epsilon} > 0$ so that for $z \in B_{\epsilon}(\zeta)$, we have
$$D(\zeta, z) + D(g_{\tau}^{T_n(\zeta) } z, \zeta)) \geq C_{\epsilon} (|z|^2 + |g_{\tau}^{T_n(\zeta)} (z) - \zeta|^2). $$

\end{lem}

\begin{proof} We use Lemma \ref{CDKNC}.  In the estimate \eqref{IM} it suffices to choose $\epsilon $ so that the $O(|z-w|^3)$ term of \eqref{IM} is smaller than
the term $C |z-w|^2$. That is, if the implicit constant in  $O(|z-w|^3)$ is $D |z-w|^3$ then we choose $\epsilon $ so that $C \geq D \epsilon$.

 \end{proof}

To  complete the proof of Theorem \ref{SCALINGTHEO}, we study the $n$th term in the sum in Corollary \ref{SMOOTHEDd1}.   The integral of concern is,
\begin{equation} \label{INT1}S_{\lambda, n}(\zeta):= \int_{\mcal_{\zeta}}
 e^{ \lambda (D(\zeta, z) + D(g_{\tau}^{T_n(\zeta) } z, \zeta)) }
B_{\lambda}(\zeta, \bar{\zeta},1, 1, T_n(\zeta), z)
 J(z; \zeta) d z.\end{equation}
As noted in Corollary \ref{SMOOTHEDd1}, the integral of that Lemma may be cut off to the ball $B_{\epsilon(0)} \subset \C^{m-1}$ of Lemma
\ref{GLOBALDECAY}  around $\zeta$ ($=0$ in the
Heisenberg coordinates) without
changing the asymptotics, since the phase has no critical points in this case.
For simplicity of notation, we  retain the notation $z \in \C^{m-1}$
for the Heisenberg coordinates and the previous notation for the disastasis and the geodesic flow, without explicitly putting in the conjugation
to Heisenberg coordinates. Thus,
\begin{equation} \label{INT2}S_{\lambda, n}(\zeta)=   \int_{B_{\epsilon}(0) }
 e^{ \lambda \Psi_n(z; \zeta)  } A(\lambda, z; \zeta)
d z.\end{equation}
 where $A(\lambda, z; \zeta) = B_{\lambda}(\zeta, \bar{\zeta},1, 1, T_n(\zeta), z)
 J(z,0) $ is supported in the ball of radius $\epsilon$ around $0$, and where 
 $$\Psi_n(z; \zeta) : =  (D(\zeta, z) + D(g_{\tau}^{T_n(\zeta) } z, \zeta)). $$
 
 \subsubsection{Proof by steepest descents}
 
 The phase $\Psi$ is positive and real, so we may apply the method of (real) steepest descent on $\C^{m-1}$. 
 
 \begin{proof} The steepest descent point
 is the minimum of the phase. We note that the phase vanishes at  $z = \zeta$, and since the phase is positive this is a global minimum 
 of the phase.   By \ref{GLOBALDECAY}, $D(z,\zeta) \not=0$ for $z \not= \zeta$, at least when $w \in B_{\epsilon}(\zeta)$, and note also that $g_{\tau}^{T_n(\zeta) } z\not=  \zeta$ if $z \not= \zeta$. Hence,
 the phase does not vanish at any  other 
$z \in B_{\epsilon}(\zeta)$. 

By \cite[Theorem 7.1]{HoI}, 
 \eqref{INT2}  admits a complete asymptotic expansion of the type,
$$S_{\lambda, n}(\zeta)
  \simeq \lambda^{-(m-1)} \frac{1}{\sqrt{\det \rm{Hess} \Psi_{\zeta}}} \sum_k \lambda^{-k} L_k (A e^{i \lambda R_3}). $$

We thus need to compute $\det \rm{Hess} \Psi_{\zeta}$. For this purpose, 
we give the  Taylor expansion to order $2$ of the phase   in a Heisenberg coordinate chart  $(z, t, \rho) = (z, \Re z_{n + 1}, \rho(w)) \in \partial M_{\tau}  \times \R_+$  centered at $\zeta = 0$,  of Section \ref{HEISCOORDSECT}.
 We fix $\rho = \tau^2$. To compute the Taylor expansion in the local coordinates, we write 
$$z = \zeta + \frac{1}{\sqrt{\lambda}} u. $$
Here,  $\zeta =0$ in the coordinates,  but   (by abuse of notation) we leave it in to  remember  that the coordinates
are centered at $\zeta$. The factor $ \frac{1}{\sqrt{\lambda}}$ is natural in the \kahler\; or Heisenberg scaling. It is not necessary to put in
this factor but it helps to keep track of the order of the terms.

 Let $L(\zeta)$ denote the Levi form on $T_{\zeta} \mcal_{\zeta}$. We now  prove the following,
\begin{lem}\label{PHASESCALE}  Let  $D_{\zeta} g^{T_n(\zeta) }=: S:
%Let $D\gtc{T_{\zeta}}(0) =: S: 
T_{\zeta} \mcal_{\zeta} \to T_{\zeta} \mcal_{\zeta} \simeq \C^{d-1}$.  With the above notation, 
$$\begin{array}{lll} \lambda (D(\zeta,  \zeta + \frac{u}{\sqrt{\lambda}} ) + D(g_{\tau}^{n T(\zeta) } ( \zeta + \frac{u}{\sqrt{\lambda}} ),  \zeta))  & = &  |u|^2_{L(\zeta)} + |S(u)|^2_{L(\zeta)} + O((\lambda^{-1}|u|)^3) . \end{array}$$
More generally, by polarization, 
$$\begin{array}{lll} \lambda (D(\zeta + \frac{v}{\sqrt{\lambda}},  \zeta + \frac{u}{\sqrt{\lambda}} ) + D(g_{\tau}^{T_n(\zeta) } ( \zeta + \frac{1}{\sqrt{\lambda}} u),  \zeta))  & = &  \langle u, v \rangle_{L(\zeta)} + \langle S u, S v \rangle_{L(\zeta)} + O((\lambda^{-1}|(u,v)|)^3) . \end{array}$$

\end{lem}

\begin{proof} By Lemma \ref{CDKNC} and  by \eqref{TAYLOR},
\begin{equation} \label{TAYLOR2} D(x,y) =  L_{\rho}(x - y) + O(|x-y|)^3 = \sum_{j,k = 1}^{m-1} \frac{\partial^2 \rho}{\partial z_j \partial \bar{z}_k }(\zeta) (x_j - y_j)
(\bar{x}_k - \bar{y}_k) + O^3 \end{equation}
so if $x = \zeta, y = \zeta + \frac{1}{\sqrt{\lambda}} u$, 
$$D(\zeta, \zeta + \frac{1}{\sqrt{\lambda}} u) = \frac{1}{\lambda}  \sum_{j,k = 1}^{m-1} \frac{\partial^2 \rho}{\partial z_j \partial \bar{z}_k }(\zeta) u_j \overline{u}_k +  R_3(\frac{u}{\sqrt{\lambda}}),$$
where $R_3$ is the third order Taylor remainder satisfying,
$$R_3(\frac{u}{\sqrt{\lambda}}) = O((\frac{|u|}{\lambda})^3).$$ 
Further, $$g_{\tau}^{T_n(\zeta) } ( \zeta + \frac{1}{\sqrt{\lambda}} u) = \zeta + \frac{1}{\sqrt{\lambda}} S u + O(\frac{|u|}{\lambda}), $$
so 
$$\begin{array}{lll} D_{\zeta}(g_{\tau}^{n T(\zeta) } ( \zeta + \frac{1}{\sqrt{\lambda}} u),  \zeta)) & = &  D(\zeta + \frac{1}{\sqrt{\lambda}} S u + O(\frac{|u|}{\lambda}), \zeta) \\&&\\& = &   \sum_{j,k = 1}^{m-1} \frac{\partial^2 \rho}{\partial z_j \partial \bar{z}_k }(\zeta) (Su)_j
\overline{(Su)}_k+ R_3'(\frac{u}{\sqrt{\lambda}}),\end{array}$$
where $R_3'$ is the third order Taylor remainder satisfying,
$$R_3'(\frac{u}{\sqrt{\lambda}}) = O((\frac{|u|}{\lambda})^3).$$

\end{proof}

Since $
    |z|_{h(\zeta)}^2 = \omega_{\zeta}\left( J_{\zeta} (z), z \right),$
it follows from  Lemma \ref{PHASESCALE} and the fact that $S$ is symplectic that 
$$\rm{Hess} \Psi_{\zeta}(0) = \omega_{\zeta} \left(J_{\zeta} \cdot, \cdot \right) +  \omega_{\zeta} \left(S  J_{\zeta} S^{-1} \cdot, \cdot \right),$$
so that, by \eqref{ETAJS} and \eqref{ID2}, and then by \eqref{DAUB} and Lemma \ref{DAULEM},
$$\begin{array}{lll} \det \rm{Hess} \Psi_{\zeta}(0) & = & \det \left[ \omega_{\zeta} \left(J_{\zeta} \cdot, \cdot \right) +  \omega_{\zeta} \left(S  J_{\zeta} S^{-1}\cdot, \cdot \right)
\right] \\&&\\
%&=& c_d \eta_{J,S} \\ &&\\
& = & 
    \ip{W_{J_{\zeta}}\left( D \gtc{T_{\zeta}} \right) \left( \Omega_{J_{\zeta}} \right)}{\Omega_{J_{\zeta}}} 
    %= \eta_{J, S} \ip{\Omega_{S J_{\zeta}S^{ - 1}}}{\Omega_{J_{\zeta}}}
\end{array}.$$

It follows that
$$S_{\lambda, n}(\zeta) 
  \simeq \lambda^{-(m-1)}   e^{i \lambda n T(\zeta)}  \left( \ip{W_{J_{\zeta}}\left( D \gtc{T_{\zeta}} \right) \left( \Omega_{J_{\zeta}} \right)}{\Omega_{J_{\zeta}}}\right). $$
  Combining with Lemma \ref{SMOOTHEDd1}, and using
\eqref{INFSERIES} and   \eqref{ABCD} (see also \eqref{ABCDintro}), we get 
  $$\begin{array}{lll}
   \chi * d P_{[0, \lambda]}^{\tau}(\zeta, \bar{\zeta})
&  = & \lambda^{\frac{m-1}{2}}  \sum_n  \hat{\chi}(n T(\zeta)) e^{i \lambda n T(\zeta)}  \gcal_n(\zeta) + O(\lambda^{\frac{m-3}{2}})
\\&&\\&=&  C_m  \lambda^{\frac{m-1}{2}} +
% \\&&\\
%& + &
C_m'    \lambda^{\frac{m-1}{2}} \Re \sum_{n = 1}^{\infty} \hat{\chi}(n T(\zeta))  e^{ - i\lambda n T(\zeta)} \gcal_n(\zeta)
 \end{array}, $$
and obtain the result stated in 
Theorem \ref{SCALINGTHEO}. Indeed, there is a complete asymptotic expansion with the given principal terms.
Here, $C_m$ denotes a dimensional constant.
The proof for $\Pi_{\lambda, \psi}$ is essentially the same, but without the factor of $\lambda^{\frac{m-1}{2}}$ throughout.

\end{proof}
\begin{rem} The power of  $\lambda$ results from 
$\lambda^{2m - 2 - \frac{m-1}{2} } \lambda^{-(m-1)} = \lambda^{\frac{m-1}{2}}$. 

In the case of $\Pi_{\chi, \tau}(\lambda)$ \eqref{chilambda}, the factor $\lambda^{-\frac{m-1}{2}}$ does not arise, and the
order is $\lambda^{m-1}$. As mentioned above,
 $ \chi* d P^{\tau}_{[0, \lambda]} (\zeta, \bar{\zeta})$
\eqref{SMOOTH} 
and $\Pi_{\chi, \tau}(\zeta, \bar{\zeta})$  \eqref{chilambda} are both  dynamical Toeplitz operators, 
with  the same canonical relation.    They only differ in their  amplitudes.

\end{rem}

\section{\label{TWOTERM} Tauberian arguments:  Completion of proof of Theorem \ref{SHORTINTSa}}

To complete the proof of Theorem \ref{SHORTINTSa}, we apply the   Tauberian argument of \cite{SV} (pages 225-6).   See also
 Theorem \ref{TTT} (cf.
\cite{SV}, Appendix B (Theorem B.4.1)); the statement and proof are reviewed in \S \ref{TAUB}.

%The term $n = 0$ equals $C_m \lambda^{m + 1}$. $\clubsuit$ NO $ \clubsuit$For the remaining %terms, 
%we  integrate the sum in  $\lambda$ to obtain the $Q_{\zeta}(\lambda)$ function \eqref{Q}.
%\begin{equation} \label{Q} Q_{\zeta}(\lambda) =  \sum_{n = 1} ^{\infty}  \frac{\sin  \lambda n %T(\zeta)} { \lambda n T(\zeta)}
% \langle e_{\Lambda_{n T(\zeta)}}, e_{\Lambda} \rangle^{-1}. \end{equation}

 We let  $N_{2; \tau, \zeta} (\lambda) = P^{\tau}_{[0, \lambda]}
(\zeta, \bar{\zeta})$, and  also define $$N_{1; \tau, \zeta} (\lambda) = 1_{[0,
\infty]} \; 
 \left(  C_m'
   \lambda^{^{ \frac{m + 1}{2}}}   + Q_{\zeta}(\lambda)
 \lambda^{ \frac{m - 1}{2}} \right), $$ where $1_{[0,
\infty]}$ is the indicator function. Both $N_j = N_{j; \tau, \zeta} (\lambda)$ are monotone
non-decreasing functions of polynomial growth which vanish for
$\lambda \leq 0$ and both satisfy the estimate of the Tauberian theorem \ref{ET}.
It follows from   \cite{SV} (p. 198 and p. 225)  that if the support of $\hat{\chi}$ contains only $\{0\}$ among the 
critical points of \eqref{SMOOTHED1} and $\hat{\chi} \equiv 1$ in some smaller neighborhood of $0$, then
$$N_{2; \tau, \zeta} * \chi (\lambda) = N_{1; \tau, \zeta} (\lambda) * \psi (\lambda) + O \left(\frac{\lambda}{\sqrt{\rho}(\zeta)}
  \right)^{\frac{m-1}{2} - 1}. $$

Moreover,  
%\eqref{ST1} and 
Theorem \ref{SCALINGTHEO} shows that if $\hat{\gamma} \in C_0^{\infty}$ has support in $(0,
  \infty)$,
  $$\gamma * dN_{2; \tau, \zeta}  (\lambda) = \gamma * d N_{1; \tau, \zeta} (\lambda)
  (\lambda) + O \left(\frac{\lambda}{\sqrt{\rho}(\zeta)}
  \right)^{\frac{m-1}{2} - 1}.$$
   By Theorem \ref{TTT} (see Theorem B.4 of \cite{SV}),
  \begin{equation} \begin{array}{lll} N_{1; \tau, \zeta}  (\lambda - o(1)) - o(\left(\frac{\lambda}{\sqrt{\rho}(\zeta)}
  \right)^{\frac{m-1}{2} - 1}) & \leq &  N_{2, \tau, \zeta} (\lambda) \\&&\\&&\leq
N_{1; \tau, \zeta} (\lambda + o(1)) + o (\left(\frac{\lambda}{\sqrt{\rho}(\zeta)}
  \right)^{\frac{m-1}{2} - 1}). \end{array} \end{equation}
Here, $o(1)$ is a positive monotone function  which tends to zero as $\lambda \to \infty$. This proves \eqref{INEQ}. When $Q_{\zeta}$ is uniformly continuous, we may simplify $  Q_{\zeta}(\lambda - o(\lambda))$ to $Q_{\zeta}(\lambda)$ and absorb the
remainder into  the error term. 

%It is clear that $Q_{\zeta} (\lambda)$ is uniformly continuous and
%even $C^1$  since its Fourier coefficients decay like
%$n^{-\frac{m-1}{2} - 1}. $ The first statement of the Proposition
%follows. The second statement is proved in the same way.

\section{\label{QzetaSECT} Jump behavior:  Proof of  Proposition \ref{CLJUMPPROP}  and Proposition \ref{QzetaPROP}}

In this section, we study the continuity or jumps of  $Q_{\zeta}(\lambda)$ and prove  \ref{CLJUMPPROP} and  Proposition \ref{QzetaPROP}.

\subsection{Classical dynamical approach:  Proof of  Proposition \ref{CLJUMPPROP}}

The natural approach to studying the right side of \eqref{gcalndef} is to put $S$ into a normal form. It is tempting to put $S$ into standard additive (resp. multiplicative) Jordan normal form as a sum (resp.  product) of
a semi-simple matrix and a nilpotent (resp. unipotent) matrix, but these matrices need not be symplectic in general and we cannot quantize the components
by the metaplectic representation, and the symplectic Jordan  normal forms are rather lengthy and complicated  (see \cite{Gutt}). Even when the matrix is put into normal form one must still extract its holomorphic part
$P = P_J S P_J$. The map $S \to P_J S P_J$ does not behave well with respect to multiplicative normal forms, and its  determinant $\det P_J S P_J$   does
not behave well with respect to additive normal forms.  For that reason, we study only the open dense set of semi-simple symplectic matrices (see 
\cite{Gutt} for the proof of density). 

We refer to Section \ref{METASECT} for background on the symplectic Linear algebra.  The article \cite{MU00} contains a list of all possible
symplectic normal forms of matrices in $\rm{Sp}(2, \R)$. Since semi-simple symplectic matrices are direct sums of symplectic matrices
in $\rm{Sp}(1, \R), \rm{Sp}(2, \R)$,  the  list in \cite{MU00} contains the building blocks (under symplectic  direct sum) of  all normal forms relevant to this article.

The proof of Proposition  \ref{CLJUMPPROP} consists of a series of Lemmas dealing with the cases of (i) elliptic symplectic matrices;
(ii) positive definite symmetric symplectic matrices (hyperbolic blocks), and (iii) semi-simple normal symplectic matrices with complex eigenvalues
(sometimes called loxodromic blocks). Since loxodromic blocks are not so familiar, we recall their definition: If one of $\pm \alpha \pm i \beta$ is an eigenvalue of $\acal$ for some $\alpha, \beta > 0$, then there exists
$S \in Sp(2, \R)$ so that 
\begin{equation} \label{LOX} S^{-1} \acal S = \begin{pmatrix} A_4 & 0 \\ & \\ 0 & D_4 \end{pmatrix}, 
\;\; \rm{where}\;\; A_4  = \begin{pmatrix} \alpha & \beta \\ & \\ - \beta  & \alpha \end{pmatrix}, \;\; D_4 =\begin{pmatrix} -\alpha & \beta \\ & \\ - \beta  & -\alpha \end{pmatrix}. \end{equation}
The associated symplectic linear map acts by complex dilation, i.e. a mixture of rotation and real dilation.

\subsubsection{\label{ELL} Elliptic symplectic matrices}

The following Lemma proves one direction of Proposition \ref{CLJUMPPROP}. 
\begin{lem}\label{ELLPROP}  If $S_{\zeta}$ is elliptic, and if $\det P_J S P_J = e^{i s_0}$, then \begin{equation} \label{INFSERIES2} Q_{\zeta}(\lambda) = \sum_{n\not=0}
\frac{e^{i  \lambda n T(\zeta)}}{n T(\zeta)} e^{i n s_0}   = \{ s_0 +  \lambda  T(\zeta) - \pi\}_{2 \pi}. \end{equation} \end{lem}

\begin{proof} 

If $S_{\zeta}$ is elliptic, and  if the polar part is the identity, i.e. $S_{\zeta}\in U(m)$. Then by Lemma \ref{DetFORM},  $ \det P_J S_{\zeta} P_J  = e^{i s_0} $
for some $s_0 \in [0,2 \pi]$. It follows from \eqref{INFSERIES} that \eqref{INFSERIES2}  holds.

%\begin{equation} \label{INFSERIES2} Q_{\zeta}(\lambda) = \sum_{n\not=0}
%\frac{e^{i  \lambda n T(\zeta)}}{n T(\zeta)} e^{i n s_0}   = \{ s_0 +  \lambda  T(\zeta) - \pi\}. \end{equation}

\end{proof}

 Next we consider non-elliptic semi-simple $S_{\zeta}$. Thus, we assume that $S_{\zeta}$ is diagonalizable over $\C$ but that
 it has some eigenvalues which are not of modulus $1$.
\subsubsection{\label{PDSS} Positive symmetric symplectic matrices}

In this section, we assume that $S$ is a symmetric symplectic matrix, and, slightly more, that all of its eigenvalues are positive. 
As discussed in Section \ref{METASECT}, if $S$ is symmetric, there   exists $U \in U(m)$ conjugating $S$ to its diagonal form. Proposition \ref{HYPLEM} is an immediate consequence of the following  Proposition,  adapted from \cite{ZZ18} in the line bundle setting.

\begin{prop}\label{PROPME}  If $S$ is positive definite symmetric symplectic, 
%with invariant vector $\xi$ and $\alpha = P_J \xi$,  
	and if the spectrum of $S$ is $\{e^{\lambda_j}, e^{-\lambda_j} \}_{j=1}^n$ with 
	$\lambda_j \geq 0$  then $$
	%\left\{ \begin{array}{ll} (i) & [P_J S P_J]^{-1} \alpha = \alpha, \\ & \\
	%(ii)&
	 \det P_J S P_J|_{T^{1,0}_0\R^{2n}} = \prod_{j=1}^n [\cosh \lambda_j].
	 % \end{array} \right. 
	 $$
	 Consequently, $Q_{\zeta}(\lambda)$ is uniformly continuous.
	 
	 \end{prop}
\begin{proof}
	
	The proof is through a series of Lemmas from \cite{ZZ18}; since the proofs are short, we repeat them here.
	
	\begin{lem} If $S$ is  positive definite symplectic, then 
		$$P_J S P_J =\half P_J (S + S^{-1}) = \half (S + S^{-1}) P_J$$

	\end{lem}
	
	\begin{proof} If $S$ is positive definite symmetric, then  $S J = J S^{-1}$. Hence,
		$$\begin{array}{lll} P_J S P_J & = & \frac{1}{4} (I - i J)  S (I - i J) = \frac{1}{4}[ S - i JS - i SJ - J SJ ]  \\&&\\& = & 
		 \frac{1}{4} [S + S^{-1}] - \frac{i}{4}J [S + S^{-1}] = \frac{1}{4}\left( (S + S^{-1}) - i J (S + S^{-1})\right)
		= \frac{1}{2} P_J (S + S^{-1}).
		\end{array}$$
	
		Also,
		$J (S + S^{-1} ) = JS + SJ = (S^{-1} + S)J$,
		so that $P_J (S + S^{-1}) = (S + S^{-1}) P_J. $
	\end{proof}

	\begin{lem} \label{EIGLEM} Let $S$ be  positive definite symmetric symplectic and  $e_j$ be eigenvectors of $S$ for eigenvalues $\lambda_1, \dots, \lambda_n$.
		Consider the basis $P_J e_k$ of $H^{1,0}_J$. Then  $$[P_J S P_J] P_J e_k =  \cosh(\lambda_j) P_J e_k, $$
		and $[P_J S P_J]^{-1} = P_J [S + S^{-1}]^{-1} P_J. $ 
	\end{lem}
	
	\begin{proof}
		
		Follows from the previous Lemma  and the fact that $(S + S^{-1})$ commutes with $P_J$:
		\[ [P_J S P_J] P_J e_k = \half P_J(S+S^{-1}) e_k = \half (e^{\lambda_j}+ e^{-\lambda_j}) P_J e_k = \cosh(\lambda_j)P_J e_k. \]
		 \end{proof}

	%Statement (i) of the Proposition follows from the fact that
	%$$[P_J S P_J] \alpha = \half (1+1)\alpha
	%= \alpha. $$
	
The determinant formula of Proposition \ref{PROPME}  follows from the fact that the eigenvalues of $P_J S P_J$ are 
	$\cosh \lambda_j$ by Lemma \ref{EIGLEM}. 
When the closed geodesic through $\zeta \in \partial M_{\tau}$ is positive definite symplectic (or, real hyperbolic for short), then the determinant
formula obviously implies that the Fourier series \eqref{Q} for  $Q_{\zeta}(\lambda)$ converges absolutely to uniformly continuous function, proving Proposition
\ref{HYPLEM}.

\end{proof}

\subsubsection{Semi-simple  symplectic matrices with complex eigenvalues}
We recall from Section \ref{METASECT}  that if
$S  \in \rm{Sp}(m, \R) $ is a normal symplectic matrix, its  polar decomposition $S = U \hat{P}_S$ satisfies  $ \hat{P}_SU = U \hat{P}_S$,
with $\hat{P} = (S^*S)^{\half}$. Since $U$ is unitary, $P_J U \hat{P} P_J = (P_J U P_J)(P_J \hat{P} P_J)$ and
$\det P_J S P_J =  \det (P_J U P_J) \det (P_J \hat{P} P_J)$.  Proposition \ref{PROPME} applies to $\det (P_J \hat{P} P_J)$, and 
obviously $|\det P_J S^n P_J| \leq |\det (P_J \hat{P}^n P_J)| \leq \cosh (n \lambda_j)$. Hence, $Q_{\zeta}(\lambda)$ is an absolutely
and uniformly convergent Fourier series.

 \subsubsection{Completion of the proof}
 
 If $S_{\zeta}$ is semi-simple symplectic, it is a direct sum of the three cases above and the coefficients \eqref{gcalndef}  are  products
 of those in the three cases. Only one block needs to be non-elliptic for the series to converge absolutely and uniformly.
 
 This completes the proof of Proposition  \ref{CLJUMPPROP}. 
  
 \subsection{Quantum approach: Proof of Proposition \ref{QzetaPROP}}

 In this section, we  use the real Schr\"odinger  representation to prove Proposition \ref{QzetaPROP}. 
 
 The first statement (1) follows immediately from the definition of \eqref{SPMEAS} and the fact that
 $$\frac{1}{2 i T(\zeta)} \int_0^{2 \pi}  \sum_{n \not=0}  e^{i \lambda n T(\zeta)}{n} e^{i n \theta} d \theta = \frac{1}{ T(\zeta)} \int_0^{2 \pi}   \{\theta  + \lambda T(\zeta) -\pi\}_{2 \pi} d\mu_{\zeta} $$
 
 The second statement follows since $ \{\theta  + \lambda T(\zeta) -\pi\}_{2 \pi} $ is continuous in $\lambda$ on $[0, 2 \pi]$, so if $ d\mu_{\zeta} $
 is absolutely continuous it is uniformly continuous in $\lambda$. 
 
 On the other hand, if $d\mu_{\zeta}$ has an atom at $e^{s_0}$ then, $Q_{\zeta}(\lambda)$ has the jump of $\mu_{\zeta}(e^{i s_0}) 
  \{\theta_0  + \lambda T(\zeta) -\pi\}_{2 \pi}$. Hence, (2) is proved.
  
  The atoms of the spectral measure of any unitary operator $W$ on a Hilbert space occur at its eigenvalues. Hence, (3) is true. Moreover,
  by definition of the spectral measure with respect to a normalized eigenvector $\Omega$, the $\mu_{\zeta}(s_0) $ equals $|\langle v_0, \Omega\rangle |^2$, proving (4). 
  
  This completes the proof of Proposition \ref{QzetaPROP}. 
  
  \begin{rem} The argument in \cite{SV} is to use the spectral theorem to write,   $$Q_{\zeta}(\lambda) = \frac{1}{T(\zeta)}  \sum_{\ell}  \{ s_{\ell}  +\lambda T(\zeta) -\pi\}_{2 \pi} ||\pi_{\Omega_{\zeta}} v_{\ell} ||^2,$$  
\end{rem}

\subsubsection{Spectral theory of metaplectic operators} The above proof is rather abstract. To apply it to the metaplectic operators,
we need to determine when they have eigenvalues (i.e. $L^2$ eigenfunctions) to produce atoms in the spectral measures, and moreover
we need to determine the projections $\pi_{\Omega_{\zeta}} v_{\ell}$.  Although quadratic
Hamiltonians and their propagators are very classical, the  only reference
we are aware of regarding  their spectral decomposition is in  \cite{MU96,MU00}. 
In \cite[Proposition 3.1]{MU00} the H\"ormander  classification of symplectic normal forms of quadratic Hamiltonians on $T^*\R^2$ is recalled,
 and in \cite[Proposition 3.2]{MU00} the corresponding Schr\"odinger operators are listed. In addition to harmonic oscillators such as  $ -\Delta + ||x||^2$, there are  magnetic Schr\"odinger operators with potential  such as  $(i D_{x_1} - b x_2)^2 + (i D_{x_2} + b x_1)^2 + \langle K x, x \rangle$ where $K$ is a real symmetric matrix; here $D_{x_j} = \frac{1}{i} \frac{\partial}{\partial x_j}$.  More generally,
 a magnetic Schr\"odinger operator with potential has the form,  $\sum_{j=1}^n (i D_{x_j} - (Bx)_j)^2 + \langle K x, x \rangle$ where $B $ is 
 a real skew-symmetric $n \times n$ matrix.

 In \cite[Proposition 3.3, Theorem 3.5, Theorem 4.7]{MU00}, the nature of the spectrum
 is determined for four types of quadratic Hamiltonians on $\R^2$.   
Of these, only the harmonic oscillator has eigenvalues. The others have absolutely
 continuous spectrum. These results imply that for $\dim M = 2,3$,  $Q_{\zeta}(\lambda)$ is uniformly continuous in all cases except 
 for elliptic closed geodesics. Thus, Proposition \ref{QzetaPROP} is proved when $\dim M =2,3$. The nature of the spectrum in the special
 case of magnetic Schr\"odinger operators in higher dimension is studied in \cite{MU00} 
 
 The cases of $T^*\R, T^*\R^2$ are fundamental by the normal form theorems above, since by \eqref{WDECOMP}  in the semi-simple case every quadratic Hamiltonian
 is a symplectic direct sum of model quadratic Hamilltonians on $T^*\R^2$ or $T^*\R$. 
 Clearly, it would be very laborious to determine the nature of the spectrum by this method for every possible type of symplectic linear transformation,
 or every possible quadratic in the symplectic classification. Hence we strict again to normal symplectic transformations. We now present some
 simple proofs of Proposition \ref{QzetaPROP} using this decomposition in the semi-simple case. 
 
 %\begin{conj} The metaplectic  quantization $W(S)$ of $S \in \rm{Sp}(m,\R)$ has an eigenvalue $e^{i \theta}$ if and only if $S \in U(m)$. 
 %\end{conj}

 The case of general harmonic oscillators can be reduced to the one-dimensional case, as the next Lemma shows.
 
  \begin{lem} Let $S = e^B$ be positive definite symplectic,  where $B \in {\mathfrak s}{\mathfrak p}(m, \R)$ and $B^* = B$.  Then,  the Weyl quantization $W(S)$ of $S$ \eqref{eta} (with the standard $J$)  has an $L^2$ eigenvector $v$
  if and only if $v$ is an $L^2$ eigenvector of  $W(B)$ if and only if the Weyl quantizations $W(B_j)$ of the diagonal blocks $B_j$  of $B$ have definite Weyl symbols. \end{lem}
  
  \begin{proof} As reviewed above,
$B$ is unitarily conjugate in $Sp(m, \R)$ to a  diagonal matrix.  Its Weyl quantization $W(B)$  is then a sum of squares of vector fields $B_j$, and the symbol
is a quadratic form in $x, \xi$ which is a  sum of squares $c_j x_j^2 + d_j \xi_j^2$ . If the coefficients $c_j, d_j$  are all positive, then $|\sigma_B(x, \xi)| \to \infty$ as $(x, \xi) \to \infty$,
and $W(B)$ has discrete spectrum. If any coefficient is negative, then it has continuous spectrum. The generalized eigenfunctions are tensor
products $v_1 \otimes v_2 \otimes \cdots \otimes v_m$ in the tensor decomposition $L^2(\R^m) = \bigotimes_{j=1}^m L^2(\R_{x_j})$,  where the jth component $v_j \in L^2(\R)$. In order that $v$ be an $L^2$ eigenfunction it is necessary and sufficient that $v_j$ be an $L^2$ eigenfunction of $B_j$ for all $j$.

For one-dimensional symmetric quadratic Schr\"odinger operators, it is known that the spectrum is discrete in the definite case and continuous
in the indefinite case.

\end{proof}

More generally, we have
 
 \begin{prop} Suppose that  $\zeta \in \partial M_{\tau}$ is a periodic point such that $D_{\zeta} g^{T(\zeta)}_{\tau}$  is a normal symplectic matrix that lies in the image of the exponential map and whose polar part $(S^*S)^{\half}$ has at least one positive eigenvalue, 
 Then $Q_{\zeta}(\lambda)$ is uniformly continuous as long as $\zeta$ is not an elliptic closed geodesic. 
   \end{prop}
 
 \begin{proof}
 Its eigenvalues come in 4-tuples $\lambda, \bar{\lambda}, \lambda^{-1},
 \bar{\lambda}^{-1}$, though it may happen that $\lambda = \bar{\lambda}$ or $\lambda = \lambda^{-1}$. Using \eqref{WDECOMP},  the generator of  $W(S)$  is  a symplectic  direct sum of the $2$ and $4$ dimensional cases studied in \cite{MU00}.  In order that the generator  have eigenvalues, it is necessary that every
 factor has an eigenvalue. But only  harmonic oscillators (or their opposites) in dimension $2, 4$ have eigenvalues. %\edit{What about $H_5$ in \cite[Theorem 4.7]{MU00}}

We may also prove the statement using polar decomposition.  Given $S \in Sp(m, \R)$, with $[S^*, S] = 0$, we get  $W(S) = W(U) W(P)$, with $W(U), W(P)$ unitary, and $[W(U), W(P)] = 0$; here $W = W_{J_0}$ (the standard complex structure). Hence,
 $L^2$ eigenfunctions are sums of joint eigenfunctions of $W(U), W(P)$, i.e. $W(S)$ has an eigenvalue $e^{i s}$ if and only if  there exists $v \in L^2(\R^m)$ such that $W(U) v = e^{i \theta} v$, 
 $WP) v = e^{i \tau} v$ with $e^{i \theta } e^{i \tau} = e^{is}$.  
If $U$ is unitary, then $U = e^{i H} $ where $H \in {\mathfrak s}{\mathfrak p}(n, \R)$ and $H^* = B$.  $U$ has an $L^2$ eigenvector $v$
  if and only if $v$ is an $L^2$ eigenvector of  $H$ if and only if the diagonal blocks $H_j$  of $B$ are definite. 
  
  It follows that the spectrum of $W_{J_{\zeta}}(S_{\zeta})$ has no eigenvalues  unless $\zeta$ is elliptic when $S_{\zeta}$ is non-degenerate
  and semi-simple. This can also be read off \cite[Proposition 3.3]{MU00}. The loxodromic case is, $$P_{\acal_4} = - \alpha(x_1 \frac{\partial}{\partial x_1} + x_2 \frac{\partial}{\partial x_2}) + \frac{\beta}{2 \pi i} (x_1 \frac{\partial}{\partial x_2} - x_2
\frac{\partial}{\partial x_1}), $$ in the notation 
  of that article (see \eqref{LOX} for the classical matrices), and it is proved there to have absolutely continuous spectrum.  The only operators
  with discrete spectrum in \cite[Proposition 3.3]{MU00} is the  harmonic oscillator. 

\end{proof}

\begin{rem} There are further cases in \cite[Theorem 3.3, Theorem 4.2]{MU00} which
  are  either degenerate or not semi-simple, and which can have dense pure point spectrum or eigenvalues of infinite multiplicity. We are
  not considering them here, for the sake of brevity, but the same methods apply to them. \end{rem}

\section{\label{SPHEREZOLL} Spheres and  Zoll manifolds}
In this section and the next Section \ref{GBSECT}, we exhibit extremals for the sup norm in the complex domain
in the case of spheres, and then prove Theorem \ref{Zoll}. In particular, the results show that the upper bound of  Theorem \ref{PWintro} 
is sharp.

\subsection{\label{SPHERESECT} Spheres}

We now prove that the sup norm bounds are sharp by showing that they are obtained for analytic continuations
of highest weight spherical harmonics.  The eigenspaces $\hcal_N$ of the Laplacian on the standard sphere $\Ss^m$
are the same as the spaces of spherical harmonics of degree $N$, i.e. restrictions of homogeneous harmonic
polynomials of degree $N$ to the surface of $\Ss^m$. We assume a basic familiarity with spherical harmonics
in what follows, and 
refer to \cite{SoZ} for background and references. Many calculations with Poisson transforms and complexified spectral
projections on  spheres can be
found in \cite{L80, G84}; the complexification of $\Ss^m$ is the homogeneous cone $z \cdot z = 0$ rather than the
actual complexification $z \cdot z =1$. In particular, in \cite[Section 6]{G84}, the $L^2$ norms in the real and complex domains
are compared. 

Since the $\Delta$ commutes with the $SO(m+1)$ action on $\Ss^m$, the Poisson transform $P^{\tau}$ conjugates the
$SO(m+1)$ action on $\Ss^m$ and on $\partial \Ss_{\tau}^m$. In particular, the operator $A$ of Definition \ref{VINTROalt} is 
a function of $\Delta$, hence is a scalar on each $\hcal_N$.

In the real domain, as stated in \eqref{REALSUP}, an $L^2$ normalized eigenfunction of a  compact
Riemannian manifold $(M,g)$  has sup-norm at
most $C \lambda^{\frac{m-1}{2}}$, and by Theorem \ref{PWintro}, an $L^2$ normalized Husimi distribution \eqref{HUSIMI} on $\partial M_{\tau}$
has sup-norm at most $C \lambda^{\frac{m-1}{2}}$ where $m = \dim M.$  The real sup norm bound is attained
by zonal spherics on $\Ss^m$. The Husimi sup norm bound is attained by analytic continuations of highest
weight spherical harmonics. Since the explicit formula become complicated for $m >2$, we illustrate
the results only on $\Ss^2$.

\subsubsection{Highest weight spherical harmonics on $\Ss^2$}
  
Highest weight spherical harmonics on $\Ss^2$ of degree $N$ are denoted by $Y_N^N$ and are ``Gaussian beams'' along the equator; see Section \ref{GBSECT} for
general Gaussian beams. In this section, we consider the $L^2$ norm and sup norm of the analytic continuation of $Y_N^N$ to $\Ss_{\tau}^2$. The $L^2$ norm
comparison may also be found in \cite[(5.9)]{L80} and \cite[Section 6]{G84}.

We recall  that in dimension 2, the normalized spherical harmonics are defined by
$$Y_N^m(\theta, \phi) = \sqrt{(2 N + 1) \frac{(N - m)!}{(N + m)!}} P^N_m(\cos \phi) e^{ im \theta},$$
where
$$P^m_{N} (\cos \phi) = \frac{1}{2 \pi} \int_0^{ 2 \pi} (i \sin \phi \cos \theta + \cos \phi)^N
e^{- i m \theta} d \theta $$
is an associated Legendre polynomial.

Up to a constant normalizing factor, the highest weight spherical harmonic  $Y_N^N$ is the restriction of the homogeneous harmonic polynomial
$(x_1 + i x_2)^N$ to $\Ss^2$. It  is
independent of $x_3$ and is a holomorphic function of $x_1 + i x_2$.
We claim that $||(x + i y)^N||_{L^2(S^2)} \sim N^{-1/4}. $ Indeed
we compute it using Gaussian integrals:
$$\begin{array}{l} \int_{\R^3} (x^2 + y^2)^{N} e^{-(x^2 + y^2 +
z^2)} dx dy dz \\ \\ = ||(x + i y)^N||_{L^2(S^2)}^2  \int_0^{\infty} r^{2N}
e^{-
r^2} r^2 dr, \\ \\
% \int_{\R^3} (x^2 + y^2)^{k} e^{-(x^2 + y^2 +
%z^2)}\\ \\  = \int_{\R^2} (x^2 + y^2)^{k} e^{-(x^2 + y^2)} \\ \\
% =   \int_0^{\infty} r^{2k} e^{-
%r^2} r dr, \\ \\
\implies  ||(x + i y)^N||_{L^2(S^2)}^2 = \frac{\Gamma(N +
1)}{\Gamma(N + \frac{3}{2})} \sim C_0 N^{-1/2}. \end{array}$$
Therefore  the normalized highest weight spherical harmonics or Gaussian beams
are  $Y_N^N \simeq C_0 N^{1/4} (x + i y)^N. $  It achieves its $L^{\infty}$ norm at $(1,0, 0)$ where it has
size $N^{1/4}. $

The analytic continuation of $Y^N_N$ to $\Ss^2_{\C} = \{(z_1, z_2 z_3)  \in \C^3: z_1^2 + z^2_2 + z_3^2 = 1\}$ is given in the usual holomorphic coordinates
on $\C^3$ by  $C_0 N^{1/4} (z_1 + i z_2)^N. $  The calculation of its $L^2$ norm on $\partial \Ss^2_{\tau}$ is lengthy,
so we opt for a simpler approach using Fermi normal coordinates. The calculation is valid in all dimensions.

 \begin{lem} \label{GBLEM} Highest weight spherical harmonics on the sphere $\Ss^m$  achieve the maximal sup norm bound of Theorem \ref{PWintro} 
 and Corollary \ref{SUPNORMCOR}.\end{lem}

 \begin{proof}
 Gaussian beams (highest weight spherical harmonics)  may be constructed in Fermi normal
coordinates $(s, y)$  along a closed geodesic $\gamma$ in the form $N^{(m-1)/4} e^{i N s} e^{- N y^2/2}; $ see Section 
\ref{GBSECT} for a detailed construction.
Here, $s$ is arc-length along $\gamma$ and $y$ are exponential normal coordinates on the normal bundle. The
factor $N^{(m-1)/4}$ is due to the $L^2$ normalization since the integral of  $e^{- N y^2}$ over $\R^{m-1}$
equals $c_m N^{- (m-1)/2}$ up to a dimensional constant $c_m$. The sup norm is achieved along the
complexified equator. 

We then complexify $s \to s + i \sigma, y \to y + i \eta \in \partial \Ss^m_{\tau}$,   to  get
$$Y_N^N(s + i \sigma, y + i \eta) = N^{(m-1)/4} e^{iN (s + i \sigma)} e^{- N (y + i \eta)^2},$$
where the $\sqrt{\rho}(s + i \sigma, y \to y + i \eta) = \tau$. A point of this kind is  $y = \eta= 0$ and
$\sigma = - \tau$, since the equator is isometrically imbedded in $\Ss^m$ and the tube function of the complexified
geodesic equals the restriction of the tube function of $\Ss^m_{\C}$. At this point, 
$$|Y_N^N(s + i \sigma, 0) |= N^{(m-1)/4} |=  N^{(m-1)/4} e^{N \tau}, $$
and we see that the  sup norm   bound of Theorem \ref{PWintro}  is attained. \end{proof}

The analytic continuation could also be calculated by analytically continuing the oscillatory integral formula, given by
$$P_N^{N} (\cos \phi) ={2 \pi} \int_0^{ 2 \pi} (i \sin \phi \cos \theta + \cos \phi)^N
e^{- i N \theta} d \theta $$  where $\phi$ is now complex, with     complex phase $\log (i \sin \phi \cos \theta + \cos \phi)
 - i  \theta. $ 
 %When $m = N$, the critical point equation is,
%$$\begin{array}{lll} d_{\theta} (i \sin \phi \cos \theta + \cos \phi) = i  && \iff - i \sin \phi \sin \theta = i  (i \sin \phi \cos \theta + %\cos \phi) \\ &&\\
%& \iff & \sin \phi\;  e^{ i \theta} = - i \cos \phi \iff i \tan \phi =  e^{- i \theta}. \end{array}$$
%This defines a unique critical point $e^{i \theta_{\phi}} \in S^1_{\C}$ if we deform $S^1$ with $S^1_{\C}$. At this critical
%point, the modulus of the integrand equals $|\sin \phi \sin \theta_{\phi}) (\tan \phi)|^N.$
 %Write $\phi = u + i v$, so that $\sin \phi = \sin u \cos i v + \sin i v \cos u, \cos \phi = \cos u \cos i v - \sin u \sin iv. $.
 %Thus, the critical point equations read,
% $$-i (\sin u \cos i v + \sin i v \cos u) \sin \theta = -( \sin u \cos i v + \sin i v \cos u) \cos \theta + i   (\cos u \cos i v - \sin u \sin %iv).$$
 %The real, resp. imaginary parts of the critical point  equation are, $$\begin{array}{l} \sinh v \cos u \sin \theta = - \sin u \cosh %%v \cos \theta + \sin u \sinh v \\ \\
 %- \sin u \cosh v \sin \theta= - \sinh v \cos u \cos \theta + \cos u \cosh v.   \end{array}$$
 Recall that $x = \sin \phi \cos \theta, y = \sin \phi \sin \theta$. Hence,  $$Y^N_N(\theta, \phi) = C_N (\cos \theta \sin \phi + i \sin \theta \sin \phi)^N
= C_N e^{i N \theta} (\sin \phi)^N. $$
This formula can be used to  calculate that $C_N = \sqrt{N +1}$, but we omit this classical calculation.  
%Indeed, by Gradshteyn-Ryzhik (3.62) (Page 369),
%$$\int_{\Ss^2} |e^{i N \theta} (\sin \phi)^N|^2 \sin \phi d \phi d \theta =  \int_{0}^{\pi} \sin^{2N+1} \phi  d \phi = 2 \frac{(2N)!!}{(2N+1)!!}.$$
%This ratio $P_N^{-1}$ satisfies $\sqrt{\pi (N+ \frac{1}{4}} \leq P_N^{-1} \leq \sqrt{\pi(N+\frac{4}{\pi} -1}. $
Then we analytically continue to get 
$$(Y_N^N)^{\C}(\theta + i p_{\theta}, \phi + i p_{\phi}) = \sqrt{N +\half} e^{i N (\theta + i p_{\theta})} (\cos (\phi + i p_{\phi}))^N. $$
On the set $p_{\theta} = - \tau, \phi = \pi/2, p_{\phi} = 0$ we find that 
$$|(Y_N^N)^{\C}(\theta + i p_{\theta}, \phi + i p_{\phi}) | \simeq N^{1/4}  e^{N \tau}. $$

\subsubsection{\label{ZONAL} Coherent states in the complex domain }

Next we use the relation between coherent states  and orthogonal projections to calculate the $L^2$ norm and $L^{\infty}$ norm of coherent states on
spheres of general dimensions.

In the real domain, the spectral projections $\Pi_N: L^2(\Ss^m) \to \hcal_N$
onto the space of spherical harmonics of degree $N$ commute with the action of $SO(m+1)$.  Let  $Y^{\vec m}_N$ denote  the joint eigenfunctions $Y^{\vec m}_N$ of $\Delta$ and of the maximal torus of $SO(m+1)$. They  are orthogonal for different
joint eigenvalues  of $\Delta$ and of the maximal torus of $SO(m+1)$, and the  kernel  \begin{equation} \label{PINKER} \Pi_N(x, y) = \sum_{\vec m} Y^{\vec m}_N(x) \overline{Y^{\vec m}_N(y)} \end{equation}
of $\Pi_N$  
%defined by
%$\Pi_N f(x) = \int_{\Ss^m} \Pi_N (x, y) f(y) dS(y), $ therefore
 satisfies
$$\Pi_N(g x, g y) = \Pi_N(x, y), \;\; g \in SO(m + 1).$$ 
Here  $dS$ is the standard surface measure. Hence $\Pi_N(x, x)$ is a constant independent of $x$. 
For each $y$, $\Pi_N(\cdot, y) $ is spherical harmonic of degree $N$ with   $L^2$
norm squared,
$$||\Pi_N(\cdot, y)||_{L^2}^2 = \int_{\Ss^m} \Pi_N(x, y) \Pi_N(y, x)
dS(x) = \Pi_N(y, y). $$
 Its integral
is $\dim \hcal_N$, 
 hence, $\Pi_N(y, y) = \frac{1}{Vol(\Ss^m)} \dim
\hcal_N.$ Hence the normalized projection 
\begin{equation} \label{YN0DEF} Y_N^0(x) = \frac{\Pi_N(x, y_0) \sqrt{Vol(\Ss^m)}}{\sqrt{\dim
\hcal_N}} \end{equation} 
kernel with `peak' at
$y_0$ achieves the maximum possible sup norm of $\sqrt{\dim \hcal_N}$. Moreover, since $\Pi_N$ is
the orthogonal projection, a standard argument using the reproducing property and the Schwartz inequality
shows that $Y_N^0(y_0)$ is maximal among all $L^2$-normalized spherical harmonics of degree $N$. We note that if $y_0$ is the fixed
point of the $S^1$ action (or, in general dimensions, the maximal torus action), then $Y_N^m(x_0) = 0$ for $m \not= 0$ and the identity above
is obvious.

The zonal spherical harmonic also admits the Legendre representation,
$$Y_N^0(\theta, \phi) = \sqrt{(2 N + 1) } P^N_0(\cos \phi),$$
where
$$P^0_{N} (\cos \phi) = \frac{1}{2 \pi} \int_0^{ 2 \pi} (i \sin \phi \cos \theta + \cos \phi)^N
d \theta. $$
If we analytically continue $\phi$ to $S^1_{\C}$, we obtain an oscillatory integral with complex phase
$\log  (i \sin \phi \cos \theta + \cos \phi)$. It has a critical point when, and only when,  $\sin \theta =0$. The stationary phase expansion brings
in the additional factor of $N^{-\half}$, explaining why the complexified  zonal harmonic is not an extremal. 
Note that the analytic continuation of \eqref{YN0DEFC} to $M_{\tau}$ is,
 \begin{equation} \label{YN0DEFC} (Y_N^0)^{\C} (z) = \frac{\Pi_N^{\C}(z, y_0) \sqrt{Vol(\Ss^m)}}{\sqrt{\dim
\hcal_N}},  \end{equation}

However we now eschew oscillatory integrals to work with projection kernels in order to identify the extremals. 
Denote by $\hcal_N^{\C}$ the holomorphic extensions of the spherical harmonics of degree $N$. For each
$\tau$ the restrictions of the harmonics to $\partial \Ss^m_{\tau}$  is a space $\hcal_N^{\C}(\tau)$  of CR holomorphic functions,
and it is easy to see that the joint eigenfunctions $Y^{\vec m}_N$ of $\Delta$ and of the maximal torus of $SO(m+1)$ are orthogonal for different
joint eigenvalues. 
We denote by \begin{equation} \label{PINKERC} \Pi_N^{\C}(z,w) = \sum_{\vec m} (Y^{\vec m}_N(z))^{\C} \overline{(Y^{\vec m}_N)^{\C} (w)} \end{equation}
  the analytic extension of $\Pi_N$. 
We denote by $\Pi_N^{\tau}(z, w)$  the restriction
of the kernel to $z, w \in \partial \Ss^m_{\tau}$. Using the natural complex conjugation on $\Ss^m_{\C}$ we also consider the
kernel $\Pi_N^{\tau}(z, \bar{w})$, which is holomorphic in $z$ and anti-holomorphic in $w$.

\begin{defin} \label{CSTDEF} Given $\tau > 0$ and $w \in \partial M_{\tau}$, we   define the `coherent state' centered at $w$ by,
 \begin{equation} \label{CST} \Phi_N^w(z) = \frac{\Pi_N^{\tau}(z,w)}{||\Pi_N^{\tau}(\cdot, w) ||_{L^2(\partial M_{\tau})}}, \;\;\; z, w \in \partial \Ss^m_{\tau}\end{equation}
 where
 $$ ||\Pi_N^{\tau}(\cdot, w) ||^2 _{L^2(\partial M_{\tau})}  = \sum_{\vec m} ||Y_N^{\vec m}||^2_{L^2(\partial M_{\tau})}. $$
 \end{defin} 
 For each $w$, the  coherent state \eqref{CST}  is an element of $\hcal_N^{\C}$, but is a scalar multiple of \eqref{PINKERC} and is not the analytic continuation of 
 \eqref{YN0DEF}.

\begin{prop}\label{ZONALPROP} Norm-squares of  coherent states of  $\Ss^m$ attain the  asymptotically maximal sup norm of \eqref{BADSUP}. 
\end{prop}

\begin{proof}

%This constant is
%of course, $$\Pi_N^{\tau}(\zeta, \bar{\zeta}) = \frac{1}{\rm{Vol}(\partial \Ss^m_{\tau})} \int_{\partial \Ss^m_{\tau}} \Pi_N(\zeta, \bar{\zeta}) dV_{\tau} (\zeta).$$
 $\Pi_N^{\tau}(z, \bar{w})$ is the orthogonal sum of $Y^{\vec m}_N \otimes \overline{Y^{\vec m}_N}$, and  is not normalized to be  an orthogonal projection,  so we cannot immediately apply the argument in the real domain 
to find the value on the diagonal, nor can we immediately conclude that either defines  an extremal for the sup norm (when
properly normalized). But we observe that, by orthogonality of the terms, 
$$\Pi_N^{\tau} Y_N^{\vec m} = ||Y_N^{\vec m}||^2_{L^2(\partial M_{\tau})} Y_N^{\vec m}. $$

 %We now specialize to $m=2$ for notational simplicity; the analogous arguments work in all dimensions.
%In dimension $2$, the standard $Y^N_m$ are orthogonal as $m$ varies in dimension 2 (and similarly
%in higher dimensions for joint eigenfunctions of $\Delta$ and a maximal torus). 
We  introduce the orthonormal basis
$$\tilde{Y}_N^{\vec m} = \frac{(Y_N^{\vec m})^{\C}}{||(Y_N^{\vec m})^{\C}||_{L^2(\partial M_{\tau})}}. $$
The orthogonal projection on by $\hcal_N^{\C}(\epsilon)$ is then
$$\tilde{\Pi}_N(z, w) = \sum_m \tilde{Y}_N^{\vec m} (z)  \overline{\tilde{Y}_N^{\vec m} (w)}, $$
and, as for any reproducing kernel,  $$||\tilde{\Pi}_N(\cdot, w) ||^2_{L^2(\partial M_{\tau})} =  \wt{\Pi}(w,w) = \frac{\dim \hcal_N}{\rm{Vol}(\partial M_{\tau})}$$ 
For fixed $w$ this $z \to \tilde{\Pi}_N(z, w) $ defines an element of $\hcal_N^{\C}(\epsilon)$, and we define a variant  of the coherent
state centered at $w$ by, \begin{equation} \label{CST2} \tilde{\Phi}_N^w(z) : = \frac{\tilde{\Pi}_N(z, w) }{||\tilde{\Pi}_N(, w) ||_{L^2(\partial M_{\tau})}}
= \sqrt{\rm{Vol}(\partial M_{\tau})}  \frac{\tilde{\Pi}_N(z, w) }{\sqrt{\dim \hcal_N}}. \end{equation}
\begin{lem} For any $w \in \partial M_{\tau}$, \eqref{CST2} achieves the sup norm bound \eqref{BADSUP}  of Theorem \ref{PWintro} \end{lem}

\begin{proof}

As in the real case,  \eqref{CST2} is the evaluation functional on $\hcal_N^{\C}(\tau)$. By the standard argument, the reproducing kernel achieves the extremal $L^2$-normalized element of $\hcal_N^{\C}(\tau)$ for the sup norm.  Namely, for any $s_N \in \hcal_N^{\C} (\tau)$,
$$|s_N(z) |= | \langle s_N, \wt \Phi_N^{z} \rangle | \leq ||s_N||_{L^2} ||\wt \Phi_N^{z}||, $$
with equality if $s_N = \wt \Phi_N^{z}. $

 Both of the kernels $\Pi_N^{\tau}(z,w)$, resp. $\wt \Pi_N^{\tau}(z,w)$,  are invariant
under the diagonal action of $SO(m+1)$.   Indeed, by analytic continuation from the real domain,
also have $\Pi_N^{\C}(g z, g w) = \Pi_N^{\C}(z, w)$ for all $z, w$.
 The group $SO(m + 1)$ acts transitively on $S^* \Ss^m$
and hence on $\partial \Ss^m_{\tau}$ (for any $\tau > 0$). It is also a holomorphic action on $\Ss^m_{\C}$.
It follows that $\Pi_N{\tau} (\zeta, \bar{\zeta})$  is constant as $\zeta$ varies. 
Since the orthogonal projection commutes with $SO(m+1)$, we also have  $\wt{\Pi}_N(g z, gw) = \wt{\Pi}_N(z,w)$. This implies that its $L^2$ norm equals
$\frac{ \sqrt{\dim \hcal_N}}{\sqrt{\rm{Vol}(\partial M_{\tau})}}.$

\end{proof}

To complete the proof we use Lemma \ref{L2LEMintro} to compare \eqref{CST} and \eqref{CST2}, or equivalently the kernels $\Pi_N^{\C}(z,w)$ and $\wt{\Pi}_N(z,w)$. By this Lemma, $||(Y_N^{\vec m})^{\C}||^2_{L^2(\partial M_{\tau})} \simeq e^{2 \tau \lambda}  \lambda^{-\frac{m-1}{2}} (1 + O(\lambda^{-1})), $ uniformly
in $\vec m$. Hence,
$$||\Pi_N^{\tau}(, w) ||^2_{L^2(\partial M_{\epsilon})} = \sum_{\vec m} ||(Y_N^{\vec m})^{\C}||^2_{L^2(\partial M_{\tau})} \simeq  e^{2 \tau \lambda}  \lambda^{-\frac{m-1}{2}} \frac{\dim \hcal_N}{\rm{Vol}(\partial M_{\tau})}  (1 + O(\lambda^{-1}).$$
Also, using Lemma \ref{L2LEMintro} term by term,
$$\wt \Pi_N^{\tau}(z,w) = e^{-2  \tau \lambda} \lambda^{- \frac{m-1}{2}} (1 + o(1)) \Pi_N(z,w). $$
The constant is canceled when we divide each side by its $L^2$ norm, 
By the definitions  \eqref{CST} and \eqref{CST2}, it follows  that, modulo terms of one lower order in $N$,
$$\begin{array}{lll}  \Phi_N^w(z) & \simeq & 
% \frac{\Pi_N^{\tau}(z,w)}{||\Pi_N^{\tau}(\cdot, w) ||_{L^2(\partial M_{\tau}})}  \simeq e^{- \tau \lambda}  \lambda^{\frac{m-1}{4}} \sqrt{\rm{Vol}(\partial M_{\tau}) } \frac{\Pi_N^{\tau}(z,w)}{\sqrt{\dim \hcal_N}}\\ &&\\
\wt \Phi_N^w(z),
%& \simeq & \frac{1}{\sqrt{\rm{Vol}(\partial M_{\tau})} } \frac{\tilde{\Pi}_N(z, w) }{\sqrt{\dim \hcal_N}}
%=  \frac{e^{-\tau \lambda} \lambda^{\frac{m-1}{4}}}{\sqrt{ \dim \hcal_N}} \sqrt{\rm{Vol}(\partial M_{\tau})} \Pi_N^{\tau}(z, w).

\end{array}$$

completing the proof of the Proposition.

\end{proof}

\begin{rem} 
Note  that the analytic continuation \eqref{YN0DEFC}  of  \eqref{YN0DEF},
with $z \in \partial M_{\tau}, y_0 \in M$, is analytically continued only in the first variable. Although  \eqref{PINKERC} and \eqref{CST} are complex analogues
of \eqref{YN0DEF}, there does not exist a fixed point of the torus action in $\partial M_{\tau}$, so \eqref{CST} does not have a single term (as \eqref{YN0DEF}
does). This explains why Proposition \ref{ZONALPROP} 
does not assert that  \eqref{YN0DEFC} is an extremal. \end{rem}

\subsection{\label{ZOLL} Zoll case: Proof of Theorem \ref{Zoll}}

We normalize the
metric so that the geodesic flow is periodic of period $2 \pi$ and
we assume that $2 \pi$ is the minimal period of periodic orbits.
We then center the intervals $I_k$ at the points $k +
\frac{\beta}{4}$ where $\beta$ is the common Morse index of the
$2\pi$-periodic geodesics.

%\begin{prop} \label{Zoll} Suppose that $(M, g)$ is a real analytic Zoll metric. Then,
%if $\sqrt{\rho}(\zeta) \geq \frac{C}{k}$ we have:
%$$\begin{array}{l} P^{\tau}_{I_k} (\zeta, \bar{\zeta})  =
%
%\left(\frac{(k + \frac{\beta}{4})}{\sqrt{\rho}(\zeta)} \right)^{
%\frac{m - 1}{2}} \left(1 + O(k^{-1}) \right)   \end{array}
%  $$
%  On the other hand, if $\sqrt{\rho}(\zeta) \leq \frac{C}{k}$ we have:
%$$\begin{array}{l} P^{\tau}_{I_k} (\zeta, \bar{\zeta})  =
 % ;
%\left((k + \frac{\beta}{4}) \right)^{m - 1} \left(1 + O(k^{-1})
%\right).   \end{array}
 % $$
%
%\end{prop}

 \begin{proof} The proof is similar to the real off-diagonal
 asymptotics in \cite{ZZoll}, and we only sketch it here. The key point is that
 $\Pi_{I_k}$ are the Fourier coefficients of the $2 \pi$  periodic unitary
 group $U(t) = e^{i t (A + \frac{\beta}{4})}$ in the sense that
 $$\Pi_{I_k} = \frac{1}{2 \pi} \int_0^{2 \pi} e^{- i (k +
 \frac{\beta}{4}) t} U(t) dt. $$ 
Here, $A = k$ in the kth cluster of eigenvalues.
It follows that
\begin{equation} \label{ZOLLP} P_{I_k}^{\tau}(\zeta, \bar{\zeta})
 = \frac{1}{2 \pi} \int_0^{2 \pi}  e^{- i (k +
 \frac{\beta}{4}) t} \Pi_{\tau} \hat{\sigma}_{\tau t} g_{\tau}^t \Pi_{\tau}(\zeta, \overline{\zeta}) dt. 
\end{equation}

We then proceed through the steps of  Theorem \ref{SCALINGTHEO}  but using \eqref{ZOLLP}
instead of the oscillatory integral in Lemma \ref{SMOOTHCORa}.  The calculations are of the same type with
$T(\zeta) = 2 \pi$ and with
$D g_{\zeta}^T = I$ for all $\zeta$ in the Zoll case. 

The principal new feature is that one does not need to use a Tauberian theorem to obtain the asymptotics for $P_{I_k}^{\tau}(\zeta, \bar{\zeta})$, but only to use the fact that
$$\sum_{k = 1}^{\infty} e^{i (k + \frac{\beta}{4}) t} P_{I_k}^{\tau}(\zeta, \bar{\zeta})$$
is a Fourier series with only positive terms. One can then obtain complete asymptotic expansions
of the Fourier coefficients by matching terms of Hardy series.  We refer to Proposition 13.10 of  \cite{BoGu} for the details.

The result is a complete asymptotic expansion of the type stated in Theorem \ref{Zoll}.

 \end{proof}

\begin{rem} To obtain an `integrated' expansion on $[0, \lambda]$ we would form the 
sums $\sum_{k=1}^N  P_{I_k}^{\tau}(\zeta, \bar{\zeta})$ and substitute the asymptotic expansion for each term.
The rather complicated inequalities of  Theorem \ref{SHORTINTSa} (3) are only necessary for choices of $\lambda$
which do not contain the full cluster of eigenvalues below the endpoint $\lambda$.
\end{rem}

\section{\label{GBSECT} Extremals: Gaussian beams associated to non-degenerate elliptic closed geodesics}
Since Gaussian beams along elliptic closed geodesics are the extremals for sup-norms in Grauert tubes, we provide some 
background on   the construction of Gaussian beams associated to an elliptic closed geodesic $\gamma$. We follow
the presentation in \cite{Z97b} and \cite[Section 9]{BB91} (see also \cite{Ral82}).

 We denote by  $\jcal_{\gamma}^{\bot}
\otimes \C$  the space of complex normal Jacobi fields along $\gamma$, a symplectic
vector space of (complex) dimension 2n (n=dim M-1) with respect to the Wronskian
$$\omega(X,Y) = g(X, \frac{D}{ds}Y) - g(\frac{D}{ds}X, Y).$$
 The linear Poincare map $P_{\gamma}$ is defined to be  the linear symplectic map on $\jcal_{\gamma}
^{\bot} \otimes \C$ defined by $P_{\gamma} Y(t) = Y(t + L_{\gamma}).$  The closed geodesic 
$\gamma$ is called non-degenerate elliptic if the eigenvalues of $P_{\gamma}$ are
of the form $\{ e^{\pm i \alpha_j}, j=1,...,n\}$ where the exponents  $\{\alpha_1,
...,\alpha_n\}$, together with $\pi$,  are independent over ${\mathbb Q}$. The associated normalized eigenvectors
will be denoted $\{ Y_j, \overline{Y_j}, j=1,...,n \}$,
\begin{equation} \label{YjDEF} P{\gamma} Y_j = e^{i \alpha_j}Y_j \;\;\;\;\;\;P_{\gamma}\overline{Y}_j =
e^{-i\alpha_j} \overline{Y}_j \;\;\;\; \omega(Y_j, \overline{Y}_k) = \delta_{jk}.\end{equation}

Let $(s, y)$ denote Fermi normal coordinates in a tubular neighborhood of $\gamma$. Let $L $ denote the length of $\gamma$. Roughly speaking,
 Gaussian beams $\Phi_{kq}(s, y)$ along $\gamma$ are oscillatory sums with positive complex phases. They have a real oscillatory
 factor $e^{i k s}$ and a transverse Hermite factor $D_q(y)$, which is the qth Hermite function in the normal direction to $\gamma$, with
 $q \in {\mathbb N}^n$. The special case $q = 0$ is the ground state Gaussian beam, and the higher $q$ are Gaussians times higher
 Hermite polynomials. In  general, Gaussian beams are only approximate eigenfunctions (quasi-modes) but in 
special cases such as surfaces of revolution (and many other $(M,g)$ with completely integrable geodesic flow),
they are exact eigenfunctions.  Given $(k,q)$ the effective `Planck constant' for the sequence with fixed $q$ and $k \to \infty$ is,    $$r_{kq} := \frac{1}{L} (2 \pi k + \sum_{j=1}^n (q_j + \frac{1}{2}) \alpha_j).$$ The associated sequence of eigenvalues of $\sqrt{\Delta}$ has the expansion, 
%The metric
%coefficients $g_{ij}$ will always be taken relative to Fermi
%normal coordinates $(s,y)$ along $\gamma$. The mth jet of $g$ along $\gamma$ will
%be denoted by $j^m_{\gamma}g$, the curvature tensor  by
%$R$ and its covariant derivatives by $\nabla^m R$.  
$$\lambda_{kq} \equiv r_{kq} + \frac{p_1(q)}{r_{kq}} + \frac{p_2(q)}{r_{kq}^2} + ...,$$
where $p_j(q)$ are polynomials whose parity and degrees are described in \cite[Section 9]{BB91}.

We now introduce the precise  Hermite functions in the Gaussian beam.  Relative to a parallel
normal frame $e(s):= (e_1(s),...,e_n(s))$ along $\gamma$ the Jacobi eigenfields have the form, 
$Y_j(s)= \sum_{k=1}^n y_{jk}(s)e_k(s).$  We denote by,
$${\bf Y(s)}: = \begin{pmatrix} y_{jk}(s) \end{pmatrix} $$
the complex $n \times n$ matrix whose $j$th column is the $j$th Jacobi eigenfield.
Also let \begin{equation} \label{GAMMADEF} \Gamma(s) := \frac{d{\bf Y(s)}}{ds} {\bf Y(s)}^{-1}. \end{equation} $\Gamma(s)$ satisfies a matrix Riccati equation, $$\dot{\Gamma} + \Gamma^2 + K = 0, $$
where $K$ is the curvature matrix $R(\dot{\gamma}, Y_j) \dot{\gamma}$,  and  is a complex symmetric $n\times n$ matrix with positive definite imaginary part \cite[Page 229]{BB91}. In fact, by \cite[9.3.11]{BB91},
\begin{equation} \label{JACOBI}  \Im \Gamma(s) = \half ({\bf Y(s)} {\bf Y^*(s)})^{-1},\end{equation}
where as usual ${\bf Y(s)}^*$ is the adjoint of ${\bf Y(s)}$.
% Since we are going to analytically continue, we also need a formula for the real part: 
%$$\Re \Gamma(s) = ?? $$
%\edit{unfinished}
 We will use the equations \cite[(9.2.22)]{BB91}, $$\left\{ \begin{array}{l} {\bf Y(s)}^* {{\bf \dot Y(s)}} - {\bf \dot Y(s)}^* {\bf Y(s)} = i I, \\ \\ 
{\bf Y(s)}^t {\bf \dot Y(s)} - {\bf \dot Y(s)}^t {\bf Y(s)} = 0. \end{array} \right.,  $$
where $Y^t$ is the transpose of $Y$. We multiply the second equation on the left  by $({\bf Y(s)}^t)^{-1}$ and on the right by ${\bf Y(s)}^{-1}$ to get
$${\bf  \dot Y(s)} {\bf Y(s)}^{-1} -({\bf Y(s)}^t)^{-1} {\bf  \dot Y(s)} = 0, \; \rm{or} \; Y^{-1 *} {\bf \dot Y(s)}^*  - {\bf \dot Y(s)}^* ({\bf Y(s)}^t)^{-1*} = 0 $$

The transverse ground state Gaussian is defined in Fermi normal coordinates by,  $$U_0(s,y)  = (\det {\bf Y(s)})^{-1/2} e^{\frac{i}{2} \langle \Gamma(s) y,y\rangle}.$$ 
Although they are of secondary interest here, the higher Hermite functions have the form,
$U_q = \Lambda_1^{q_1}...\Lambda_n^{q_n} U_0$, where $\Lambda_j$ are certain creation operators associated to the Jacobi data \cite[Section 9]{BB91}.

The Gaussian  beams can now be defined by the formal series,
$$\Phi_{kq}(s,\sqrt{r_{kq}}y) =e^{ir_{kq}s} \sum_{j=0}^{\infty}
r_{kq}^{-\frac{j}{2}} U_q^{\frac{j}{2}}(s, \sqrt{r_{kq}}y,r_{kq}^{-1})$$
with $U_q^0 = U_q$ (see [BB92]). The functions $U_q^{\frac{j}{2}}$ are found by solving transport equations. As is usual in the theory of quasi-modes, the infinite series represents a formal asymptotic expansion, and means that if one truncates the series
at $j = N$, then the resulting finite series solves the Laplace equation up to a remainder of order  $r_{kq}^{-\frac{N}{2}}$.  We are mainly interested in the case $q =0$, in which case the Gaussian beam is given by,
$$\Phi_{k0}(s,\sqrt{r_{k0}}y) =e^{ir_{k0}s} \sum_{j=0}^{\infty}
r_{k0}^{-\frac{j}{2}} U_0^{\frac{j}{2}}(s, \sqrt{r_{k0}}y,r_{k0}^{-1}).$$
\begin{rem} The calculations here are much simpler on spheres than for general Gaussian beams in Section \ref{GBSECT}; 
due to the constant curvature on spheres,  the fact that all geodesics are closed, the matrix $\Gamma(s)$
\eqref{GAMMADEF}  is simply $i I$ for a closed geodesic of $\Ss^m$. 

\end{rem}

We say that an eigenfunction is a Gaussian beam when it admits such an asymptotic expansion. The Gaussian beam is exponentially concentrated
in a tubular neighborhood of radius $\frac{1}{\sqrt{r_{kq}}}$. around $\gamma$. Changing variables to $u  = \sqrt{r_{kq}} y$,  one may approximate its $L^2$ norm-square by,
$$\int_0^L \int_{|y| \leq \frac{1}{\sqrt{r_{k0}}}} |U_0(s, \sqrt{r_{k0}} y)|^2 d s dy \simeq C_0  (r_{k0})^{- \frac{(\dim M-1)}{2}}. $$
We are only interested in orders of magnitude and  omit further details. It follows that the $L^2$ normalized Gaussian beam has the form,
$$ C_0\;\; k^{ \frac{(\dim M-1)}{2}}  \Phi_{k0}(s,\sqrt{r_{k0}}y), $$
where $C_0 >0$ is a positive constant. Thus, the sup-norm in the real domain of the Gaussian beam is asymptotically $C_0\;  k^{ \frac{(\dim M-1)}{2}}.$

The linear Poincar\'e map $D_{\zeta} g_{\tau}^L$ in this case is given by \eqref{YjDEF} or in the Grauert tube notation by, 
$$ S_{\zeta} = 
\left( \begin{array}{ll} \Im\dot{Y}(L)^* \;\;\;& \Im Y(L)^*\\
\Re \dot{Y}(L)^*\;\;\;&\Re Y(L)^*  \end{array} \right).$$
By \eqref{YjDEF}, it is diagonalizable over $\C$ as a block-diagonal matrix
$$S_{\zeta} \simeq  \bigoplus_{j=1}^n \begin{pmatrix} e^{i \alpha_j} & 0 \\ &\\ 0 & e^{- i \alpha_j} \end{pmatrix},$$
where $\simeq$ denotes unitary equivalence in $GL(n, \C)$; the right side is of course not in $\rm{Sp}(n,\R)$. The metaplectic quantization
of $S_{\zeta}$ is the exponential of a Harmonic oscillator Hamiltonian $\hat{H}_{\vec \alpha}  $  with frequencies $\alpha_j$, i.e. 
$$W_J(S_{\zeta}) = \exp i \hat{H}_{\vec \alpha}, \;\; \hat{H}_{\vec \alpha} = \sum_{j=1}^n D_{x_j}^2 + \alpha_j x_j^2, $$
with eigenvalues $\lambda_{\vec k} = \sum_{j=1}^n \alpha_j (k_j + \half), \vec k \in {\mathbb N}^n$. 
Thus, $W_J(S_{\zeta})$ has eigenvalues $\exp (i\sum_{j=1}^n \alpha_j (k_j + \half)).$ In this model, the ground state $\Omega_{\zeta}$  is the standard Gaussian,
which is the eigenfunction of  eigenvalue $\lambda_{\vec 0}$. Hence,
$W_J(S_{\zeta}) \Omega_{\zeta} = e^{i\half |\vec \alpha|} \Omega_{\zeta}$ with $\vec \alpha  = \sum_j \alpha_j$, and therefore,
$\langle W_J(S^{\ell}_{\zeta}) \Omega_{\zeta}, \Omega_{\zeta} \rangle = e^{i\half \ell |\vec \alpha|} $. By \eqref{INFSERIES},
$$Q_{\zeta} (\lambda) = \frac{1}{2i}  \sum_{n\not=0}  \frac{e^{i n (\lambda  L + \half i  |\vec \alpha|)}}{n} = \{\lambda L + \half |\vec \alpha| - \pi\}_{2 \pi}. $$

It is not straightforward to calculate the $L^2$ norm and sup norm of the analytic continuation of the Gaussian beam to $\partial M_{\tau}$. The analytic continuation is given in analytic Fermi normal coordinates $(s + i \sigma, y + i \eta)$ by,
$$\Phi^{\C}_{k0}(s + i \sigma,\sqrt{r_{k0}}(y + i \eta) ) = k^{ \frac{(\dim M-1)}{2}}  e^{ir_{k0}(s + i \sigma)} \sum_{j=0}^{\infty}
r_{k0}^{-\frac{j}{2}} U_0^{\frac{j}{2}}(s + i \sigma, \sqrt{r_{k0}}(y + i \eta),r_{k0}^{-1}),$$
with leading order term,  $$U^{\C}_0(s + i \sigma,  \sqrt{r_{k0}}(y + i \eta))  = (\det {\bf Y(s + i \sigma)})^{-1/2} e^{ir_{k0}(s + i \sigma)}  e^{\frac{i}{2}  r_{k0}  \langle \Gamma(s + i \sigma)  (y + i \eta) ,(y + i \eta) \rangle}.$$
Here, $\sqrt{\rho}(s + i \sigma, y + i \eta) = \tau$. Upon analytic continuation, it is not clear that the  Gaussian beam need should be  concentrated in the complexification
of the real tube around $\gamma$, i.e. in a phase space tube around the phase space orbit $\gamma$, since the damping Hermite factor in the real
domain can grow
exponentially outside the tube once it is analytically continued. When $\sigma = - \tau, y = \eta = 0$ it is evident that it attains the maximal
value $k^{ \frac{(\dim M-1)}{2}}  e^{k \tau}$.  By Lemma \ref{L2LEMintro}, the $L^2$ norm is asymptotically  $k^{- \frac{(\dim M-1)}{2}}  e^{k \tau}$,
so the sup norm of the Husimi distribution is of order $k^{\frac{(\dim M-1)}{2}}.$

% \edit{Incomplete. 
%Write $\Gamma(s) = \Gamma_{\R}(s) + i \Gamma_{\Im}(s)$ for the real/imaginary parts. Then, the real part of the %exponent is,
%$$ \begin{array}{lll} \Re \frac{i}{2}  \langle \Gamma(s)  (y + i \eta) ,(y + i \eta) \rangle & = & 
%  - \half \langle  \Gamma_{\Im}(s) y, y \rangle  + \half \langle  \Gamma_{\Im}(s) \eta, \eta \rangle)   
%-  \langle \Gamma_{\Re}(s)  y, \eta \rangle. \end{array}$$ 

 %The supremum is attained when $y = \eta = 0$ and has order $e^{-  \tau k}$.
 
%$$2 \Re \dot{Y} Y^{-1} := \dot{Y} Y^{-1} +  Y^{*-1} \dot{Y}^*  =  (Y^t)^{-1} \dot{Y}  + Y^{*-1} \dot{Y}^*. $$
%The real part of the Riccati equation is,
%$$\dot{\Gamma}_{\R} + \Gamma_{\R}^2 - \Gamma_{\Im}^2 + K = 0. $$}

%\subsection{\label{HUSIMITORUS}  Sup norms of Husimi distributions on flat tori}
%In this section we continue with the computations in Section \ref{TORUS}, and  compute the sup norms of Husimi distributions of plane wave eigenfunctions of %the flat torus. 
 
%We recall that the complexified plane waves 
%are the complex exponentials $e_k^{\C}(x + i \xi) = e^{i \langle x
%+ i \xi, k \rangle}, $ and clearly $$| e^{i \langle x
%+ i \xi, k \rangle}|^2  = e^{- 2 \langle \xi, k \rangle}. $$

\subsection{Geometric interpretation} We briefly explain the symplectic geometry underlying the extremals for
sup norms in both the real and complex domain.

In the real domain, the extremal eigenfunctions for sup-norms are zonal spherical harmonics of each degree $N$ (i.e eigenfunctions invariant under rotations
 around the third axis).  The proof is that all other eigenfunctions vanish at the fixed points (poles) of these rotations, hence the universal pointwise
 Weyl asymptotics \eqref{LWL} can only hold at a pole  $x$  if the zonal harmonics attain the maximal sup norm bound at $x$.  The symplectic geometry underlying this sup norm behavior is that zonal spherics harmonics $Y_N^0$  (indeed, the
 entire basis of joint eigenfunctions $Y_N^m$ of $\Delta$ and of rotation around the third axis) are semi-classcial
 Lagrangian distributions associated to the meridian Lagrangian $\Lambda_0 \subset S^*\Ss^2$ of unit co-vectors
 tangent to the family of meridian geodesics between the poles. The extremal sup norm is attained at the poles
 and reflects the `blow-down' singularity of the Lagrange projection  $\pi: \Lambda_0 \to \Ss^2$ over the poles .

 The vanishing of modes $Y_N^m$ with $m \not=0$ at the poles has  no analogue
for Husimi distributions in the Grauert tube (i.e. phase space) setting,  because  there are no fixed points in $\partial M_{\tau}$ for the lift of the rotation group. 
 Indeed, analytic continuations of zonal spherical harmonics do not attain maximal sup-norm growth in the complex domain. Rather,  the extremals are 
Gaussian beams (highest weight spherical harmonics), which are extremals for low $L^p$ norms, but not high $L^p$ norms,  in the real domain. As mentioned
above, coherent states (Definition \ref{CSTDEF}) also attain the maximum, but are not  Husimi distributions of eigenfunctions.

From the symplectic geometric viewpoint, the explanation requires background in theory of Toeplitz operators
and their associated symplectic cones in \cite{BoGu}. Briefly, the Hardy  space $H^2(\partial M_{\tau})$ of boundary
values of holomorphic functions in $M_{\tau}$ is a Hilbert space associated to the symplectic cone
$\Sigma_{\tau} \subset T^* \partial M_{\tau}$ spanned by the action form $\alpha_{\tau}$. That is,
$\pi: \Sigma_{\tau} \to \partial M_{\tau}$ is an $\R_+$ bundle whose fiber over $\zeta$ consists of $\R_+\alpha_{\zeta}$. 
As reviewed in Section \ref{GRAUERTSECT}, the metric $g$ induces an identification of $\Sigma_{\tau} \simeq
S^*_{\tau} M$ (covectors of length $\tau$). Hence, the Lagrangian submanifold $\Lambda_0 \subset S^*_{\tau} M$
may be identified with a Lagrangian submanifold of $\partial M_{\tau}$ and of $\Sigma_{\tau}$.  Obviously,
the natural projection $\pi: \Lambda_0 \to \partial M_{\tau}$ is an embedding rather than a Lagrangian projection.
Consequently, there is no `singularity' to cause sup norm blowup. On the other hand, the symplectic geometry
associated to  the highest weight spherical harmonics $Y_N^N$ (or any  Gaussian beam) is the closed geodesic along which it concentrates, lifted by its unit tangent vectors to $\partial M_{\tau}$. This geodesic is a singular leaf of the
foliation of $\partial M_{\tau}$ by orbits of the Hamiltonian torus $\R^2/\Z^2$ action generated by the geodesic
flow together with rotations. This singularity does cause extremal behavior in the associated modes.

\section{\label{APPENDIX} Appendix}

\subsection{Integrated Weyl laws in the real domain}

    The geodesic flow
$G^t$ of $(M, g)$ of a real analytic Riemannian manifold  is of
one of the following two types:

\begin{enumerate}

\item  {\it aperiodic:} The Liouville measure of the closed
 orbits of $G^t$, i.e. the set of vectors lying on closed geodesics,  is zero; or

\item  {\it periodic = Zoll:} $G^T = id$ for some  $T>0$;
henceforth $T$ denotes the minimal period.  The common Morse index
of the $T$-periodic geodesics will be denoted by $\beta$.

\end{enumerate}

In the real domain, the  two-term Weyl laws counting eigenvalues
of $\sqrt{\Delta}$ are very different in these two cases. 

\begin{enumerate}

\item Let $I_{\lambda} = [0, \lambda]$ and let $N(\lambda) = \int_M \Pi_{I_{\lambda}}(x,x) dV(x)$.   In the {\it aperiodic} case, the
Duistermaat-Guillemin-Ivrii  two term Weyl law states
$$N(\lambda ) = \#\{j:\lambda _j\leq \lambda \}=c_m \;
Vol(M, g) \; \lambda^m +o(\lambda ^{m-1})$$
 where $m=\dim M$ and where $c_m$ is a universal constant.

\item  In the {\it periodic} case,
 the spectrum of $\sqrt{\Delta}$ is a union of eigenvalue clusters $C_N$ of the form
\begin{equation} \label{CLUSTER} C_N=\{(\frac{2\pi}{T})(N+\frac{\beta}{4}) +
 \mu_{Ni}, \; i=1\dots d_N\} \end{equation}
with $\mu_{Ni} = 0(N^{-1})$.   The number $d_N$ of eigenvalues in
$C_N$ is a polynomial of degree $m-1$.
\end{enumerate}
In the aperiodic case, we can choose the center of the spectral interval $I_{\lambda}$ arbitrarily. In the Zoll
case we center it along the arithmetic progression $\{(\frac{2\pi}{T})(N+\frac{\beta}{4}) \}$.
We refer to \cite{Ho,SV,ZZoll} for background and further
discussion.

\subsection{  \label{TAUB} Tauberian Theorems}

We record here the statements of the Tauberian theorems that we
use in the article. Our main reference is \cite{SV}, Appendix B
and we follow their notation.

We denote by $F_+$ the class of real-valued, monotone
nondecreasing functions $N(\lambda)$  of polynomial growth
supported on $\R_+$. The following  Tauberian theorem uses only
the singularity at $t = 0$ of $\widehat{dN}$ to obtain a one term
asymptotic of $N(\lambda)$ as $\lambda \to \infty$:
\begin{theo} \label{ET} Let $N \in F_+$ and let $\psi \in
\scal (\R)$ satisfy the conditions:  $\psi$ is  even,
$\psi(\lambda)
> 0$ for all $\lambda \in \R$,   $\hat{\psi} \in
C_0^{\infty}$, and $\hat{\psi}(0) = 1$. Then,
$$\psi * dN(\lambda) \leq A \lambda^{\nu} \implies |N(\lambda) - N *
\psi(\lambda)| \leq C A \lambda^{\nu}, $$ where $C$ is independent
of $A, \lambda$.
\end{theo}

To obtain a two-term asymptotic formula, one needs to take into
account the other singularities of $\widehat{dN}$. We let $\psi$
be as above, and also introduce a second test function $\gamma \in
\scal$ with $\hat{\gamma} \in C_0^{\infty}$ and with the supp
$\hat{\gamma} \subset (0, \infty)$.

\begin{theo} \label{TTT} Let $N_1, N_2 \in F_+$ and assume:
\begin{enumerate}

\item $N_j * \psi (\lambda) = O(\lambda^{\nu}), (j = 1,2);$

\item $ N_2 * \psi (\lambda) = N_1 * \psi (\lambda) +
o(\lambda^{\nu}); $

\item $\gamma * d N_2 (\lambda) = \gamma * d N_1(\lambda) +
o(\lambda^{\nu})$.

\end{enumerate}

Then,

$$N_1(\lambda - o(1)) - o(\lambda^{\nu}) \leq N_2(\lambda) \leq
N_1(\lambda + o(1)) + o (\lambda^{\nu}). $$

\end{theo}

This Tauberian theorem is useful when the non-zero singularities
of $\widehat{dN_2}$ are as strong as the singularity at $t = 0$
and $N_2$ does not have two term polynomial asymptotics.

\subsection{\label{NOTAPP} Notational Appendix}
In this section, we list the main notations. \bigskip

\subsubsection{Notation for Husimi distributions and Weyl sums} 

\begin{enumerate} \item Husimi distributions: \eqref{HUSIMI} 
$$\frac{|\phi_{\lambda}^{\C}(\zeta)|^2}{|| \phi_{\lambda}||^2_{L^2(\partial M_{\tau})}}.$$

\item \eqref{CXSPMa} Analytic continuations of spectral projections with eigenvalues in the interval $I_{\lambda}$:
 $$  \Pi_{I_{\lambda}}^{\C}(\zeta, \bar{\zeta}):=
\sum_{j: \lambda_j \in  I_{\lambda}}
|\phi_{j}^{\C}(\zeta)|^2. $$ $I_{\lambda} $ could be a short interval  $[\lambda, \lambda
+ 1]$ of frequencies or a long window $[0, \lambda]$.

\item  \eqref{TCXSPM} 
`Tempered'  spectral
projections
$$P_{ I_{\lambda}}^{\tau}(\zeta, \bar{\zeta}) =
\sum_{j: \lambda_j \in  I_{\lambda}} e^{-2 \tau \lambda_j}
|\phi_{j}^{\C}(\zeta)|^2
\tau).
$$

\item Renormalized spectral projections:$$
\wt{P} _{[0, \lambda}^{\tau}(\zeta, \bar{\zeta}) =
\sum_{j: \lambda_j \leq \lambda }
\frac{|\phi_{j}^{\C}(\zeta)|^2}{ ||\phi_j^{\C}||^2_{L^2(\partial M_{\tau})}} , \;\; (\sqrt{\rho}(\zeta) =
\tau),$$
adapted to the  Husimi distributions \eqref{HUSIMI}.

\item Dual Poisson-wave group:
\eqref{CXWVGP}
$$ \begin{array}{lll} U_{\C} (t + 2 i \tau, \zeta, \bar{\zeta}) &= & \sum_j
e^{(- 2 \tau + i t) \lambda_j} |\phi_j^{\C}(\zeta)|^2.
%\\ && \\ & =
%& \int_{\R} e^{i t \lambda} d_{\lambda} P_{[0, \lambda]}^{\tau}
%(\zeta, \bar{\zeta}).
 \end{array} $$ 
 
 \item Poisson kernel \eqref{PTKER}:
$$P^{\tau}(\zeta, y) = \sum_j e^{- \tau \lambda_j} \phi_j^{\C}(\zeta) \phi_j(y). $$

\end{enumerate}

\subsubsection{Notation for Grauert tubes, CR geometry,  geodesic flow} 
\begin{enumerate}

\item $J_{\zeta}$: complex structure at $T_{\zeta} M_{\tau}$. \bigskip

\item Complexified CR subspace and type decomposition: 
$$H^{1,0}_{\zeta} \oplus H^{0,1}_{\zeta}. $$

 \item  Geodesic flow in the Grauert tube setting \eqref{gtdef}:
$$g^t = \exp t \Xi_{\sqrt{\rho}}, \;\; G^t = \exp t \Xi_{\rho}. $$
Its restriction to $\partial M_{\tau}$ 
\eqref{gtau} $$g_{\tau}^t: \partial M_{\tau} \to \partial M_{\tau}.$$

\item $\pcal$: periodic orbits of $g^t_{\tau}$. $T(\zeta)$ \eqref{Tzeta}: period of the periodic point $\zeta \in \pcal$.

\item Linearization (Poincar\'e map) of $g^t_{\tau}$ at $\zeta \in \pcal$
\eqref{DGn} $$ D g^{n T(\zeta)} _{\zeta}:  H^{1,0}_{\zeta} \oplus H^{0,1}_{\zeta} \to H^{1,0}_{\zeta} \oplus H^{0,1}_{\zeta}. $$

\item  {\it osculating Bargmann-Fock space} (Definition \ref{OSCBFDEF}): $\hcal^2_{\zeta}$ at $\zeta \in \partial M_{\tau}$. 

\item Vacuum (ground) state in $\hcal^2_{\zeta}: \Omega_{J_{\zeta}}.$

\item Metaplectic representation $W_{J_{\zeta}} \;(D g^{n T(\zeta)}_{\zeta}) $  of  $(D g^{n T(\zeta)}_{\zeta})$. 

\item \eqref{ABCDintro}  $$\gcal_n(\zeta): =  \langle W_{J_{\zeta}} \;(D g^{n T(\zeta)}_{\zeta}) \ \Omega_{J_{\zeta}} \rangle. $$ \end{enumerate}

\subsection{Notation for linear symplectic algebra} 

\begin{enumerate}

\item Projection to $H^{1,0}_J$ \eqref{PJ} : $ P_J = \half(I - i J): V \otimes \C \to  H^{1,0}_J, \;\; \bar{P}_J = \half(I + i J): V \otimes \C \to H^{0,1}_J.$

\item Holomorphic block of a linear symplectic map \eqref{PDEF}:   $ P_J S P_J =  P = \half(A+D + i (C-B)).$

\item  Ground state of Bargmann-Fock space \eqref{GSJ}: $\Omega_{J} (v) := e^{-\half \sigma(v, J v)}$.

\item 
 The Bargmann-Fock space of a symplectic
vector space $(V, \sigma)$ with compatible complex structure $J \in \jcal$ (Section \ref{HEISMETSECT}): 
$$\hcal_{J} = \{f  e^{-\half \sigma(v, J v)} \in L^2(V, dL), f\; \mbox{is\; entire \; J-\;holomorphic}\}. $$

\end{enumerate}

%By \eqref{PJ}, the upper left block, \begin{equation} \label{PDEF} P_J S P_J =  P = \half(A+D + i (C-B)), \end{equation}
%arises  in \eqref{ABCD}-\eqref{ABCDintro}  and in the leading order symbol in Proposition \ref{LINEAR}. \footnote{The notation $ \hat{P}_S$ for a positive matrix should not be %confused with the block $P= P_J S P_J$.  }
 %To put this block into context, let  $V$ be a real symplectic vector space and let 
  %$V_{\C} = V \otimes \C$ be its  complexification. Decompose $V_{\C} = V^{1,0} \oplus V^{0,1}$ into eigenspaces of the standard complex structure $J$. 
 %Denote  the projection to the `holomorphic component' by
%\begin{equation} \label{pi10} \pi^{1,0} : V \otimes \C \to V^{1,0}. \end{equation}

\end{document}